\documentclass{siamltex}
\usepackage{amsfonts,amsmath,amssymb}
\usepackage{mathrsfs,mathtools}
\usepackage{enumerate}
\usepackage{xspace,MACROS/mydef}
\usepackage{hyperref}
\usepackage{esint}
\usepackage{graphicx}
\usepackage{psfrag}
\DeclareMathAlphabet{\mathpzc}{OT1}{pzc}{m}{it}

\usepackage[usenames,dvipsnames]{color}
\newcommand{\slope}{\mathfrak{s}}
\newtheorem{remark}[theorem]{Remark}
\newcommand{\eremk}{\hbox{}\hfill\rule{0.8ex}{0.8ex}}
\numberwithin{equation}{section}

\DeclareMathOperator{\dist}{dist}
\newcommand{\calA}{{\mathcal A}}
\newcommand{\calC}{{\mathcal C}}
\newcommand{\calG}{{\mathcal G}}
\newcommand{\calL}{{\mathcal L}}

\newcommand{\calM}{{\mathcal M}}
\newcommand{\calN}{{\mathcal N}}
\newcommand{\calO}{{\mathcal O}}
\newcommand{\calT}{{\mathcal T}}

\newcommand{\bA}{{\boldsymbol A}}
\newcommand{\bmr}{{\boldsymbol r}}
\newcommand{\bmc}{{\boldsymbol c}}
\newcommand{\bbN}{{\mathbb N}}
\newcommand{\bbR}{{\mathbb R}}
\title{
Tensor FEM for spectral fractional diffusion\thanks{The results in this paper were obtained when the authors met at the MFO Oberwolfach 
        during the WS 1711 in March 2017. 
The research of RHN was supported in part by NSF grants DMS-1109325 and DMS-1411808. 
The research of EO was supported in part by CONICYT through project FONDECYT 3160201. 
The research of AJS was supported in part by NSF grant DMS-1418784.
The research of JMM was supported by the Austrian Science Fund (FWF) project F 65. 
}}
\author{Lehel Banjai\thanks{Maxwell Institute for Mathematical Sciences, School of Mathematical  \& Computer Sciences, Heriot-Watt University, Edinburgh EH14 4AS, UK  (\texttt{l.banjai@hw.ac.uk}).}
\and 
Jens M. Melenk\thanks{Institut  f\"{u}r  Analysis  und  Scientific  Computing,  Technische  Universit\"{a}t  Wien,  A-1040  Vienna, Austria (\texttt{melenk@tuwien.ac.at}).}
\and
Ricardo H.~Nochetto\thanks{Department of Mathematics and Institute for Physical
Science and Technology, University of Maryland, College Park, MD 20742, USA (\texttt{rhn@math.umd.edu}).}
\and
Enrique Ot\'arola\thanks{Departamento de Matem\'atica, Universidad T\'ecnica Federico Santa Mar\'ia, Valpara\'iso, Chile (\texttt{enrique.otarola@usm.cl}).}
\and
Abner J.~Salgado\thanks{Department of Mathematics, University of Tennessee, Knoxville, TN 37996, USA (\texttt{asalgad1@utk.edu}).}
\and
Christoph Schwab\thanks{Seminar for Applied Mathematics, ETH Z\"{u}rich, ETH Zentrum, HG  G57.1, CH8092 Z\"{u}rich, Switzerland (\texttt{christoph.schwab@sam.math.ethz.ch}).}}

\date{Draft version of \today.}

\begin{document}

\maketitle
\begin{abstract}
We design and analyze several Finite Element Methods (FEMs) 
applied to the Caffarelli-Silvestre extension  
that localizes the fractional powers of symmetric, coercive, 
linear elliptic operators in bounded domains with Dirichlet boundary conditions. 
We consider open, bounded, polytopal but not necessarily convex 
domains $\Omega \subset \bbR^d$ with $d=1,2$.
For the solution to the extension problem, we establish 
analytic regularity with respect to the extended variable $y\in (0,\infty)$.
We prove that the solution belongs to countably normed, 
power--exponentially weighted 
Bochner spaces of analytic functions with respect to $y$,
taking values in corner-weighted Kondat'ev type Sobolev spaces in $\Omega$.
In $\Omega\subset \bbR^2$,
we discretize with continuous, piecewise linear, Lagrangian FEM ($P_1$-FEM) 
with mesh refinement near corners, and prove that 
first order convergence rate is attained for compatible data $f\in \Ws$.

We also prove that tensorization of a $P_1$-FEM in $\Omega$ 
with a suitable $hp$-FEM in the extended variable achieves 
log-linear complexity with respect to $\calN_\Omega$, the number 
of degrees of freedom in the domain $\Omega$.
In addition, we propose a novel, \emph{sparse tensor product FEM} 
based on a multilevel $P_1$-FEM in $\Omega$ and 
on a $P_1$-FEM on radical--geometric meshes in the extended variable.
We prove that this approach also achieves 
log-linear complexity with respect to $\calN_\Omega$.
Finally, under the stronger assumption 
that the data is analytic in $\overline{\Omega}$, 
and without compatibility at $\partial \Omega$,
we establish \emph{exponential rates of convergence of $hp$-FEM
for spectral, fractional diffusion operators} in energy norm.
This is achieved by a combined tensor product $hp$-FEM 
for  the Caffarelli-Silvestre extension 
in the truncated cylinder $\Omega \times (0,\Y)$
with anisotropic geometric meshes that are refined towards $\partial\Omega$.
We also report numerical experiments for model problems
which confirm the theoretical results.
We indicate several extensions and generalizations of the
proposed methods to other problem classes and 
to other boundary conditions on $\partial \Omega$.
\end{abstract}

\begin{keywords}
Fractional diffusion, nonlocal operators, weighted Sobolev spaces, regularity estimates, 
finite elements, anisotropic $hp$--refinement, corner refinement, sparse grids, exponential convergence.
\end{keywords}

\begin{AMS}
26A33,   
65N12,   
65N30.   
\end{AMS}

\section{Introduction}
\label{sec:introduccion}
We are interested in the design and analysis of a variety of efficient numerical techniques 
to solve problems involving certain fractional powers of 
the linear, elliptic, self-adjoint, second order, differential operator 
$
\calL w = - \DIV( A \GRAD w ) + c w
$,
supplemented with homogeneous Dirichlet boundary conditions. 
The coefficient $A\in L^\infty(\Omega,\GL(\R^d))$ 
is symmetric and uniformly positive definite and 
$0 \leq c \in L^\infty(\Omega,\R)$
(additional regularity requirements will be imposed in the course
 of our convergence rate analysis ahead).
We denote by $\Omega$ a bounded domain of $\R^d$ ($d=1,2$), 
with Lipschitz boundary $\partial\Omega$ and 
further properties imposed as required:
the FEM convergence theory in Section~\ref{S:FEMy} will focus 
on polygonal domains $\Omega \subset {\mathbb R}^2$, 
the $hp$-FEM results in Section~\ref{sec:hpx} require 
analytic $\partial\Omega$.

The Dirichlet problem for the fractional Laplacian is as follows: 
Given a function $f$ and $s \in (0,1)$, we seek $u$ such that
\begin{equation}
\label{fl=f_bdddom}
    \calL^s u = f \quad \text{in } \Omega\;. 
\end{equation}

An essential difficulty in the analysis of \eqref{fl=f_bdddom} 
and in the design of efficient numerical methods for this problem
is that 
 $\calL^s$ is a nonlocal operator
\cite{CS:11,CT:10,CS:07,CDDS:11,Landkof}. 
In the case of the Dirichlet Laplacian $\calL=-\Delta$, 
Caffarelli and Silvestre in \cite{CS:07} localize 
it
by using a nonuniformly elliptic PDE posed in one more spatial dimension.
They showed that any power $s \in (0,1)$ of
the fractional Laplacian in $\R^d$ can be realized as the Dirichlet-to-Neumann map 
of an extension to the upper half-space $\R^{d+1}_+$. 
This result was extended by Cabr\'e and Tan \cite{CT:10} and 
by Stinga and Torrea \cite{ST:10} 
to bounded domains $\Omega$ and more general operators, 
thereby obtaining an extension posed on the
semi--infinite cylinder $\C := \Omega \times (0,\infty)$; 
we also refer to \cite{CDDS:11}.
This extension is the following local boundary value problem
\begin{equation}
\label{alpha_harm_intro}
  \begin{dcases}
    \mathfrak{L} \ue = -\DIV\left( y^\alpha \bA \nabla \ue \right) + c y^\alpha \ue = 0  & \text{ in } \C, 
    \\
    \ue = 0  &\text{ on } \partial_L \C, \\
    \partial_{\nu^\alpha} \ue = d_s f  &\text{ on } \Omega \times \{0\},
  \end{dcases}
\end{equation}
where $\bA = \diag\{A,1\} \in L^\infty(\bar\C,\GL(\bbR^{d+1}))$,
$\partial_L \C := \partial \Omega \times (0,\infty)$ 
signifies the lateral boundary of $\C$, 
$d_s: = 2^{1-2s} \Gamma(1-s)/\Gamma(s)$ is a positive normalization constant 
and the parameter $\alpha$ is defined as $\alpha = 1-2s \in (-1,1)$ \cite{CS:07,ST:10}. 
The so--called conormal exterior derivative of $\ue$ at $\Omega \times \{ 0 \}$ is
\begin{equation}
\label{def:lf}
\partial_{\nu^\alpha} \ue = -\lim_{y \rightarrow 0^+} y^\alpha \ue_y.
\end{equation}
We shall refer to $y$ as the \emph{extended variable} and to 
the dimension $d+1$ in $\R_+^{d+1}$ the \emph{extended dimension} of 
problem \eqref{alpha_harm_intro}. Throughout the text, points $x \in \C$ 
will be written as $x = (x',y)$ with $x' \in \Omega$ and $y > 0$.
The limit in \eqref{def:lf} must be understood in the distributional sense \cite{CT:10,CS:07,ST:10}. 
With the extension $\ue$ at hand, 
the fractional powers of $\calL$ in \eqref{fl=f_bdddom} and 
the Dirichlet-to-Neumann operator of problem \eqref{alpha_harm_intro} 
are related by
\begin{equation} \label{eq:identity}
  d_s \calL^s u = \partial_{\nu^\alpha} \ue \quad \text{in } \Omega\;.
\end{equation}

In \cite{NOS} the extension problem \eqref{alpha_harm_intro} was first 
used as a way to obtain a numerical technique to approximate the solution to \eqref{fl=f_bdddom}. 
A piecewise linear finite element method ($P_1$-FEM) was proposed and analyzed.
In this work, we extend the results of \cite{NOS} in several directions:
\begin{enumerate}[a)]
\item 
In Theorem~\ref{thm:P1Graded}, 
we generalize the error analysis of \cite{NOS}, 
based on the localization of $\calL^s$ given by \eqref{alpha_harm_intro}, 
to nonconvex polygonal domains $\Omega \subset {\mathbb R}^2$, 
under the requirement of Lipschitz regularity in $\Omega$ for $A$ and $c$,
and for $f\in \Ws$ in \eqref{def:Hs} ahead.
  
\item 
In Theorem~\ref{thm:glob_reg_ue} we prove, 
again under Lipschitz regularity in $\Omega$ for $A$ and $c$,
\emph{weighted $H^2$ (with respect to the extended variable $y$) regularity estimates} for the solution 
$\ue$ of 
\eqref{alpha_harm_intro}.  
We use these to propose a novel, sparse tensor product $P_1$-FEM in $\C$
which is realized by invoking (in parallel) $\calO(\log \calN_\Omega)$
many instances of anisotropic tensor product $P_1$-FEM in $\C$.
We prove, in Theorem~\ref{thm:sparse}, that, 
when the base of the cylinder $\C$ is a 
polygonal domain $\Omega \subset {\mathbb R}^2$, 
this approach yields a method with 
$\calO(\calN_\Omega\log \calN_\Omega)$ degrees of freedom 
realizing the (optimal) asymptotic convergence rate of 
$\calN_\Omega^{-1/2}$.
  
\item 
We show, in Theorem~\ref{thm:hpy-gradedx}, 
that a full tensor product approach of an $hp$-FEM in the extended variable $y$ 
with $P_1$-FEM in $\Omega$ yields the same rate.
To achieve this,
we establish \emph{weighted analytic regularity of $\ue$
with respect to the extended variable $y$}, in terms of countably normed
weighted Bochner-Sobolev spaces.
This extends, in the case $d=2$, recent
work \cite{MPSV17} to a general diffusion operator $\calL$ in \eqref{fl=f_bdddom} 
and to nonconvex, polygonal domains, 
under the requirement of Lipschitz regularity in $\Omega$ for $A$ and $c$.

\item We propose in Section \ref{S:diagonalization-abstract-setting} 
a novel \emph{diagonalization technique} which decouples the 
degrees of freedom introduced by a Galerkin (semi-)discretization in the extended variable.
It reduces the $y$-semidiscrete Caffarelli-Stinga extension 
to the solution of independent, singularly perturbed
\emph{second order} reaction-diffusion equations in $\Omega$.
This decoupling allows us to establish exponential convergence 
for \emph{analytic data} $f$ \emph{without boundary compatibility} 
as discussed in the following item \ref{item:e}).
The diagonalization also permits to block-diagonalize the stiffness matrix 
of the fully discrete problem 
with corresponding
befits for the solver complexity of the linear system of equations.

\item 
\label{item:e}
We establish an \emph{exponential convergence} rate \eqref{eq:ExpN1/4} of 
a local $hp$-FEM 
for the fractional differential operator $\mathfrak L$ in \eqref{alpha_harm_intro}.
This requires, however, the data  $A$, $c$ and $f$ to be \emph{analytic} in $\bar{\Omega}$ and 
the boundary $\partial \Omega$ to be analytic as well.
For brevity of exposition, 
we detail the mathematical argument in intervals $\Omega \subset \mathbb{R}^1$ 
and in bounded domains $\Omega \subset \bbR^2$ with analytic boundary 
    $\partial\Omega$, and for constant coefficients $A$ and $c$,
and only outline the necessary extensions, with references, 
for polygons $\Omega \subset {\mathbb R}^2$; 
see Theorems~\ref{thm:hp-for-fractional},
\ref{thm:hp-for-fractional-2D} and Remark~\ref{rem:hp-for-fractional}.
  
\item 
We present numerical experiments in each of the previous cases 
which illustrate our results, and indicate their sharpness.

\item 
We indicate how the presently developed discretizations and error bounds
extend in several directions, in particular to three dimensional polyhedral
domains $\Omega$, to Neumann or mixed Dirichlet-Neumann boundary 
conditions on $\partial \Omega$, etc.
\end{enumerate}

To close the introduction, we comment on other 
numerical approaches to fractional PDEs.
In addition to \cite{NOS}, numerical schemes 
that deal with spectral fractional powers of elliptic operators
have been proposed  in \cite{MPSV17} and \cite{BP:13}. The very recent work \cite{MPSV17} 
adopts the same Galerkin framework as \cite{NOS} and the present article and, 
independently,
proposes to use high order discretizations in the extended variable to exploit analyticity. 
The starting point of \cite{BP:13} is the so-called Balakrishnan formula, 
a contour integral representation of the inverse $\calL^{-s}$.
Upon discretizing the integral by a suitable quadrature formula, 
the numerical scheme of \cite{BP:13} results in a collection of (decoupled) 
singularly perturbed reaction diffusion problems in $\Omega$.
This connects \cite{BP:13} with our approach in Section~\ref{sec:hpx}.
However, the decoupled reaction diffusion problems in $\Omega$
which arise in our approach result from
a Galerkin discretization in the extended variable.
For the integral definition of the fractional Laplacian in several dimensions
we mention, in particular, the analysis of \cite{AcostBorthSINUM2017,DeliaGunzburger}. 
We refer the reader to \cite{ThreeMethods} 
for a detailed account of all the approaches mentioned above.
%
\section{Notation and preliminaries}
\label{sec:notation}
We adopt the notation of \cite{NOS,Otarola}:
For $\Y > 0$ the truncated cylinder with base $\Omega$ and 
height $\Y$ is $\C_\Y = \Omega \times (0,\Y)$,
its lateral boundary is $\partial_L\C_\Y = \partial \Omega \times (0,\Y)$.
If $x\in \C$ we set $x=(x',y)$ with $x'\in \Omega$ and $y \in (0,\infty)$.
By $a \lesssim b$ we mean $a \leq Cb$, 
with a constant $C$ that neither depends 
on $a$, $b$ or the discretization parameters. The notation $a \sim b$ signifies
$a \lesssim b\lesssim a$. The value of $C$ might change at each occurrence. 
\subsection{Fractional powers of elliptic operators}
\label{sub:fracLaplace}
To define $\calL^s$, as in \cite{NOS},
we invoke spectral theory \cite{BS}. 
The operator $\calL$ induces an inner product 
$\blfa{\Omega}(\cdot,\cdot)$ on $H^1_0(\Omega)$
\begin{equation}
\label{eq:blfOmega} 
\blfa{\Omega}(w,v) = \int_\Omega \left( A \nabla w \cdot \nabla v + c w v \right) \diff x', 
\end{equation} 
and $\calL$ is an isomorphism $H^1_0(\Omega) \rightarrow H^{-1}(\Omega)$ 
given by $u \mapsto \blfa{\Omega}(u,\cdot)$. 
The eigenvalue problem: 
Find $(\lambda,\phi) \in {\mathbb R} \times H^1_0(\Omega) \setminus \{0\}$ such that 
$$
 \blfa{\Omega}(\phi,v) = \lambda ( \phi, v)_{L^2(\Omega)}
 \quad \forall v \in H^1_0(\Omega)
$$
has a countable collection of solutions 
$\{\lambda_k, \varphi_k \}_{k \in \mathbb{N}} \subset \R^+ \times H_0^1(\Omega)$, 
with the real eigenvalues enumerated in increasing order, counting multiplicities, 
and such that $\{\varphi_k\}_{k \in \mathbb{N}}$ is an orthonormal basis of 
$L^2(\Omega)$ and an orthogonal basis of $(H_0^1(\Omega), \blfa{\Omega}(\cdot,\cdot))$. 
In terms of these eigenpairs, we introduce, for $s \ge 0$, the spaces 
\begin{equation}
\label{def:Hs}
  \Hs = \left\{ w = \sum_{k=1}^\infty w_k \varphi_k: \| w \|_{\Hs}^2 =
  \sum_{k=1}^{\infty} \lambda_k^s w_k^2 < \infty \right\}.
\end{equation}
%
We denote by $\Hsd$ the dual space of $\Hs$. 
The duality pairing between $\Hs$ and $\Hsd$ will be denoted by $\langle \cdot,\cdot \rangle$. 
Through this duality pairing, we identify 
elements of $f \in \Hsd$ with sequences $(f_k)_{k}$ 
with $\sum_{k} f_k^2 \lambda_k^{-2s} = \|f\|^2_{\Hsd}$, 
which allows us to extend the definition of the norm in \eqref{def:Hs} to $s < 0$.
We have the isometries $\|w\|^2_{L^2(\Omega)} = \|w\|^2_{{\mathbb H}^0}$ and 
$\blfa{\Omega}(w,w) = \|w\|^2_{{\mathbb H}^1}$; 
by (real) interpolation between $L^2(\Omega)$ and $H^1_0(\Omega)$, 
we infer for $s \in (0,1)$ that $\Hs = [L^2(\Omega), H^1_0(\Omega)]_s$.  

For functions $w = \sum_{k} w_k \varphi_k \in {\mathbb H}^1(\Omega)$, 
the operator $\calL:{\mathbb H}^1(\Omega)\rightarrow {\mathbb H}^{-1}(\Omega)$ 
takes the form $\calL w = \sum_{k} \lambda_k w_k \varphi_k$. For 
$s \in (0,1)$ and $w  = \sum_{k} w_k \varphi_k \in {\mathbb H}^s(\Omega)$, 
the operator $\calL^s:\Hs \rightarrow \Hsd$ is defined by 
\begin{equation} \label{def:second_frac}
\calL^s w  
= 
\sum_{k=1}^\infty \lambda_k^{s} w_k \varphi_k. 
\end{equation} 
%
\subsection{The extension property}
\label{sub:CaffarelliSilvestre}
Both extensions, the one by Caffarelli--Silvestre for $\Omega=\mathbb{R}^d$ \cite{CS:07}
and that of Cabr\'e--Tan \cite{CT:10} and Stinga--Torrea for $\Omega$ bounded 
and general elliptic operators \cite{ST:10} require us to deal with the \emph{nonuniformly} 
(but local) linear, second order elliptic equation  \eqref{alpha_harm_intro}.
Here,  Lebesgue and Sobolev spaces with the weight $y^{\alpha}$ for 
$\alpha \in (-1,1)$ \cite{BCdPS:12,CT:10,CS:07, CDDS:11} naturally arise.
If $D \subset \R^{d+1}$, we define $L^2(y^\alpha,D)$ as the Lebesgue space 
for the measure $|y|^{\alpha}\diff x$. 
We also define the weighted Sobolev space
\[
H^1(y^{\alpha},D) =
  \left\{ w \in L^2(y^{\alpha},D): | \nabla w | \in L^2(y^{\alpha},D) \right\},
\]
where $\nabla w$ is the distributional gradient of $w$. 
We equip $H^1(y^{\alpha},D)$ with the norm
\begin{equation}
\label{wH1norm}
\| w \|_{H^1(y^{\alpha},D)} =
\left(  \| w \|^2_{L^2(y^{\alpha},D)} + \| \nabla w \|^2_{L^2(y^{\alpha},D)} 
\right)^{\frac{1}{2}}.
\end{equation}
In view of the fact that $\alpha \in (-1,1)$, 
the weight $y^\alpha$ belongs to the Muckenhoupt class $A_2(\R^{d+1})$
\cite{Javier,FKS:82,GU,Muckenhoupt,Turesson}. 
This, in particular, implies that $H^1(y^{\alpha},D)$ with norm \eqref{wH1norm} 
is Hilbert and $C^{\infty}(D) \cap H^1(y^{\alpha},D)$ is dense in $H^1(y^{\alpha},D)$ 
(cf.~\cite[Proposition 2.1.2, Corollary 2.1.6]{Turesson}, \cite{KO84} and \cite[Theorem~1]{GU}). 

To analyze problem \eqref{alpha_harm_intro} we define the weighted Sobolev space
\begin{equation}
  \label{HL10}
  \HL(y^{\alpha},\C) = \left\{ w \in H^1(y^\alpha,\C): w = 0 \textrm{ on } \partial_L \C\right\}.
\end{equation}
As \cite[inequality (2.21)]{NOS} shows, the following \emph{weighted Poincar\'e inequality} holds:
\begin{equation}
\label{Poincare_ineq}
\| w \|_{L^2(y^{\alpha},\C)} \lesssim \| \nabla w \|_{L^2(y^{\alpha},\C)}
\quad \forall w \in \HL(y^{\alpha},\C).
\end{equation}
Consequently, the seminorm on $\HL(y^{\alpha},\C)$ is equivalent to \eqref{wH1norm}. 
For $w \in H^1(y^{\alpha},\C)$, $\tr w$ denotes its trace onto $\Omega \times \{ 0 \}$ 
which satisfies (see \cite[Proposition 2.5]{NOS})
\begin{equation}
\label{Trace_estimate}
\tr \HL(y^\alpha,\C) = \Hs,
\qquad
  \|\tr w\|_{\Hs} \leq C_{\tr} \| w \|_{\HL(y^\alpha,\C)}.
\end{equation}
Define the bilinear form 
$\blfa{\C}: \HL(y^{\alpha},\C) \times  \HL(y^{\alpha},\C) \to \R$ 
by
\begin{equation}
\label{eq:blf-a} 
\blfa{\C}(v,w) =  \int_\C y^\alpha (\bA \nabla v \cdot \nabla w + c v w) \diff x' \diff y,
\end{equation}
and note that it is continuous and, owing to \eqref{Poincare_ineq}, 
it is also coercive. 
Consequently, it induces an inner product on $\HL(y^\alpha,\C)$ and the \emph{energy norm} $\normC{\cdot}$:
\begin{equation}
\label{eq:norm-C} 
\normC{v}^2:= \blfa{\C}(v,v) \sim \|\nabla v\|^2_{L^2(y^\alpha,\C)}\;.
\end{equation}
Occasionally, we will restrict the integration to the truncated cylinder $\C_\Y$. 
The corresponding bilinear form and norm are denoted by
\begin{equation}
\label{eq:blf-A-truncated} 
\blfa{\C_\Y}(v,w) := \int_{\C_\Y} y^\alpha (\bA \nabla v \cdot \nabla w + c v w ) \diff x' \diff y, 
\qquad \normCC{v}^2 = 
\blfa{\C_\Y}(v,v)\;. 
\end{equation}
With these definitions at hand, 
the weak formulation of \eqref{alpha_harm_intro} reads: 
Find $\ue \in  \HL(y^{\alpha},\C)$ such that
\begin{equation}
\label{eq:ue-variational} 
\blfa{\C}(\ue,v) = d_s \langle f,\tr v \rangle \qquad \forall v \in \HL(y^\alpha,\C).
\end{equation}

The fundamental result of Caffarelli and Silvestre \cite{CS:07} then reads as follows
(see also \cite[Proposition 2.2]{CT:10} and \cite[Theorem 1.1]{ST:10} 
for bounded domains and for general elliptic operators):
given $f \in \Hsd$, let $u \in \Hs$ solve \eqref{fl=f_bdddom}. 
If $\ue \in \HL(y^{\alpha},\C)$ solves \eqref{eq:ue-variational}, then $u = \tr \ue$ and 
\begin{equation}
\label{eq:identity2}
  d_s \calL^s u = \partial_{\nu^\alpha} \ue \quad \text{in } \Omega.
\end{equation}
%
\section{A first order FEM for fractional diffusion}
\label{sec:apriori}
The first work that, in a numerical setting, exploits the 
identity \eqref{eq:identity2} for the design and analysis of 
a finite element approximation of solutions to \eqref{fl=f_bdddom} is \cite{NOS}; 
see also \cite{Otarola}. 
Let us briefly review the main results of \cite{NOS}. 

First, \cite{NOS} truncates $\C$ to $\C_\Y$ and 
places homogeneous Dirichlet boundary conditions on $y = \Y$, thus obtaining an approximation $\ve$
(which, by slight abuse of notation, is understood to coincide with its
 extension by zero from $\C_\Y$ to $\C$).
The error committed in this approximation is exponentially small: 
There holds (see \cite[Theorem 3.5]{NOS})
\begin{equation*}
  \| \nabla(\ue - \ve) \|_{L^2(y^{\alpha},\C )} \lesssim e^{-\sqrt{\lambda_1} \Y/4} \| f\|_{\Hsd},
\end{equation*}
where $\lambda_1>0$ is the first eigenvalue of the operator $\calL$.  

Second, \cite{NOS} develops a regularity theory for $\ue$ in 
weighted Sobolev spaces; see Theorem~\ref{thm:glob_reg_ue} below for a generalization. 
These results reveal that the second order regularity of $\ue$ 
in the extended direction is lost as $y\downarrow 0$.
Thus, \emph{graded} meshes in the extended variable $y$ play a fundamental role.
In the notation of the present work, with a mesh $\calT$ on $\Omega$ 
and a mesh $\calG^M$ on $(0,\Y)$ that is graded towards $y = 0$, the truncated
cylinder $\C_\Y$ is partitioned by tensor product elements $K \times I$ with 
$K \in \calT$ and $I \in \calG^M$. 
On this mesh, the tensor product space $\V^{1,1}_{h,M}(\calT,\calG^M)$ of piecewise 
bilinears in $\Omega\times (0,\Y)$ (see \eqref{eq:TPFE} for the precise definition) 
is used in a Galerkin method. 
The Galerkin approximation $\ue_{h,M} \in \V^{1,1}_{h,M}(\calT,\calG^M)$ 
of 
$\ve$ satisfies a best approximation property \`a la C\'ea. 
%
From there, upon studying piecewise polynomial interpolation in Muckenhoupt 
weighted Sobolev spaces \cite{NOS,NOS2} error estimates were obtained
under the assumption that $f \in \Ws$ and that $\Omega$ is convex
(see \cite[Theorem 5.4]{NOS} and \cite[Corollary 7.11]{NOS}):
\begin{theorem}[a priori error estimate]
\label{TH:fl_error_estimates}
Let $\calG^M$ be suitably graded towards $y = 0$ and $\V^{1,1}_{h,M}$ 
be constructed with tensor product elements 
and $\ue_{h,M} \in \V^{1,1}_{h,M}$ denote the Galerkin approximation to $\ve$.  
Then, for suitable truncation parameter $\Y \sim \log \calN_{\Omega,\Y}$ we have, 
with the total number of unkowns 
$\calN_{\Omega,\Y}:= \# \calT \#\calG^M$ 
\begin{align*}
  \| u - \tr \ue_{h,M} \|_{\Hs} &\lesssim
  \| \nabla(\ue - \ue_{h,M}) \|_{L^2(y^\alpha,\C)} \\ &\lesssim
|\log \calN_{\Omega,\Y} |^s(\calN_{\Omega,\Y})^{-1/(d+1)} \|f \|_{\mathbb{H}^{1-s}(\Omega)}.
\end{align*}
\end{theorem}

\begin{remark}[complexity]
Up to logarithmic factors, Theorem~\ref{TH:fl_error_estimates} 
yields rates of convergence of
$(\calN_{\Omega,\Y})^{-1/(d+1)}$. 
In terms of error versus work, this $P_1$-FEM
is sub-optimal as a method to compute in $\Omega$. 
In this paper we propose and study $P_1$-FE methods in $\Omega$ that 
afford an error decay $(\calN_{\Omega,\Y})^{-1/d}$ 
(up to possibly logarithmic terms).
\eremk
\end{remark}

\section{Analytic regularity}\label{S:analytic-regularity}
We obtain regularity results for the solution of \eqref{alpha_harm_intro} 
that will underlie the analysis of the various FEMs in Section~\ref{S:FEMy} and \ref{sec:hpx}. 
We begin by recalling 
that if $u = \sum_{k=1}^{\infty} u_k \varphi_k $ solves \eqref{fl=f_bdddom}, 
then the unique solution $\ue$ 
of problem \eqref{alpha_harm_intro} admits the representation \cite[formula (2.24)]{NOS}
\begin{equation} \label{eq:SepVar}
\ue(x',y) = \sum_{k=1}^\infty u_k \varphi_k(x') \psi_k(y),\quad  u_k := \lambda_k^{-s} f_k .
\end{equation}
We also recall that $\{ \lambda_k, \varphi_k \}_{k \in \mathbb{N}}$ is the set of eigenpairs of the elliptic operator $\calL$, supplemented with homogeneous Dirichlet boundary conditions. The functions $\psi_k$ solve
\begin{equation}
\label{psik}
\begin{dcases}
  \frac{\diff^2}{\diff y^2}\psi_k(y) + \frac{\alpha}{y} \frac{\diff}{\diff y} \psi_k(y) - \lambda_k \psi_k(y) = 0, 
& y \in (0,\infty), 
\\
\psi_k(0) = 1, & \lim_{y\rightarrow \infty} \psi_k(y) = 0.
\end{dcases}
\end{equation}
Thus, if $s = \tfrac{1}{2}$, we have 
$\psi_k(y) = \exp(-\sqrt{\lambda_k}y)$ \cite[Lemma 2.10]{CT:10}; more generally, 
if $s \in (0,1) \setminus \{ \tfrac{1}{2}\}$, then \cite[Proposition 2.1]{CDDS:11}
\[
 \psi_k(y) = c_s (\sqrt{\lambda_k}y)^s K_s(\sqrt{\lambda_k}y),
\]
where $c_s = 2^{1-s}/\Gamma(s)$ and $K_s$ denotes the modified Bessel function of the second kind. 
We refer the reader to \cite[Chapter 9.6]{Abra} 
for a comprehensive treatment of the Bessel function $K_s$ 
and recall the following properties.

\begin{lemma}[properties of $K_\nu$]
\label{lemma:bessel}
The modified Bessel function of the second kind $K_\nu$ satisfies:
\begin{enumerate}[(i)]

\item \label{itemi} 
For $\nu >-1$ and $z>0$, $K_{\nu}(z)$ is real and positive \cite[Chapter 9.6]{Abra}.

\item \label{itemii} 
For $\nu \in \R$, $K_{\nu}(z) = K_{-\nu}(z)$ \cite[Chapter 9.6]{Abra}.
  
\item \label{itemiii}
For $\nu > 0$, \cite[estimate (9.6.9)]{Abra} 
\begin{equation}
\label{asympat0}
\lim_{z \downarrow 0} \frac{K_{\nu}(z)}{\sr \Gamma(\nu) \left( \sr z \right)^{-\nu} } = 1.
\end{equation}
  
\item \label{itemiv}For $\ell \in \mathbb{N}$, \cite[formula (9.6.28)]{Abra}
\begin{equation}
\label{eq:diff_formula}
\left( \frac{1}{z} \frac{\diff }{\diff z}\right)^{\ell} \left( z^{\nu} K_{\nu}(z) \right)
= (-1)^{\ell} z^{\nu-\ell}K_{\nu-\ell}(z).
\end{equation}
  
\item \label{itemv} 
For $z > 0$, $z^{\min\{ \nu,1/2 \}} e^z K_{\nu}(z)$ is a decreasing function  
\cite[Theorem 5]{MS:01}.
  
\item \label{itemvi} 
For $\nu >0$, \cite[estimate (9.7.2)]{Abra} 
\[
 K_{\nu}(z) \sim \sqrt{ \frac{\pi}{2z} } e^{-z}, \qquad z \rightarrow \infty, 
\quad | \arg z | \leq 3\pi/2-\delta, \quad \delta >0.
\]
\end{enumerate}
\end{lemma}
\begin{remark}[consistency for $s= \frac{1}{2}$]
\label{rk:consistency}
A basic computation allows us to conclude that $c_{\frac{1}{2}} = \sqrt{\frac{2}{\pi}}$. 
On the other hand, 
formulas (9.2.10) and (9.6.10) in \cite{Abra} 
yield $K_{\frac{1}{2}}(z) = \sqrt{\tfrac{\pi}{2z}}e^{-z}$. 
We thus have arrived at
 \[
  \lim_{s \rightarrow \frac{1}{2}}\psi_k(y)= \exp(-\sqrt{\lambda_k}y) \quad \forall y>0
 \]
 for all $y>0$.
\eremk
\end{remark}

We now analyze the regularity properties of $\ue$ when $s \in (0,1)$.
On the basis of the representation formula \eqref{eq:SepVar} we see that 
it is essential to derive regularity estimates for the solution $\psi_k$ of problem \eqref{psik}. 
To accomplish this task, we define the function $\psi(z) = c_s z^s K_s(z)$ and notice that
\begin{equation}
\label{eq:phi}
\frac{\diff^2}{\diff z^2}\psi(z)- \psi(z) + \frac{\alpha}{z} \frac{\diff}{\diff z}\psi(z) = 0, \quad 
z \in (0,\infty), \quad \psi(0) = 1, \quad \lim_{z\rightarrow \infty} \psi(z) = 0.
\end{equation}
This, for any $\ell \in \mathbb{N}_0$, allows us to obtain that 
\begin{align*}
\frac{\diff^{\ell+2}}{\diff z^{\ell+2}} \psi(z) 
&=  \frac{\diff^\ell}{\diff z^{\ell}} \psi(z) - \alpha \frac{\diff^\ell}{\diff z^{\ell}} \left( z^{-1} \frac{\diff}{\diff z} \psi(z)  \right) 
\\
&=  \frac{\diff^\ell}{\diff z^{\ell}} \psi(z) - \alpha \sum_{j=0}^\ell \binom{\ell}{j}  
\frac{\diff^j}{\diff z^{j}}(z^{-1}) \frac{\diff^{\ell-j}}{\diff z^{\ell-j}}\psi'(z).
\end{align*}
We thus have arrived at the bound
\begin{equation}
\label{eq:psi_kl+2}
\left| \frac{\diff^{\ell+2}}{\diff z^{\ell+2}} \psi(z) \right| 
\leq 
\left| \frac{\diff^\ell}{\diff z^{\ell}} \psi(z) \right| 
+ |\alpha| \sum_{j=0}^\ell \frac{\ell!}{(\ell-j)!} z^{-(1+j)} 
\left| \frac{\diff^{\ell+1-j}}{\diff z^{\ell+1-j}} \psi(z) \right|,
\end{equation}
which is essential to derive the following asymptotic result.

\begin{lemma}[behavior of $\psi$ near $z=0$]
Let $\psi$ solve \eqref{eq:phi}, $z \in (0,1)$ and $\ell \in \mathbb{N}$. 
Then there is a constant $C$ independent of $z$, $\ell$ and $s$ such that
\begin{equation}
\label{eq:by_induction}
  \left| \frac{\diff^{\ell}}{\diff z^{\ell}} \psi(z) \right| \leq C d_s \ell!  z^{2s-\ell},
 \end{equation}
where, as before, $d_s = 2^{1-2s}\Gamma(1-s)/\Gamma(s)$.
\label{lem:nearz0}
\end{lemma} 
\begin{proof}
We proceed by induction. Let us first assume that $\ell=1$. 
The differentiation formula \eqref{eq:diff_formula} with $\ell = 1$ yields that
\begin{equation}
\label{eq:phi_prime}
  \psi'(z) = c_s (z^s K_s(z) )' = -c_s z^s K_{s-1}(z) = -c_s z^s K_{1-s}(z),
\end{equation}
where we used Lemma~\ref{lemma:bessel} \eqref{itemii}. 
The asymptotic formula \eqref{asympat0} shows that there is 
$\tilde C$ independent of $s$ such that, for every $z \in (0,1)$, we have
\[
  \left| \frac{ K_{1-s}(z) }{ \tfrac12 \Gamma(1-s) (\tfrac12 z )^{-(1-s)} } -1 \right| \leq \tilde C.
\]
Set $C = \tilde{C} + 1$ to arrive at the fact that we have, for all $z \in (0,1)$,
\[
  \left| \frac{\diff}{\diff z} \psi(z) \right| \leq \left| \frac{ K_{1-s}(z) }{ \tfrac12 \Gamma(1-s) (\tfrac12 z )^{-(1-s)} } \right| \left( \frac12 \Gamma(1-s) \left(\frac12 z \right)^{-(1-s)}  \right)c_s z^s \leq C d_s z^{2s-1},
\]
which is \eqref{eq:by_induction} for $\ell=1$.

We now assume that \eqref{eq:by_induction} holds for every $j \leq \ell +1 $. 
This, on the basis of the bound \eqref{eq:psi_kl+2}, implies that
\begin{align*}
   \left| \frac{\diff^{\ell+2}}{\diff z^{\ell+2}} \psi(z) \right| & \leq C d_s \ell! z^{2s-\ell}  + 
   C d_s z^{2s-\ell-2} \sum_{j=0}^\ell \frac{\ell!}{(\ell-j)!} (\ell+1-j)!
   \\
   & \leq  C d_s \ell!  z^{2s-\ell-2} \left[ 1 + \sum_{i=1}^{\ell + 1} i\right] 
   = C d_s \ell!  z^{2s-\ell-2}\left[ 1 + \frac{1}{2}(\ell + 1)(\ell + 2) \right],
\end{align*}
because $z \in (0,1)$. 
Therefore
\[
\left| \frac{\diff^{\ell+2}}{\diff z^{\ell+2}} \psi(z) \right| \leq C d_s (\ell+2)! z^{2s-\ell-2},
\]
as we intended to show.
\end{proof}

We now analyze the behavior of $\psi$ for large values of $z$. 
In particular, we will show that $\psi$ and all its derivatives decay exponentially 
with respect to $z$. 

\begin{lemma}[behavior of $\psi$ for $z$ large]
Let $\psi$ solve \eqref{eq:phi}, $z \geq 1$, $\ell \in \mathbb{N}_0$ 
and $\epsilon \in (0,1)$. 
Then there is a constant $C_{\epsilon,s}$ that is
independent of $z$ and $\ell$ such that
\begin{equation}
\label{eq:by_induction_exp}
\left| \frac{\diff^{\ell}}{\diff z^{\ell}} \psi(z) \right| \leq
C_{\epsilon,s} \ell! \epsilon^{-\ell} z^{s-\ell-\frac{1}{2}} e^{-(1-\epsilon) z},
\end{equation}
where $C_{\epsilon,s}$ blows up when $\epsilon \uparrow 1$.
\label{lem:nearzinfty}
\end{lemma}
\begin{proof}
The proof is a consequence of Cauchy's integral 
formula for derivatives \cite{MR510197,MR0147623} 
and Lemma~\ref{lemma:bessel} \eqref{itemvi}. Let $\epsilon \in (0,1)$ 
and $B_{\sigma}(\zeta)$ denote the ball with center $\zeta$ 
and radius $\sigma$. For a fixed $z \geq 1$, we thus have that
\[
\left| \frac{\diff^{\ell}}{\diff z^{\ell}} \psi(z) \right| 
= 
\left| \frac{\ell!}{2 \pi i} 
\int_{\zeta \in \partial B_{\epsilon z}(z)} \frac{\psi(\zeta)}{(\zeta-z)^{\ell+1}} \diff \zeta \right| 
\leq \ell! \epsilon^{-\ell} z^{-\ell} \max_{\zeta \in \partial B_{\epsilon z}(z)} | \psi(\zeta) |,
\]
where $\ell \in \mathbb{N}_0$. 
We now recall that $\psi(z) = c_s z^s K_s(z)$ and invoke Lemma~\ref{lemma:bessel} \eqref{itemvi} to conclude that
\[
 \left| \frac{\diff^{\ell}}{\diff z^{\ell}} \psi(z) \right| \leq C_{\epsilon,s} c_s \ell!
 \epsilon^{-\ell} z^{s-\ell-\frac{1}{2}}e^{-(1-\epsilon) z},
\]
with $C_{\epsilon,s} = C \max \{ (1+\epsilon)^{s-\frac{1}{2}},(1-\epsilon)^{s-\frac{1}{2}} \}$ 
and $C$ such that $K_{s}(z) \leq C z^{-\frac{1}{2}}e^{-z}$ for $z \geq 1$. 
Notice that $C_{\epsilon,s}$ can be bounded independently of $s \in (0,1)$ 
and that blows up when $\epsilon \uparrow 1$. This concludes the proof.
\end{proof}

\begin{remark}[Cauchy's integral formula]
The technique used in the proof of Lemma \ref{lem:nearzinfty} that is based on the 
Cauchy's integral formula can also be applied to analyze the behavior of $\psi$ near $z=0$. 
However, the obtained estimate with such a technique is not quite as sharp 
as \eqref{eq:by_induction} since it includes the term 
$\epsilon^{-\ell}$ with $\epsilon \in (0,1)$, 
as it appears in the estimate \eqref{eq:by_induction_exp}.
\eremk
\end{remark}

To analyze global regularity properties 
of the $\alpha$--harmonic extension $\ue$, 
we define the weight
\begin{equation}
\label{eq:weight}
\omega_{\beta,\gamma}(y) = y^{\beta}e^{\gamma y}, \qquad 0 \leq \gamma < 2 \sqrt{\lambda_1},
\end{equation}
with a parameter $\beta\in \R$ that will be specified later, and we recall that the parameter $\lambda_1>0$ is the smallest eigenvalue of $\calL$. With the weight \eqref{eq:weight} at hand, we define the weighted norm
\begin{equation}
 \label{eq:weighted_norm}
 \| v \|_{L^2(\omega_{\beta,\gamma},\C)} 
:= \left( \int_0^{\infty} \int_{\Omega} \omega_{\beta,\gamma}(y) 
          |v(x',y)|^2 \diff x' \diff y
 \right)^{\frac{1}{2}}.
\end{equation}

We now proceed to study how certain weighted integrals 
of the derivatives of $\psi$ behave. 
To do so, we define, for $\beta$, $\delta \in \mathbb{R}$, $\ell \in \mathbb N$, and $\lambda > 0$
\begin{equation}
\label{eq:defofbigPhi}
\Phi(\delta,\gamma,\lambda) 
= 
\int_{ 0 }^{ \infty } z^{ \delta }e^{\gamma z/\sqrt{\lambda}} \left| \psi(z) \right|^2 \diff z
\end{equation}
and
\begin{equation}
\label{eq:defofbigPsi}
\Psi_\ell(\beta,\gamma,\lambda) 
= \int_{ 0 }^{ \infty } z^{\beta + 2\ell}e^{\gamma z/\sqrt{\lambda}} \left| \frac{\diff^\ell}{\diff z^\ell} \psi(z) \right|^2 \diff z;
\end{equation}
$\gamma$ is such that \eqref{eq:weight} holds.
Let us now bound the integrals $\Phi(\delta,\gamma,\lambda)$ and $\Psi_\ell(\beta,\gamma,\lambda)$.

\begin{lemma}[bounds on $\Phi$ and $\Psi_\ell$]
\label{lem:bound_Psi_finite_interval}
Let $\delta > -1$, $\beta > -1 - 4s$, $\ell \in \bbN$ and 
let $\gamma$ be such that $0 \leq \gamma < 2 \sqrt{\lambda_1}$. 
If $\lambda \geq \lambda_1$, then we have that
\begin{equation}
\label{eq:bound_on_Phi}
\Phi(\delta,\gamma,\lambda) \lesssim 1,
\end{equation}
where the hidden constant is independent of $\lambda$. 
In addition, there exists $\kappa > 1$ such that we have the following bound
\begin{equation}
\label{eq:bound_on_Psi}
\Psi_\ell(\beta,\gamma,\lambda) \lesssim \kappa^{2\ell} (\ell!)^2,
\end{equation}
where the hidden constant is independent of $\ell$ and $\lambda$.
\end{lemma}
\begin{proof}
We derive \eqref{eq:bound_on_Psi}. As a first step, we write $\Psi_\ell = \Psi_\ell(\beta,\gamma,\lambda)$ as follows:
\begin{equation}
\label{eq:bound_on_Psi_step1}
 \Psi_\ell = \int_{ 0 }^{ 1 } z^{\beta + 2\ell}e^{\frac{\gamma z}{\sqrt{\lambda}}} \left| \frac{\diff^\ell}{\diff z^\ell} \psi(z) \right|^2 \diff z +  \int_{ 1 }^{ \infty } z^{\beta + 2\ell}e^{\frac{\gamma z}{\sqrt{\lambda}}} \left| \frac{\diff^\ell}{\diff z^\ell} \psi(z) \right|^2 \diff z = \textrm{I} + \textrm{II},
\end{equation}
and estimate each term separately. 

We start by bounding $\textrm{I}$. To accomplish this task we notice that, 
since $0 \leq \gamma < 2 \sqrt{\lambda_1}$ and $\lambda \geq \lambda_1$ we have that
\[
 \sup_{z \in (0,1)}e^{\frac{\gamma z}{\sqrt{\lambda}}} < \sup_{z \in (0,1) } e^{2 z} \le e^{2}.
\]
Consequently, an application of the results of Lemma \ref{lem:nearz0} yields
\[
 \textrm{I}= \int_{ 0 }^{ 1 } z^{\beta + 2\ell}e^{\frac{\gamma
     z}{\sqrt{\lambda}}} \left| \frac{\diff^\ell}{\diff z^\ell}
 \psi(z) \right|^2 \diff z \lesssim d_s^2 (\ell!)^2 
  \int_0^1 z^{\beta+2\ell+2(2s-\ell)} \diff z 
 \lesssim d_s^2 (\ell!)^2,
\]
where last integral converges because $\beta > -1 - 4s$. 
Notice that the hidden constant blows up when $\beta \downarrow -1 - 4s$.

We now estimate the term $\textrm{II}$ in \eqref{eq:bound_on_Psi_step1}. 
To do this we utilize the estimate \eqref{eq:by_induction_exp} of Lemma \ref{lem:nearzinfty} 
as follows:
\[
\textrm{II} \leq C_{\epsilon}^2 c_s^2 (\ell!)^2 \epsilon^{-2\ell} 
\int_1^{\infty} z^{\beta + 2\ell} z^{2s-2\ell- 1} e^{\frac{\gamma z}{\sqrt{\lambda}}}  
                e^{-2(1-\epsilon)z} \diff z.
\]
Define
\[
\hat \gamma:= \sup_{\lambda \geq \lambda_1} \left( \frac{\gamma}{\sqrt{\lambda}} - 2(1-\epsilon) \right).
\]
Notice that, since 
$0 \leq \frac{\gamma}{\sqrt{\lambda_1}} < 2 $ 
by \eqref{eq:weight},
the parameter $\epsilon \in (0,1)$ 
can be selected such that $ \hat \gamma < 0$. 
Consequently
\[
\textrm{II} \lesssim C_{\epsilon}^2 c_s^2 (\ell!)^2 \epsilon^{-2\ell} 
\int_1^{\infty} z^{\beta + 2s-1}  e^{\hat \gamma z} \diff z 
\lesssim  C_{\epsilon}^2 c_s^2 
(\ell!)^2 \epsilon^{-2\ell}.
\]

Replacing the estimates for the terms $\textrm{I}$ and $\textrm{II}$ into 
\eqref{eq:bound_on_Psi_step1} and considering $\kappa = \epsilon^{-1} > 1$
we arrive at the desired estimate \eqref{eq:bound_on_Psi}.
To obtain the estimate \eqref{eq:bound_on_Phi} we decompose $\Phi$ 
as in \eqref{eq:bound_on_Psi_step1} and use that, as estimate \eqref{asympat0} 
shows, $\psi$ is bounded as $z \downarrow 0^+$ and decays exponentially 
to zero as $z \uparrow \infty$; see Lemma~\ref{lemma:bessel} \eqref{itemv} and \eqref{itemvi}. 
For brevity, we skip details.
\end{proof}

Now, on the basis of Lemma \ref{lem:bound_Psi_finite_interval}, 
we provide global regularity results for the $\alpha$-harmonic extension $\ue$ 
in weighted Sobolev spaces.

\begin{theorem}[global regularity of $\ue$]
\label{thm:glob_reg_ue}
Let $\ue \in \HL(y^{\alpha},\C)$ solve \eqref{alpha_harm_intro} with $s \in (0,1)$. 
Let $0 \leq \tilde{\nu} < s$ and $0 \leq \nu < 1+s$. 
Then there exists $\kappa > 1$ such that the following holds
for all $\ell \in \mathbb{N}_0$
with the weight $w_{\beta,\gamma}$ given by \eqref{eq:weight}: 
\begin{align}
\label{eq:reg_y_l}
 \|\partial_{y}^{\ell+1} \ue \|_{L^2(\omega_{\alpha + 2 \ell -2\tilde{\nu},\gamma},\C)} 
& \lesssim \kappa^{\ell+1} (\ell+1)! \| f \|_{\mathbb{H}^{-s + \tilde{\nu}}(\Omega)}, 
\\
\label{eq:reg_x_grad_mu}
 \| \nabla_{x'} \partial_{y}^{\ell + 1} \ue \|_{L^2(\omega_{\alpha+2(\ell+1)-2\nu,\gamma},\C)} 
  & \lesssim \kappa^{\ell+1}(\ell+1)! \| f \|_{\mathbb{H}^{-s+\nu}(\Omega)},
\\
\label{eq:reg_x_mu}
 \| \calL_{x'} \partial_{y}^{\ell + 1} \ue \|_{L^2(\omega_{\alpha+2(\ell+1)-2\nu,\gamma},\C)} 
&\lesssim \kappa^{\ell+1}(\ell+1)! \| f \|_{\mathbb{H}^{1-s+\nu}(\Omega)}. 
\end{align}
In all these inequalities, the implied
constants are independent of $\ell$, $\ue$ and $f$.
In addition, if  $0 \leq \nu' < 1-s$ then
\begin{align}
\label{eq:reg_x_Delta}
 \|  {\calL_{x'}} \ue \|_{L^2(\omega_{\alpha-2\nu',\gamma},\C)} 
     &\lesssim \| f \|_{\mathbb{H}^{1-s+\nu'}(\Omega)}, \\
\label{eq:reg_x}
 \|  \nabla_{x'} \ue \|_{L^2(\omega_{\alpha-2\nu',\gamma},\C)}  
     &\lesssim \| f \|_{\mathbb{H}^{-s+\nu'}(\Omega)}, \\
\label{eq:reg_L2}
 \|  \ue \|_{L^2(\omega_{\alpha-2\nu',\gamma},\C)}  
     &\lesssim \| f \|_{\mathbb{H}^{-1-s+\nu'}(\Omega)}, 
\end{align}
where the constant implied in $\lesssim$ is independent of $\ue$ and $f$.
\end{theorem}
\begin{proof}
We follow \cite[Theorem 2.7]{NOS} and thus invoke the representation formula \eqref{eq:SepVar} 
to arrive at
\[
\|\partial_{y}^{\ell+1} \ue \|^2_{L^2(\omega_{\alpha + 2 \ell -2\sigma,\gamma},\C)}  
= \sum_{k=1}^{\infty} f_k^2 \lambda_k^{-2s}
  \int_0^\infty y^{\alpha + 2 \ell -2\sigma}e^{\gamma y}
\left|\frac{\diff^{\ell+1} }{\diff y^{\ell+1}} \psi_k (y) \right|^2 \diff y\;.
\]
We introduce the change of variable $z = \sqrt{\lambda_k} y$ and
recall that $\psi(z) = c_s z^s K_s(z)$ and
$\psi_k(y)=\psi(\sqrt{\lambda_k} y)$ as well as 
the definition of $\Psi_{\ell}$ given as in \eqref{eq:defofbigPsi}, to obtain that
\begin{multline*}
\|\partial_{y}^{\ell+1} \ue \|^2_{L^2(\omega_{\alpha + 2 \ell -2\sigma,\gamma},\C)}  
= 
\sum_{k=1}^{\infty} f_k^2 \lambda_k^{-2s + (\ell + 1) - 
\left(\tfrac{\alpha+2\ell-2\sigma}{2}\right) -\tfrac12} 
\Psi_{\ell+1}(\alpha-2\sigma-2,\gamma,\lambda_k)
\\
\lesssim (\ell+1)!^2 \kappa^{2(\ell+1)}\sum_{k=1}^{\infty} f_k^2 \lambda_k^{\sigma-s} 
= 
(\ell+1)!^2 \kappa^{2(\ell+1)}\| f\|^2_{\mathbb{H}^{-s+\sigma}(\Omega)},
\end{multline*}
where the last inequality follows from 
the estimate \eqref{eq:bound_on_Psi} with $\beta = \alpha - 2\sigma -2 = 1-2s -2\sigma -2 > -1-4s$. 
%

We now derive \eqref{eq:reg_x_mu}; the proof of the estimate
\eqref{eq:reg_x_grad_mu} follows by using similar arguments. 
As before, we arrive at
\begin{align*}
  & \|  \calL_{x'} \partial_y^{\ell+1} \ue \|^2_{L^2(\omega_{\alpha+2(\ell+1)-2\nu,\gamma},\C)} 
  \\
  &= \sum_{k=1}^{\infty} f_k^2 \lambda_k^{2(1-s)} \int_0^\infty y^{\alpha+2(\ell+1)-2\nu} 
 e^{\gamma y}\left|\frac{\diff^{\ell+1} }{\diff y^{\ell+1}} \psi_k (y) \right|^2 \diff y
  \\
  &= \sum_{k=1}^\infty f_k^2 \lambda_k^{2(1-s) + (\ell+1) - \left( \tfrac{\alpha+2(\ell+1)-2\nu}{2} \right) -\tfrac12} \Psi_{\ell+1}(\alpha-2\nu,\gamma,\lambda_k),
\end{align*}
where we applied again the change of variable $z = \sqrt{\lambda_k} y$ 
and used the definition of $\Psi_{\ell}$ given by \eqref{eq:defofbigPsi}. 
We now notice that $\alpha-2\nu > 1-2s - 4 -2s = -1-4s$. 
Thus an application of the estimate \eqref{eq:bound_on_Psi} with 
$\beta = \alpha - 2\nu$ reveals that
\[
   \|  \mathcal{L}_{x'}  \partial_y^{\ell+1} \ue \|^2_{L^2(\omega_{\alpha+2(\ell+1)-2\nu,\gamma},\C)} 
   \lesssim \kappa^{2(\ell+1)} (\ell+1 )!^2 \| f \|^2_{\mathbb{H}^{1-s +\nu}(\Omega)}.
\]
This yields \eqref{eq:reg_x_mu}. 

The proofs of \eqref{eq:reg_x_Delta}, \eqref{eq:reg_x}, \eqref{eq:reg_L2} 
rely on similar arguments using that $\nu' < 1-s$ implies 
$\delta:= \alpha - 2 \nu' = 1-2s - 2 \nu' > 1-2s-2(1-s) = -1$, 
and thus, as a consequence of \eqref{eq:bound_on_Phi}, 
that $\Phi(\delta,\gamma,\lambda) \lesssim 1$. 
This concludes the proof.
\end{proof}

\section{$h$-FE discretization in $\Omega$}\label{S:FEMy}
%
We now begin with the discretization of \eqref{eq:ue-variational}.
The structure of this section is as follows:
in Section~\ref{S:NtFESpc}, we introduce the FE approximation 
in $\Omega$ and fix notation on Finite Element spaces.
Section \ref{S:FEdiscretization} introduces the FE discretization
in $\calC$ in abstract form.
Section \ref{S:ErrSplt} next addresses a basic decomposition of
the FE discretization error which decomposes the FE discretization
error into two parts: a semidiscretization error with respect to $x'\in \Omega$,
and a corresponding error with respect to $y\in (0,\Y)$, where $0<\Y<\infty$
denotes a truncation parameter of the cylinder $(0,\infty)$.
Section \ref{S:hFEErrAn} then addresses two first order tensor product
FEMs in $\calC$.
The first one, as in \cite{NOS}, is a full tensor product FEM and for it we show 
the first order rate of convergence in $\Omega$, but at superlinear complexity
in terms of the number $\calN_\Omega$ of degrees of freedom in $\Omega$.
To reduce the complexity, we propose the second, novel approach:
by sparse tensor product $P_1$ discretization of the extended problem in $\C$,
we show the same convergence rate, but with (essentially) linear complexity
in terms of $\calN_\Omega$
requiring only marginally more regularity of the data $f$ in $\Omega$.
Section~\ref{S:hpFEM} addresses the use of an $hp$-FEM in 
the extended variable $y$, combined with a $P_1$-FEM in $\Omega$.
\subsection{Notation and FE spaces}
\label{S:NtFESpc}
For a truncation parameter $\Y>0$ (which is fixed, and which will be selected ahead),
we denote by $\calG^M$ a generic partition of $[0,\Y]$ into $M$ intervals. 
In particular, the following two types of partitions, that are refined 
towards $y = 0$, will be essential for our purposes:
\begin{enumerate}[$\bullet$]
\item 
\emph{Graded} meshes $\calG^{k}_{gr,\eta}$. 
Here $k$ indicates the mesh size near $y = 1$ and $\eta$ characterizes the mesh
grading towards $y = 0$; see Section~\ref{S:p1y} for details.

\item \emph{Geometric} meshes $\calG^{M}_{geo,\sigma}$. 
This mesh has $M$ elements and $\sigma \in (0,1)$ is the subdivision ratio; 
see Section~\ref{S:hp-interpolant} for details.
\end{enumerate} 

Given a mesh $\calG^M = \{ I_m \}_{m=1}^M$ in $[0,\Y]$, 
where $I_m = [y_{m-1},y_m]$, $y_0 = 0$ and $y_M = \Y$, we associate to $\calG^M$
a \emph{polynomial degree distribution} $\bmr = (r_1,r_2,\dots,r_M) \in \bbN^M$.
With these ingredients at hand we define the finite element space
\[
S^\bmr((0,\Y),\calG^M) 
= 
\left \{ v_M \in C[0,\Y]: v_M|_{I_m} \in \mathbb{P}_{r_m}(I_m), I_m \in \calG^M, m =1, \dots, M \right \}.
\]
We also define the subspace of $S^\bmr((0,\Y),\calG^M)$ of functions that vanish at $y=\Y$:
\[
S_{\{\Y\}}^\bmr((0,\Y),\calG^M) = \left \{ v_M \in S^\bmr((0,\Y),\calG^M): v_M(\Y) = 0 \right \}.
\]
In the particular case that $r_i=r$  for $i=1,\ldots,M$, 
we write $S^r((0,\Y),\calG^M)$ or $S_{\{\Y\}}^r((0,\Y),\calG^M)$ as appropriate.
In $\Omega$, we consider Lagrangian FEM of polynomial degree $q\geq 1$
based on shape-regular, simplicial triangulations denoted by $\calT$.
Denote by 
$h(\calT) = \max \{\diam(K) : K \in \calT \}$ 
the mesh width of $\calT$. We thus introduce 
\[
S^q_0(\Omega,\calT)
= 
\left \{ v_h \in C(\bar \Omega): v_h|_{K} \in \mathbb{P}_{q}(K) \quad 
\forall K \in \calT, \ v_h|_{\partial \Omega} = 0 \right \}.
\]
In what follows we will also consider nested sequences $\{ \calT^\ell \}_{\ell \geq 0}$ 
of triangulations of $\Omega$ that are generated by bisection--tree 
refinement of a coarse, regular initial triangulation $\calT^0$ of $\Omega$. 
We denote by $h_\ell =\max\{\diam(K) : K \in \calT^\ell \}$ the mesh width of $\calT^\ell$.

By $\Pi_{x'}^q: H^1_0(\Omega) \to S^q_0(\Omega,\calT)$, 
we denote a FE quasi--interpolation operator defined on 
$L^2(\Omega)$ that, when restricted to $H^1_0(\Omega)$,
preserves homogeneous Dirichlet boundary conditions.
We assume that $\Pi_{x'}^q$ has optimal asymptotic approximation 
properties in $L^2(\Omega)$ and $H^1(\Omega)$ on regular, locally refined, and nested
bisection--tree mesh sequences 
$\{ \calT^\ell \}_{\ell \geq 0}$ in $\Omega$.
In addition, we assume that $\Pi_{x'}^q$ is concurrently stable in $L^2(\Omega)$ and $H^1(\Omega)$.
In the particular case that $q\leq12$ we will set $\Pi_{x'}^q$ to be the $L^2(\Omega)$ 
projection onto $S^q_0(\Omega,\calT)$.
We refer, in particular, to \cite{GspzHeineSiebert2016} for a verification 
of the requisite stability and approximation properties over nested bisection--tree meshes.

We define the finite--dimensional \emph{tensor product space}
\begin{equation}\label{eq:TPFE}
\V^{q,\bmr}_{h,M}(\calT,\calG^M) 
:= 
S^q_0(\Omega,\calT) \otimes S_{\{ \Y\}}^\bmr((0,\Y),\calG^M)
\subset \HL(y^{\alpha},\C)\;,
\end{equation}
and write $\V_{h,M}$ if the arguments are clear from the context.
In the ensuing error analysis, we also require 
semidiscretizations which are based on the following (infinite--dimensional) 
Hilbertian tensor product spaces
\begin{equation}\label{eq:xySemiDis}
\begin{array}{l} 
\displaystyle
\V^q_{h}(\C_\Y) := S^q_0(\Omega,\calT_h) \otimes \HL(y^\alpha, (0,\Y)) \subset \HL(y^{\alpha},\C)\;,
\\ 
\displaystyle
\V^\bmr_{M}(\C_\Y) := H^1_0(\Omega)\otimes S_{\{\Y\}}^\bmr((0,\Y),\calG^M) \subset \HL(y^{\alpha},\C)\;.
\end{array}
\end{equation}
Both of them are closed subspaces of $\HL(y^{\alpha},\C)$, so that
Galerkin projections with respect to the inner product 
given by the bilinear form $\blfa{\C_\Y}$ in \eqref{eq:blf-A-truncated}
are well defined. We denote these projections by $G_h^q$ and $G_M^\bmr$, respectively.
To the space $\V^{q,\bmr}_{h,M}(\calT,\calG^M)$, defined in \eqref{eq:TPFE}, 
we can also associate a Galerkin projection with respect to $\blfa{\C_\Y}$. 
We remark that this projector is the composition of the semidiscrete projections:
\begin{equation}\label{eq:TPOp}
G^{q,\bmr}_{h,M} 
= G^q_{h} \circ G^\bmr_{M}
= G^\bmr_{M} \circ G^q_{h}: \;\HL(y^\alpha,\C) \to \V^{q,\bmr}_{h,M}(\calT,\calG^M)\;.
\end{equation}
%
\subsection{FE discretization and quasioptimality}
\label{S:FEdiscretization}
The \emph{FE approximation} $\ue_{h,M}$ is defined as $\ue_{h,M} = G_{h,M}^{q,\bmr} \ue \in \V_{h,M}$, 
\ie it satisfies
\begin{equation}\label{eq:alpha_weak_UhM}
 \blfa{{\C_\Y}}(\ue_{h,M} , \phi)
  = d_s \langle f, \tr \phi\rangle \quad \forall \phi \in \V_{h,M}\;.
\end{equation}
Coercivity of $\blfa{\C_\Y}$ immediately implies existence and uniqueness of $\ue_{h,M}$. 
In addition, Galerkin orthogonality gives quasioptimality of $\ue_{h,M}$. 
More precisely, as in \cite[Section 4]{NOS}, we have the following result.

\begin{lemma}[C\'ea and truncation]
\label{lem:GalErr}
Let $\ue$ be the solution to problem \eqref{eq:ue-variational} 
and let 
$\ue_{h,M} = G_{h,M}^{q,\bmr} \ue$ its finite element approximation that solves \eqref{eq:alpha_weak_UhM}. 
Then we have 
\begin{equation}
\label{eq:GalErr}
\begin{aligned}
\| \nabla(\ue - \ue_{h,M}) \|_{L^2(y^{\alpha},\calC)} 
& \lesssim 
\min_{v_{h,M} \in \V_{h,M}}
\|  \nabla(\ue - v_{h,M}) \|_{L^2(y^{\alpha},\C_\Y)}
\\
& + 
\| \nabla \ue \|_{L^2(y^{\alpha},\calC\backslash \calC_\Y)}\;,
\end{aligned}
\end{equation}
where the hidden constant does not depend on $\V_{h,M}$.
\end{lemma}

As already noted in \cite[Prop.~{3.1}]{NOS}, 
the second term on the right hand side of \eqref{eq:GalErr} is exponentially small in $\Y$. 
More precisely, using \eqref{eq:reg_y_l} and \eqref{eq:reg_x}
we get, with the selection $\gamma < 2 \sqrt{\lambda_1}$, that
\begin{equation} \label{eq:ExpDec}
\| \nabla \ue \|_{L^2(y^\alpha, \C\setminus\C_\Y)} 
\lesssim 
\exp(-\gamma \Y/2) \| f \|_{{\mathbb H}^{-s}(\Omega)}. 
\end{equation}
%
\subsection{FE error splitting}
\label{S:ErrSplt}
As \eqref{eq:ExpDec} shows, the second term on 
the right hand side of of \eqref{eq:GalErr} decays exponentially in $\Y$. 
Thus, we now concentrate on estimating the first one. 

As in \cite{NOS,MPSV17}, we separate the errors incurred by discretizations 
with respect to $x'$ and $y$ as follows.
\begin{lemma}[dimensional error splitting]
\label{L:ErrSplt}
Let $\ue$ be the solution to problem \eqref{eq:ue-variational} and let $\ue_{h,M}$ denote 
its approximation defined as the solution to \eqref{eq:alpha_weak_UhM}.
Assume that on the sequence $\{\calT^\ell \}_{\ell \geq 1}$ of regular, 
simplicial triangulations of $\Omega$ the quasi-interpolation operator 
$\Pi^q_{x'}$ is concurrently uniformly stable on $L^2(\Omega)$ and $H^1(\Omega)$.
Let $\pi^{\bmr}_y:H^1(y^\alpha, (0,\Y)) \rightarrow S^\bmr_{\{\Y\}}((0,\Y),\calG^M)$ 
be a linear projector.
Then 
\begin{equation}\label{eq:ErrSplt}
  \begin{aligned}
  \min_{ v_{h,M} \in \V_{h,M}} \|  \nabla(\ue - v_{h,M}) \|_{L^2(y^{\alpha},\calC_\Y)} 
  &\lesssim 
  \| \nabla( \ue - \Pi^q_{x'} \ue ) \|_{L^2(y^{\alpha},\C_\Y)} \\
  &+ \|  \nabla(\ue - \pi^\bmr_y \ue) \|_{L^2(y^{\alpha},\C_\Y)} \;,
  \end{aligned}
\end{equation}
where the hidden constant does not depend on the dimension of $\V_{h,M}$.
\end{lemma}
\begin{proof}
The desired estimate follows from the tensor-product structure 
of the finite element space defined in \eqref{eq:TPFE} and the triangle inequality, 
upon choosing in \eqref{eq:ErrSplt} the function $v_{h,M} := \Pi^q_{x'} \otimes \pi^\bmr_y \ue$. 
\end{proof}
\subsection{$h$-FE error analysis}
\label{S:hFEErrAn}
In the present subsection
we analyze convergence rates and complexity for 
two particular instances of the FE-space 
$\V^{q,\bmr}_{h,M}(\calT,\calG^M)$:

\begin{enumerate}[(a)]
  \item The case when $\bmr = (1,1,\dots,1)$ on a graded mesh $\calG^M$ and $q=1$. 
A particular instance of this was first introduced in \cite{NOS}; see Section \ref{sec:apriori}. 
Generalizing the results of \cite{NOS,MPSV17}, we allow 
$\Omega \subset \bbR^2$ to be a polygon with finitely many straight sides and corners $\{ \bmc \}$.
This will mandate the use of a sequence of nested triangulations 
$\{ \calT^\ell \}_{\ell \geq 1}$ of the domain $\Omega$ 
with, in general, local refinement towards the corners $\bmc\in \partial\Omega$.
   
\item 
The case $\bmr = (1,1,\ldots,1)$ on a nested sequence 
$\{ \calG^{\ell'} \}_{\ell' \geq 1}$ of graded meshes in $(0,\Y)$. 
At the same time, we also consider multilevel
approximations in $\Omega$ on a sequence $\{ \calT^\ell \}_{\ell \geq 1}$ 
of nested triangulations with appropriate corner refinement in $\Omega$,
a particular instance being the so-called bisection--tree refinements.
\end{enumerate}
   
In all cases, we bound the first term on the right hand side of \eqref{eq:GalErr}.
\subsubsection{$P_1$-FEM in $\Omega$ with mesh refinement at $\bmc$}
\label{S:P1FEM}
In a bounded polygon $\Omega \subset \bbR^2$ with straight sides
and corners $\bmc$ we consider the Dirichlet problem 
\begin{equation}\label{eq:DirLw}
\calL w = g \mbox{ in } \Omega\;,\quad w = 0 \mbox{ on } \partial \Omega \;,
\end{equation}
for $g \in H^{-1}(\Omega)$.
It is immediate that problem \eqref{eq:DirLw} has a unique solution $w\in H^1_0(\Omega)$. 
However, in general the solution $w$ does not belong to $H^2(\Omega)$. 
Under additional regularity assumptions on $A$ and $c$, 
it rather belongs to weighted
Sobolev spaces of Kondrat'ev type in $\Omega$ which we now define.

For a finite set $\{ \bmc \}$ of corners of $\Omega$ 
and $x\in \Omega$ we define $\Phi(x) = \prod_{\bmc} |x-\bmc|$.
To follow standard notation, for $0 \leq \beta \in \R$, 
we set $L^2_\beta(\Omega)=L^2(\Phi^{2\beta},\Omega)$.
We also define the space $H^2_\beta(\Omega)$ as the closure of $H^2(\Omega) \cap H^1_0(\Omega)$ 
with respect to the norm
\begin{equation}\label{eq:H2}
\| w \|_{H^2_\beta(\Omega)} = \| w \|_{H^1(\Omega)} + \|  D^2 w \|_{L^2(\Phi^{2\beta},\Omega)} 
\;.
\end{equation}

With this setting at hand, we present the following 
result on regularity shift in weighted Sobolev spaces for the solution of problem \eqref{eq:DirLw}.
\begin{proposition}[weighted regularity estimate]
\label{prop:H2Regw}
Let $A \in W^{1,\infty}(\Omega,\GL(\R^2))$ be uniformly positive definite, 
$c\in W^{1,\infty}(\Omega,\R)$ and $g\in L^2_\beta(\Omega)$. 
Then, for every polygon $\Omega\subset {\mathbb R}^2$, 
there exists $\beta \geq 0$ such that 
the solution $w$ of \eqref{eq:DirLw} belongs to $H^2_\beta(\Omega)$ and 
\begin{equation}\label{eq:wgtapriori}
\| w \|_{H^2_\beta(\Omega)} \lesssim \| \calL w \|_{L^2_\beta(\Omega)} = \| g \|_{L^2_\beta(\Omega)}
\;,
\end{equation}
where the hidden constant is independent of $g$.
\end{proposition} 
\begin{proof}
This is result is a particular case of \cite[Theorem 1.1]{BacutaLiNistor2016}. 
It suffices to set, in the notation of this reference, $m=1$, $b_j=0$, and $\beta = 1 - a$.
\end{proof}

\begin{remark}[Laplacian]
\label{Rmk:SingLaplace}
In the special case that $\calL = -\Delta$, \ie when \eqref{eq:DirLw} corresponds to the Dirichlet 
Poisson problem in a polygon $\Omega$, the parameter $\beta$ must satisfy
$\beta > 1-\min_{\bmc} \pi/\omega_\bmc$, where 
$0<\omega_{\bmc} < 2\pi$ is the interior opening angle of $\Omega$ at the vertex $\bmc$.
If $\Omega$ is convex, the choice $\beta = 0$ is admissible,
and then \eqref{eq:wgtapriori} reduces to the classical 
regularity shift for the Dirichlet problem of the Poisson equation in convex domains.
We refer the reader to the discussion in \cite[equations (2) and (3)]{BacutaLiNistor2016} 
for more details.
\eremk
\end{remark}
 
Proposition~\ref{prop:H2Regw} and the regularity of $\ue$ given in Theorem \ref{thm:glob_reg_ue}
imply the following regularity result for $\ue$ in weighted norms in $\Omega$.
\begin{proposition}[global regularity of $\ue$: weighted estimates in $\Omega$]
\label{prop:wgtU}
Let $\ue \in \HL(y^{\alpha},\C)$ solve \eqref{alpha_harm_intro} with $s \in (0,1)$. 
Let $0 \leq \nu' < 1-s$. 
Assume that $\Omega\subset {\mathbb R}^2$ is a polygon and that 
$A$ and $c$ satisfy the assumptions of Proposition \ref{prop:H2Regw}.
Then there exists $\beta \geq 0$, 
which depends only on $\Omega$, $A$, and $c$, such that 
\begin{equation}
\label{eq:H2beta-a}
  \| \ue \|_{L^2(\omega_{\alpha-2\nu',\gamma},(0,\infty);H^2_\beta(\Omega))} 
  \lesssim 
  \| f \|_{\mathbb{H}^{1-s+\nu'}(\Omega)},
\end{equation}
where the weight $\omega_{\beta,\gamma}$ is defined as in \eqref{eq:weight}. 
In addition, for $\ell \in {\mathbb N}_0$, and $0 \leq \tilde \nu < 1+s$, 
there exists $\kappa >1$ such that
\begin{equation}
\label{eq:H2beta-b}
\| \partial_y^{\ell+1}\ue \|_{L^2(\omega_{\alpha+2(\ell+1)-2\tilde \nu ,\gamma},(0,\infty);H^2_\beta(\Omega))} 
\lesssim 
\kappa^{\ell+1}(\ell+1)! 
\| f \|_{\mathbb{H}^{1-s+\tilde \nu }(\Omega)}\;.
\end{equation} 
In both estimates, the hidden constants are independent of $\ue$ and $f$.
\end{proposition}
\begin{proof}
The proof for \eqref{eq:H2beta-b} follows from \eqref{eq:reg_y_l} 
and that of \eqref{eq:H2beta-a} from \eqref{eq:reg_L2}
by using the the weighted regularity shift \eqref{eq:wgtapriori}. 
In fact, for a fixed $y>0$ and $m\in \bbN_0$, set $w =  \partial_y^{m} \ue(\cdot,y)$ in \eqref{eq:DirLw}. 
Notice that $g = \partial_y^{m} \calL_{x'} \ue(\cdot,y)$. Since $\beta \geq 0$ 
we have that $g \in L^2_\beta(\Omega)$ and estimate \eqref{eq:wgtapriori} holds. 
Square it and multiply it by either $\omega_{\alpha-2\nu',\gamma}$ if $m=0$, 
or $\omega_{\alpha+2m-2\nu,\gamma}$ when $m\geq 1$.
Integration with respect to $y$ over $(0,\infty)$ allows us then to conclude.
\end{proof}

The previous regularity result will be the basis for the analysis of a
$P_1$-FEM on properly refined meshes in $\Omega$ and it 
will allow us to recover the full first order convergence rate; 
see Theorem \ref{thm:P1Graded} below.

To accomplish this task, 
we associate with $H^2_\beta(\Omega)$ a sequence $\{ \calT^\ell_\beta \}_{\ell \geq 0}$
of bisection--tree meshes in $\Omega$ which, as constructed in 
\cite{GspMrin_IMAJNA2009}, are properly refined towards
the corners $\{\bmc\}$ of $\Omega$.
Bisection--tree meshes are uniformly shape regular 
(see, e.g., \cite[Lemma 1]{NochVees12:PrimAFEM}) 
and, as shown in \cite{GspzHeineSiebert2016}, the $L^2$-projections 
$\Pi_\beta^\ell := \Pi_{x'}^1: L^2(\Omega) \to S^1_0(\Omega,\calT^\ell_\beta)$
are uniformly stable in $L^2(\Omega)$ and also in $H^1(\Omega)$.
In addition, they satisfy optimal asymptotic error bounds, \ie for every $\ell\geq 0$ and every 
$w \in H^1_0(\Omega)$ we have
\begin{equation}\label{eq:Piellbeta-a}
N_\ell \| w - \Pi^\ell_\beta w \|^2_{L^2(\Omega)}  \lesssim \|w\|^2_{H^1(\Omega)},
\end{equation}
where $N_\ell = \dim S^1_0(\Omega,\calT^\ell_\beta) = \calO(h_\ell^{-2})$. 
In addition, for every $w\in H^2_\beta(\Omega)$, there holds
\begin{equation}\label{eq:Piellbeta}
N_\ell \| w - \Pi^\ell_\beta w \|^2_{L^2(\Omega)} 
+
\| \nabla_{x'} (w - \Pi^\ell_\beta w) \|^2_{L^2(\Omega)} 
\lesssim N_\ell^{-1} \| w \|^2_{H^2_\beta(\Omega)} 
\;.
\end{equation}
%
In view of the embedding $H^2_\beta(\Omega) \hookrightarrow C^0(\bar\Omega)$, the nodal interpolant is well--defined and \cite[Section 5]{MSS17_711} shows that \eqref{eq:Piellbeta} holds for such an interpolant.
We now use that $\Pi_\beta^\ell$ reproduces the discrete space $S^1_0(\Omega,\calT^\ell_\beta)$ and, owing to \cite{GspzHeineSiebert2016},
that it is bounded uniformly with respect to $\ell$ concurrently 
in $L^2(\Omega)$ and in $H^1(\Omega)$ to conclude \eqref{eq:Piellbeta}.

%
\begin{remark}[other quasi-interpolants]
\label{Rmk:xPrjction}
The $L^2$-projection in the previous argument can also be replaced with 
Scott-Zhang type quasi-interpolants that are projections onto 
$S^1_0(\Omega,\calT^\ell_\beta)$ and have suitable 
local stability properties in both $L^2$ and $H^1$. 
Such operators are constructed, e.g., in 
\cite[Lemma~{4}]{aurada-feischl-fuehrer-karkulik-praetorius15} 
by dropping in  the classical Scott-Zhang operator \cite{SZ:90} 
the degrees of freedom associated with nodes on $\partial\Omega$ and noting
that the remaing operator is well-defined and 
(locally) stable in $L^2(\Omega)$. 
\eremk
\end{remark}
%
\subsubsection{Linear interpolant $\pi^1_\eta$ on radical-geometric meshes in $[0,\Y]$}
\label{S:p1y}
%
%
To approximate the solution $\ue$ with respect to the extended variable $y$, 
we shall use a continuous, piecewise linear interpolant on suitably refined 
meshes $\calG^{k}_{gr,\eta}$ in $[0,\Y]$. 
The mesh is 
radical 
on $[0,1]$ 
and geometric 
on $[1,\Y]$, and the parameter $k$ indicates the mesh size 
near the point $1$. 
Specifically, for $\Y > 1$, $\eta > 0$ and $k = 1/N$ for an integer $N \in {\mathbb N}$, 
the mesh $\calG^{k}_{gr,\eta}$ is given by
\begin{subequations}
\label{eq:radical-geometric-mesh}
\begin{align}
\calG^{k}_{gr,\eta} &:= \{ I_{i} \,|\, i=1,\ldots,N\} \cup \{J_{j} \,|\, j=1,\ldots,N'\}, 
\\
\label{eq:radical_mesh}
I_{i} &= \bigl[ ((i-1) k)^\eta,(ik)^\eta\bigr],  \qquad i=1,\ldots,N, \\
J_{j} &= \bigl[\exp((j-1)k), \exp(jk)\bigr], \qquad j=1,\ldots,N'-1:= \lfloor N \log \Y\rfloor - 1,  
\\
J_{N'}& = \bigl[\exp((N'-1)k),\Y\bigr]. 
\end{align}
\end{subequations} 
Given $\eta$ and $\Y$, 
we denote by $\pi^1_{\eta}:C((0,\Y]) \rightarrow S^1((0,\Y),\calG^{k}_{gr,\eta})$ 
the piecewise linear interpolation operator over all the elements of the mesh $\calG^{k}_{gr,\eta}$
with the exception of the first one, 
\ie $I_1$. 
On that element, $\pi^1_{\eta}$ corresponds to the linear interpolant in the midpoint of $I_{1}$ 
and the right endpoint of $I_{1}$. 
The operator
\[
\pi^1_{\eta,\{\Y\}}:C((0,\Y])\rightarrow S^1_{\{\Y\}} \left( (0,\Y),\calG^{k}_{gr,\eta} \right) 
\]
is obtained from $\pi^1_\eta$ by subtracting a linear function on the 
element abutting at $\Y$ so as to satisfy $(\pi^1_{\eta,\{\Y\}} u)(\Y) = 0$. 
These operators naturally extend to Hilbert space valued functions.
The approximation properties of these operators are as follows.

\begin{lemma}[interpolation error estimates]
\label{lemma:graded-exponential-mesh} 
Let $X$ be a Hilbert space, 
$\alpha \in (-1,1)$, $\theta \in (0,1]$, and $0 \leq \gamma' < \gamma$.  
Let the mesh grading parameter $\eta$ that defines the mesh 
$\calG^{k}_{gr,\eta}$ satisfy $\eta \theta  \ge 1$. 
In this setting the following assertions hold.
\begin{enumerate}[(i)]
\item 
\label{item:lemma:graded-exponential-mesh-0} 
The number of elements in $\calG^{k}_{gr,\eta}$ is bounded by 
$k^{-1} (1 + \log \Y)$.  
\item 
\label{item:lemma:graded-exponential-mesh-i} 
For every  $u \in C((0,\Y];X)$ with 
$ u^\prime \in L^2(\omega_{\alpha+2(1-\theta),\gamma},(0,\Y);X)$ 
we have
\begin{align}
\label{eq:lemma:GEM-2}
\|u - \pi^1_{\eta} u\|_{L^2(\omega_{\alpha,\gamma'},(0,\Y);X)} 
& \lesssim 
k \|u^\prime \|_{L^2({\omega}_{\alpha+2(1-\theta),\gamma},(0,\Y);X)}, 
\\
\label{eq:lemma:GEM-4}
\|u - \pi^1_{\eta,\{\Y\}} u\|_{L^2({\omega}_{\alpha,\gamma'},(0,\Y);X)} 
& \lesssim
 k  \|u^\prime \|_{L^2({\omega}_{\alpha+2(1-\theta),\gamma},(0,\Y);X)} 
\\ \nonumber
& 
+ \sqrt{\Y k} \Y^\alpha \exp(\Y \gamma'/2) \|u(\Y)\|_X . 
\end{align}
Furthermore, under the assumption that $\lim_{y \rightarrow \infty} u(y) = 0$ in $X$ 
and the constraint 
\begin{equation}
\label{eq:lemma:GEM-10}
\Y^{1/2+\alpha} \exp(-\Y \gamma/2) \leq k^{1/2}
\end{equation}
the following estimate holds: 
\begin{align}
\label{eq:lemma:GEM-20}
\|u - \pi^1_{\eta,\{\Y\}} u\|_{L^2(y^\alpha,(0,\Y);X)} 
& \lesssim k  \|u^\prime \|_{L^2(\omega_{\alpha+2(1-\theta),\gamma},(0,\Y);X)} . 
\end{align}
\item 
\label{item:lemma:graded-exponential-mesh-ii} 
For $u \in C((0,\Y];X)$ with $u'' \in L^2(\omega_{\alpha+2(1-\theta),\gamma},(0,\Y);X)$ 
and $j \in \{0,1\}$
\begin{align}
\label{eq:lemma:GEM-25}
\|(u - \pi^1_{\eta} u)^{(j)}\|_{L^2(\omega_{\alpha,\gamma'},(0,\Y);X)} 
& \lesssim 
 k^{2-j} 
\|u^{\prime\prime} \|_{L^2(\omega_{\alpha+2(1-\theta),\gamma}(0,\Y);X)}, 
\\
\label{eq:lemma:GEM-27}
\|(u - \pi^1_{\eta,\{\Y\}} u)^{(j)}\|_{L^2(\omega_{\alpha,\gamma'},(0,\Y);X)} 
& \lesssim
 k^{2-j}  \|u^{\prime\prime} \|_{L^2(\omega_{\alpha+2(1-\theta),\gamma}(0,\Y);X)}
\\ \nonumber
& \qquad 
+ (\Y k)^{1/2-j} \Y^\alpha \exp(\Y \gamma'/2) \|u(\Y)\|_X. 
\end{align}
Furthermore, under the assumption that, for $j \in \{0,1\}$, $\lim_{y \rightarrow \infty} u^{(j)}(y)=0$ in $X$,  
and the constraint 
\begin{equation}
\label{eq:lemma:GEM-30}
\Y^{1/2+2\alpha} \exp(-\Y \gamma/2) \leq k^{2}
\end{equation}
the following estimate holds for $j \in \{0,1\}$:
\begin{align}
\label{eq:lemma:GEM-40}
\|(u - \pi^1_{\eta,\{\Y\}} u)^{(j)}\|_{L^2(y^\alpha,(0,\Y);X)} 
& \lesssim k^{2-j}  \|u^{\prime\prime} \|_{L^2(\omega_{\alpha+2(1-\theta),\gamma},(0,\Y);X)} . 
\end{align}
\end{enumerate}
\end{lemma}
\begin{proof}
We present the details for the proof of \eqref{item:lemma:graded-exponential-mesh-i}, 
as that of \eqref{item:lemma:graded-exponential-mesh-ii} is similar. 
The technique used to obtain interpolation error estimates on the radical mesh on $[0,1]$ 
is well-established; see, for instance, \cite[Example~{3.47}]{phpSchwab1998}. 
We introduce
the mesh points $y_i:= (ik)^\eta$, $i=0,\ldots,N$ so that  $I_{i} = [y_{i-1},y_i]$.
 
For the first element $I_{1} = [y_0,y_1] = [0,k^\eta]$, 
we invoke the estimate \eqref{eq:lemma:I1-10} with the choice 
$\delta = 1 - \theta \in [0,1)$ and a scaling argument to conclude that
%
\begin{equation} \label{eq:lemma:graded-exponential-mesh-50}
\|u - \pi^1_{\eta} u\|^2_{L^2(y^\alpha,I_1;X)} 
\lesssim 
k_1^{2\theta} \|u^\prime\|^2_{L^2(y^{\alpha+2(1-\theta)},I_{1};X)}, 
\end{equation}
where $k_1= |I_1|$; we recall that $\theta \in (0,1]$.

Over the remaining elements $I_{i} $, $i=2,\ldots,N$, of $[0,1]$, 
we use that $
k_i \lesssim k y_{i-1}^{(\eta-1)/\eta},
$
where 
$k_i = |I_{i}| = y_i - y_{i-1} $ 
and $\eta$ defines the radical mesh on $[0,1]$ as in \eqref{eq:radical_mesh}.
We thus recall the standard interpolation estimate 
$$
\|u - \pi^1_{\eta} u\|^2_{L^2(I_{i})} \lesssim k_i^2 \|u^\prime\|^2_{L^2(I_{i})}
$$
and obtain, upon using that 
$\max_{y \in I_{i}} y^\alpha \lesssim \min_{y \in I_{i}} y^\alpha$ and tensorization with $X$, 
the bound
\begin{align}
\nonumber 
\|u - \pi^1_{\eta} u\|^2_{L^2(y^\alpha,I_{i};X)} 
&\lesssim k_i^2 \|u^\prime\|^2_{L^2(y^\alpha,I_{i};X)} 
\lesssim 
k^2 y_{i-1}^{2(\eta-1)/\eta} \|u^\prime\|^2_{L^2(y^\alpha,I_{i};X)} 
\\
\label{eq:lemma:graded-exponential-mesh-100}
& \lesssim 
k^2  \|u^\prime\|^2_{L^2(y^{\alpha+2(\eta-1)/\eta},I_{i,};X)} 
\lesssim k^2  \|u^\prime\|^2_{L^2(y^{\alpha+2(1-\theta)},I_{i};X)}
\;.
\end{align}
The last relation holds because $\eta\theta \ge 1$.

For the elements beyond $y=1$, we begin by setting, for $j=1,\ldots,N'$,
$J_{j} := [\widetilde y_{j-1},\widetilde y_j] = \bigl[\exp((j-1)k, \exp(jk)\bigr]$.
Let us now notice that, since $k \leq 1$,
\begin{equation}
\label{eq:lemma:graded-exponential-mesh-200}
|J_{j}|  = \exp((j-1) k) (1-e^k) \sim \widetilde y_{j-1} k, \qquad j=1,\ldots,N'-1\;. 
\end{equation}
Using that the weight functions $\omega_{\alpha,\gamma'}$ and $\omega_{\alpha,\gamma}$, 
defined as in \eqref{eq:weight}, are slowly varying over the intervals $J_{j}$, 
\ie
\begin{equation}
\label{eq:slowly:varying}
\max_{y \in J_{j}} \omega_{\alpha,\gamma'}(x)  \lesssim \min_{y \in J_{j}} \omega_{\alpha,\gamma'}(x)
\quad \textrm{and} \quad
\max_{y \in J_{j}} \omega_{\alpha,\gamma}(x)  \lesssim \min_{y \in J_{j}} \omega_{\alpha,\gamma}(x),
\end{equation}
 we obtain
\begin{align*}
  \sum_{j} \|u - \pi^1_{\eta} u\|^2_{L^2(\omega_{\alpha,\gamma'},J_{j};X)} 
    &\lesssim 
  \sum_{j} |J_{j}|^2 \|u^\prime \|^2_{L^2(\omega_{\alpha,\gamma'}, J_{j};X)}  \\
    &\lesssim
  k^2 \sum_{j} \widetilde y_{j-1}^2 e^{-(\gamma - \gamma') \widetilde y_{j-1}} 
  \|u^\prime \|^2_{L^2(\omega_{\alpha,\gamma},J_{j};X)},
\end{align*}
where in the last step we used \eqref{eq:lemma:graded-exponential-mesh-200}. 
Using now that 
$\widetilde y_{j-1}^2 e^{-(\gamma - \gamma') \widetilde y_{j-1}} \lesssim 1$ 
and \eqref{eq:slowly:varying}, again, we finally arrive at
\begin{equation}
  \sum_{j} \|u - \pi^1_{\eta} u\|^2_{L^2(\omega_{\alpha,\gamma'},J_{j};X)} 
\lesssim
k^2 \sum_{j} 
\|u^\prime \|^2_{L^2(\omega_{\alpha,\gamma},J_{j};X)}.
\label{eq:lemma:graded-exponential-mesh-300}
\end{equation}
Combining 
\eqref{eq:lemma:graded-exponential-mesh-50}, 
\eqref{eq:lemma:graded-exponential-mesh-100}, and \eqref{eq:lemma:graded-exponential-mesh-300} 
finishes the proof of the approximation  properties of $\pi^1_{\eta}$. 
The correction on the 
last element to obtain \eqref{eq:lemma:GEM-4} for the operator $\pi^1_{\eta,\{\Y\}}$ 
is straight forward in view of \eqref{eq:lemma:graded-exponential-mesh-200}. 
The estimate \eqref{eq:lemma:GEM-20} follows from \eqref{eq:lemma:GEM-4} by controlling
$\|u(\Y)\|_X$ with the aid of Lemma~\ref{lemma:inftyExp}. 

The proof of (\ref{item:lemma:graded-exponential-mesh-ii}) follows along similar lines. 
\end{proof}

It is worth stressing that the choices $k = 2^{-\ell}$ lead to nested meshes.

\begin{corollary}[nested meshes]
\label{cor:nested-graded-meshes}
For every fixed $\eta \ge 0$, $\Y \ge 1$ and 
for $k_\ell = 2^{-\ell}$, the sequence
$\{ \calG^{k_\ell}_{gr,\eta}\}_{\ell=0}^\infty$ of graded meshes in 
$(0,\Y)$ is nested and has $\calO(2^\ell (1 + \log \Y))$ elements. 
\end{corollary}
\begin{proof}
For fixed $\Y>0$, 
it follows directly from the definition of the mesh points \eqref{eq:radical-geometric-mesh}, 
in terms of $k$, that the meshes are nested.
\end{proof}
\subsubsection{Tensor $P_1$-FEM in $\calC$ with corner mesh refinement in $\Omega$}
\label{S:TP1FEM}
%
We now provide a convergence estimate in refined meshes over, not necessarily convex, polygons.

\begin{theorem}[error estimates]
\label{thm:P1Graded}
Let $u\in \mathbb{H}^s(\Omega)$ and $\ue\in\HL(y^\alpha,\C)$ solve \eqref{fl=f_bdddom} and 
\eqref{alpha_harm_intro}, respectively, 
with $f\in \mathbb{H}^{1-s}(\Omega)$ and $\Omega\subset {\mathbb R}^2$ 
a bounded polygon with straight sides and (a finite set of) corners $\{ \bmc \}$.
Let $\beta \ge 0$ be such that \eqref{eq:wgtapriori} holds and 
let $\{\calT^\ell_\beta\}_{\ell}$ be a sequence of graded meshes 
that satisfy \eqref{eq:Piellbeta-a} and \eqref{eq:Piellbeta}. 
Let $\calG^{k}_{gr,\eta}$ be the graded--exponential mesh 
of \eqref{eq:radical-geometric-mesh} with $\eta$ chosen to satisfy $\eta s > 1$, 
$k = 1/N$ with $N \in {\mathbb N}$
chosen so that $2h_\ell^{-1} \geq 1/k \geq h_\ell^{-1}$,
and with the cut-off $\Y>0$ chosen as 
\begin{equation}\label{eq:Ylogh}
\Y \sim |\log h_\ell| \;.
\end{equation}
Denote by $\ue_{h_\ell,M}$ the solution of \eqref{eq:alpha_weak_UhM} 
over the space $\V_{h,M}^{1,1}(\calT^{\ell}_\beta, \calG^{k}_{gr,\eta})$. 
In this setting we have the following error estimate
\begin{equation}\label{eq:FTconv}
\| u - \tr \ue_{h,M} \|_{\Hs} \lesssim \|\nabla (\ue - \ue_{h,M})\|_{L^2(y^\alpha,\C)} 
\lesssim h_\ell \| f \|_{\Ws}. 
\end{equation}
%
In addition, the total number of degrees of freedom behaves like
\begin{equation}\label{eq:NFT}
\calN_{\Omega,\Y} :=
\dim \V_{h,M}^{1,1}(\calT^{\ell}_\beta, \calG^{k}_{gr,\eta})
= \calO(h_\ell^{-3} \log |\log h_\ell|)
= \calO(\calN_\Omega^{1+1/2}\log \log \calN_\Omega) 
\;,
\end{equation}
where $\calN_\Omega = \# \calT_\beta^\ell$.
\end{theorem}

Before proving Theorem~\ref{thm:P1Graded}, 
we note a corollary that follows from a simple interpolation argument.

\begin{corollary}[reduced regularity]
\label{cor:P1Graded}
Assume that the meshes are constructed as in Theorem~\ref{thm:P1Graded} and that
$f \in {\mathbb H}^{-s + \sigma}(\Omega)$, with $\sigma \in [0,1]$. 
Then we have
\begin{equation}\label{eq:FTconv-cor}
\| u - \tr \ue_{h,M} \|_{\mathbb{H}^s(\Omega)} \lesssim \|\nabla (\ue - \ue_{h,M})\|_{L^2(y^\alpha,\C)} 
\lesssim 
h_\ell^\sigma \| f \|_{\mathbb{H}^{\sigma-s}(\Omega)},
\end{equation}
where the hidden constant also depends on $\sigma$.
\end{corollary}

The proof of Theorem~\ref{thm:P1Graded} follows similar arguments to 
\cite{NOS} and \cite[Section 4.1]{MPSV17} and uses the stability and 
approximation properties \eqref{eq:Piellbeta} of $\Pi_\ell^\beta$. 
For completeness we provide the details.

\begin{numberedproof}{of Theorem~\ref{thm:P1Graded}}
For the given choice of $k$, $\eta$ and $\Y$,
we denote by $\pi^{1,\ell}_{\eta,\{\Y\}}$ the nodal interpolation operator
on the mesh \eqref{eq:radical-geometric-mesh},
which we analyzed in Lemma~\ref{lemma:graded-exponential-mesh}.

By Lemmas~\ref{lem:GalErr} and \ref{L:ErrSplt}, 
and by the choice \eqref{eq:Ylogh} (recall \eqref{eq:ExpDec})
 it suffices to bound 
\[
\|  \nabla(\ue - \pi^{1,\ell}_{\eta,\{ \Y \}} \ue) \|_{L^2(y^{\alpha}, \calC_{\Y})}
+ 
\| \nabla(\ue - \Pi^\ell_\beta \ue) \|_{L^2(y^{\alpha},\calC_{\Y})}
=: I + II \;.
\]
Recalling that $\nabla = (\nabla_{x'}, \partial_y)$ 
we split the first term $I$ into
\[
  I \lesssim 
  \|  \partial_y ( \ue - \pi^{1,\ell}_{\eta,\{ \Y \}} \ue ) \|_{L^2(y^{\alpha},\calC_{\Y})}
+
\| \nabla_{x'}(\ue - \pi^{1,\ell}_{\eta,\{ \Y \}} \ue) \|_{L^2(y^{\alpha},\calC_{\Y})} 
  = : I_a + I_b
\;.
\]
In view of \eqref{eq:Ylogh}, 
we immediately obtain that the conditions \eqref{eq:lemma:GEM-10} and \eqref{eq:lemma:GEM-30} 
of Lemma \ref{lemma:graded-exponential-mesh} are satisfied. 
We can thus, since $\eta s >1$, bound the term $I_a$ using 
Lemma~\ref{lemma:graded-exponential-mesh}, item (\ref{item:lemma:graded-exponential-mesh-ii}), with $j=1$
and $X=L^2(\Omega)$  and the term $I_b$ using 
Lemma~\ref{lemma:graded-exponential-mesh}, 
item (\ref{item:lemma:graded-exponential-mesh-i}) with $X=H^1_0(\Omega)$. 
We have thus arrived at
\begin{align*}
I \lesssim I_a + I_b \lesssim h_\ell \|f\|_{{\mathbb H}^0(\Omega)},
\end{align*}
where we have also used the regularity estimates 
of Theorem \ref{thm:glob_reg_ue}.

We apply the same splitting to the term $II$ to arrive at
\[
II  \lesssim 
\|  \partial_y ( \ue - \Pi^{\ell}_{\beta} \ue ) \|_{L^2(y^{\alpha},\calC_{\Y})}
+
\| \nabla_{x'}(\ue - \Pi^{\ell}_{\beta} \ue) \|_{L^2(y^{\alpha},\calC_{\Y})}
 = : II_a + II_b
\;.
\]
Since $N_\ell = \calO(h_\ell^{-2})$ we have, from \eqref{eq:Piellbeta-a}, that
\[
II_a  = \displaystyle
\| \partial_y \ue -  \Pi^{\ell}_{\beta} (\partial_y \ue) \|_{L^2(y^\alpha, \calC_{\Y})}
\lesssim 
h_\ell \| \partial_y \ue \|_{L^2(\omega_{\alpha,0},(0,\Y);H^1(\Omega))} 
\;.
\]
To bound $II_b$ we use \eqref{eq:Piellbeta} and obtain
\[
  II_b \lesssim h_\ell \| \ue \|_{L^2(\omega_{\alpha,0},(0,\Y);H^2_\beta(\Omega))}
\;.
\]
Using the regularity estimate \eqref{eq:H2beta-a} with $\nu'=0$ 
we conclude the proof of \eqref{eq:FTconv}. 

To obtain \eqref{eq:NFT}, we first note that by 
Lemma~\ref{lemma:graded-exponential-mesh} 
item \eqref{item:lemma:graded-exponential-mesh-0}, 
the number of elements in $\calG^{k}_{gr,\eta}$ with $h_\ell=2^{-\ell}$ 
and with the choice $\Y \simeq |\log h_\ell| \simeq \ell$ 
is $\calO(2^\ell \log \ell)$. 
We finally observe that the total 
number of degrees of freedom in the tensor product space
is the product of the dimensions of the component spaces, i.e., 
$\calO(h_\ell^{-2} h_\ell^{-1} \log |\log h_\ell|)$.
\end{numberedproof}

%
%
\subsubsection{Sparse grid $P_1$-FEM with corner mesh refinement}
\label{S:sGP1FEM}
%
The convergence order \eqref{eq:FTconv} is optimal, however, the 
complexity of the method implied by \eqref{eq:NFT} is superlinear with respect to 
the number of degrees of freedom in $\Omega$, $\calN_\Omega$.

To reduce the complexity to nearly linear, in what follows we develop
a sparse tensor product approach.
It is based on the subspace hierarchies
\[
\{ S^1_0(\Omega, \calT^\ell_\beta) \}_{\ell \geq 1}\;,\quad 
\{ S^1_{\{\Y\}}((0,\Y), \calG^{2^{-\ell'}}_{gr,\eta}) \}_{\ell' \geq 1},
\]
where $\{ \calT^\ell_\beta \}_{\ell \geq 1}$ is the nested sequence of bisection--tree meshes in $\Omega$ which are $\beta$-graded toward the corners $\{\bmc\}$ 
in such a way that first-order convergence in $h_\ell = \calO(2^{-\ell})$ is achieved; 
the sequence $\{ \calG^{2^{-\ell'}}_{gr,\eta} \}_{\ell'\geq 1}$ 
consists of nested graded meshes on $[0,\Y]$
that achieve, 
for functions belonging to weighted $H^2$-spaces in $(0,\Y)$, 
as introduced in Theorem \ref{thm:glob_reg_ue},
first order convergence
(cf. the precise statements in~Lemma~\ref{lemma:graded-exponential-mesh}
 and in Corollary~\ref{cor:nested-graded-meshes}).

For $\ell,\ell'\geq 0$, we denote by 
\[
\Pi^\ell_\beta:L^2(\Omega) \to S^1_0(\Omega, \calT^\ell_\beta)
\mbox{ and }
\pi^{1,\ell'}_{\eta,\{\Y\}}: 
   C_{}((0,\Y]) \to S^1_{\{\Y\}}((0,\Y), \calG^{2^{-\ell'}}_{gr,\eta})
\]
the corresponding (quasi)interpolatory projections introduced in Section \ref{S:NtFESpc}.
Define in addition $\Pi^{-1}_\beta := 0$ and $\pi^{1,-1}_{\eta,\{\Y\}} := 0$.
Then, for $L\in {\mathbb N}_0$, we define the \emph{sparse tensor product space} 
as 
\begin{equation}\label{eq:DefsGV}
\hat{\V}^{1,1}_L(\C_\Y)
= 
\sum_{{\ell,\ell'\geq 0},{\ell+\ell'\leq L}}
S^1_0(\Omega, \calT^\ell_\beta) \otimes  S^1_{\{\Y\}}((0,\Y), \calG^{2^{-\ell'}}_{gr,\eta})
\;.
\end{equation}
We immediately comment that the sum in \eqref{eq:DefsGV} is not direct, and 
by zero extension we, evidently, have $\hat{\V}^{1,1}_L(\C_\Y) \subset \HL(y^\alpha, \C)$. 

We define the approximation $\hat{\ue}_L \in \hat{\V}^{1,1}_L(\C_\Y)$ as 
the solution to \eqref{eq:alpha_weak_UhM} with $\hat{\V}^{1,1}_L(\C_\Y)$ taking the role
of $\V_{h,M}$ there. 

\begin{remark}[implementation]
\label{rmk:CmbFrml}
The computation of the sparse tensor FE approximation
$\hat{\ue}_L \in \hat{\V}^{1,1}_L(\C_\Y)$
by directly evaluating 
\eqref{eq:alpha_weak_UhM} 
would require an explicit representation of the 
sparse tensor product subspace $\hat{\V}^{1,1}_L(\C_\Y)$
and therefore, in particular, an explicit basis
for the ``increment spaces'' in \eqref{eq:DefsGV}, 
i.e., for the complements of 
$
S^1_0(\Omega,\calT^{\ell-1}_\beta)$ in $ 
S^1_0(\Omega, \calT^\ell_\beta )$ 
and the complements of 
   $S^1_{\{\Y\}}((0,\Y), \calG^{2^{-{(\ell'-1)}}}_{gr,\eta})$ in  
$S^1_{\{\Y\}}((0,\Y), \calG^{2^{-\ell'}}_{gr,\eta})$.  
Construction of bases for the increment spaces 
is possible, based on ideas from multiresolution analyses. 
We opt, instead, to compute 
$\hat{\ue}_L \in \hat{\V}^{1,1}_L(\C_\Y)$
from the so-called combination formula 
(see, e.g., \cite[Section 4.2, Equation (4.6)]{CombFrml_HPS_2013}).
It is based on anisotropic 
$\HL(y^{\alpha},\C)$-Galerkin projections
\begin{equation}\label{eq:TPAnsio}
G^{1,1}_{\ell,\ell'} := G^1_{h_\ell} \circ G^1_{2^{-\ell'}}:
\HL(y^\alpha;\C)
\to 
{\V}^{1,1}_{h,M}(\calT_{\ell},\calG^{2^{-\ell'}}_{gr,\eta})\;,
\end{equation}
with the semidiscrete projections in \eqref{eq:TPOp}.
The projectors $G^{1,1}_{\ell,\ell'}$ in \eqref{eq:TPAnsio} can be realized with
standard FE bases in $\Omega$ and in $(0,\Y)$.
The combination formula then takes the following form: 
denoting by $\ue_{\ell,\ell'}:= G^{1,1}_{\ell,\ell'}\ue$, there holds, 
with the understanding that $\ue_{-1,j} = 0$ for $j \in {\mathbb N}_0$,
\[
\begin{array}{rcl}
\hat{\ue}_L & = & \displaystyle
\sum_{\ell = 0}^L \left(\ue_{\ell,L-\ell} - \ue_{\ell-1,L-\ell}\right). 
\end{array}
\]
\end{remark}

The convergence of our sparse grids scheme is the content of the next result.

\begin{theorem}[convergence for sparse grids]
\label{thm:sparse} 
Let $\beta \ge 0$ be such that \eqref{eq:wgtapriori} holds. 
Let $1 < \nu < 1+s$. 
Let $\eta (\nu-1) \ge 1$. 
Select $\Y \sim |\log h_L|$ with a sufficiently large implied constant. 
Let $f \in {\mathbb H}^{-s+\nu}(\Omega)$. 
Then the sparse tensor product space $\hat{\V}^{1,1}_L(\C_\Y)$ of \eqref{eq:DefsGV} 
and the corresponding Galerkin approximation
$\hat \ue_{L} \in \hat{\V}^{1,1}_L(\C_\Y)$ to $\ue$ 
satisfy 
\begin{align}
\label{eq:thm:sparse-10} 
\|\nabla( \ue - \hat \ue_{L})\|_{L^2(y^\alpha, \C)} 
&\lesssim h_L |\log h_L| \|f\|_{{\mathbb H}^{-s+\nu}(\Omega)}, 
\\
\label{eq:thm:sparse-20} 
\operatorname*{dim} \hat{\V}^{1,1}_L(\C_\Y) 
&\lesssim \calN_\Omega \log \log \calN_\Omega. 
\end{align}
\end{theorem}
\begin{proof}
We begin by proving \eqref{eq:thm:sparse-20}.
From the condition $\Y \sim |\log h_L| \sim L$, 
we have, by Lemma~\ref{lemma:graded-exponential-mesh}, 
item (\ref{item:lemma:graded-exponential-mesh-0}),
that 
$\#(\calG^{2^{-\ell'}}_{gr,\eta}) 
  \lesssim 2^{\ell'} | \log h_L|  \sim 2^{\ell'} \log L$. 
Consequently,
\begin{equation}\label{eq:DimhatVL}
\dim \hat{\V}^{1,1}_L(\C_\Y) 
\lesssim 
\sum_{{\ell,\ell'\geq 0},{\ell+\ell'\leq L}}
2^{2\ell+\ell'} |\log L| 
\lesssim 
2^{2L} \log L 
\sim 
\calN_\Omega \log \log \calN_\Omega
\;,
\end{equation}
where we have also used that $N_\ell = {\rm dim} (S^1_0(\Omega,\calT^\ell_\beta)) \sim 2^{2\ell}$.

We now study the error of our method. 
From Lemma~\ref{lem:GalErr} and \eqref{eq:ExpDec} 
it suffices to study the best approximation error in $\V^{1,1}_L(\C_\Y)$. 
To do so,
%
we introduce the \emph{sparse tensor product interpolation} projector
\[
  \hat{\Pi}^L_\Y: C((0,\Y]; L^2(\Omega)) \to \hat{\V}^{1,1}_L(\C_\Y)
\]
which is defined by
%
\begin{equation}\label{eq:SPTInterp}
\hat{\Pi}^L_\Y w
:= 
\sum_{{\ell,\ell'\geq 0},{\ell+\ell'\leq L}}
(\Pi^\ell_\beta - \Pi^{\ell-1}_\beta) \otimes (\pi^{1,\ell'}_{\eta,\{\Y\}} - \pi^{1,\ell'-1}_{\eta,\{\Y\}}) w 
\;.
\end{equation}
%

We can now, as in the proof of Theorem \ref{thm:P1Graded}, split the error into
\begin{equation}
\label{eq:tensor-10}
  \begin{aligned}
  \min_{\hat{v}_L \in \hat{\V}_L} \|  \nabla(\ue - \hat{v}_L) \|^2_{L^2(y^{\alpha},\calC_\Y)}
    &\lesssim
  \|  \partial_y (\ue - \hat{\Pi}^L_\Y \ue) \|^2_{L^2(y^{\alpha},\calC_\Y)} \\
  &+
  \|  \nabla_{x'}(\ue - \hat{\Pi}^L_\Y \ue) \|^2_{L^2(y^{\alpha},\calC_\Y)} 
  =:
  I + II \;.
  \end{aligned}
\end{equation}
Each one of these terms can now 
be bounded in the usual sparse grid fashion,
\emph{provided} that $\ue$ has so-called mixed regularity.
To do this we introduce the operators 
\[
Q^\ell_\beta := \Pi^\ell_\beta - \Pi^{\ell-1}_\beta\;,
\quad 
q^{1,\ell'}_\eta:= \pi^{1,\ell'}_{\eta,\{\Y\}} - \pi^{1,\ell'-1}_{\eta,\{\Y\}}\;.
\]

Let us bound term $I$ in \eqref{eq:tensor-10}. 
From the estimate \eqref{eq:reg_x_grad_mu} of 
Theorem~\ref{thm:glob_reg_ue} we infer 
\begin{align}
\label{eq:sparse-10} 
  \|\partial_y^2 \ue \|_{L^2(\omega_{\alpha+2(2-\nu),\gamma},(0,\infty);H^1(\Omega))} &\lesssim 
  \|f\|_{{\mathbb H}^{-s + \nu}(\Omega)}, 
\qquad 0 \leq \nu <  1 + s. 
\end{align}
Of interest to us is the case $1 < \nu < 1 +s < 2$.
Then, with the mesh grading parameter $\eta$ satisfying $\eta (-1+\nu) \ge 1$
and upon assuming that $\Y \ge C L$ for $C > 0$ sufficiently large 
so that the condition \eqref{eq:lemma:GEM-30} is satisfied 
we estimate 
\begin{align*}
I & \leq  
\sum_{\ell+\ell' > L} \| \partial_y (Q^\ell_\beta \otimes q^{1,\ell'}_y \ue) \|_{L^2(y^\alpha,\C_\Y)} 
\\
&\leq  
\sum_{\ell+\ell' > L} 
\| \partial_y 
[((I_{x'}\otimes q^{1,\ell'}_y) \circ (Q^\ell_\beta \otimes I_y) \ue] \|_{L^2(y^{\alpha},\C_\Y)}
\\
&\lesssim
\sum_{\ell+\ell' > L} 
2^{-\ell} \|   \partial_y [(I_{x'}\otimes q^{1,\ell'}_y) \ue] \|_{L^2(y^{\alpha},(0,\Y); H^1(\Omega))},
\end{align*}
where, in the last step, we used the approximation property \eqref{eq:Piellbeta-a}. 
We now apply the estimate \eqref{eq:lemma:GEM-40} with $j=0$, 
$\theta = \nu -1$ and $X = H^1(\Omega)$, to arrive at
\[
  I \lesssim \sum_{\ell+\ell' > L} 
  2^{-\ell-\ell'} \| \partial^2_y \ue \|_{L^2(\omega_{\alpha+2(2-\nu),\gamma},(0,\Y);H^1(\Omega))}
  \lesssim
  L2^{-L}  \| f \|_{\mathbb{H}^{-s+\nu}(\Omega)}\;,
\]
where in the last step we have used the regularity estimate \eqref{eq:sparse-10}.

Let us now bound, using similar arguments, the term $II$ in \eqref{eq:tensor-10}. 
From \eqref{eq:H2beta-a} and \eqref{eq:H2beta-b} we obtain, for $1 \leq \nu < 2-s$, the regularity estimate
\begin{equation}
\label{eq:sparse-20} 
\|\partial_y \ue\|_{L^2(\omega_{\alpha + 2(2-\nu),\gamma},(0,\infty);H^2_\beta(\Omega))} 
\lesssim \|f\|_{{\mathbb H}^{-s + \nu}(\Omega)}. 
\end{equation}
Hence, for $\eta (-1 + \nu) \ge 1$, 
and again under the condition that $\Y \ge C L$ so that \eqref{eq:lemma:GEM-10} is satisfied, 
we can estimate 
\begin{align*}
II & \leq  
\sum_{\ell+\ell' > L} 
\| \nabla_{x'} (Q^\ell_\beta \otimes q^{1,\ell'}_y \ue) \|_{L^2(y^\alpha,\C_\Y)} 
\\
&\leq  
\sum_{\ell+\ell' > L} 
\|  \nabla_{x'}
[((I_{x'}\otimes q^{1,\ell'}_y) \circ (Q^\ell_\beta \otimes I_y) \ue] \|_{L^2(y^\alpha,\C_\Y)}
\\
&\lesssim
\sum_{\ell+\ell' > L} 
2^{-\ell} \|  (I_{x'}\otimes q^{1,\ell'}_y) \ue \|_{L^2(y^\alpha,(0,\Y); H^2_\beta(\Omega))}
\end{align*}
where in the last step we used the approximation properties of $\Pi_\beta^\ell$, 
as stated in \eqref{eq:Piellbeta}. 
The approximation properties of $\pi^{1,\ell'}_{\eta,\{\Y\}}$ given in \eqref{eq:lemma:GEM-20} 
with the regularity estimate of \eqref{eq:sparse-20} allow us to conclude that
\[
  II \lesssim \sum_{\ell+\ell' > L} 
2^{-\ell-\ell'} \| \partial_y \ue \|_{L^2(\omega_{\alpha+2(2-\nu),\gamma},(0,\Y);H^2_\beta(\Omega))}
\lesssim
L2^{-L}  \| f \|_{\mathbb{H}^{-s+\nu}(\Omega)}\;.
\]

Collecting the bounds obtained for $I$ and $II$ yields the result.
\end{proof}

Theorem~\ref{thm:sparse} shows that it is possible to obtain near 
optimal order convergence for fractional diffusion in $\Omega$,
by using only $P_1$-FEM in both $\Omega$ and the extended dimension.
An alternative approach is based on exploiting \emph{analytic regularity} 
of the solution of the extended problem. 
In this case,
\emph{exponentially convergent $hp$-FEM with respect to the extended variable $y$}
will achieve near optimal order for conforming $P_1$-FEM in $\Omega$, 
as observed recently in \cite{MPSV17}, and, as we show (by a different argument) 
in Section~\ref{S:hpFEM}, see Theorem~\ref{thm:hpy-gradedx}.
%

\subsection{$hp$-FEM in $(0,\infty)$ and $P_1$-FEM in $\Omega$}
\label{S:hpFEM}
The discretizations in the preceding Sections \ref{S:sGP1FEM} and \ref{S:TP1FEM}
were of first order in $x'$ and $y$. 
We showed that full tensor product FEM
allows to achieve first order convergence in $\Omega$ 
at the expense of superlinear complexity \eqref{eq:NFT}.
Here, we address the use of the so-called $hp$-FEM in $(0,\Y)$; 
the analytic regularity estimates derived
in Section \ref{S:analytic-regularity} allow us to prove 
\emph{exponential convergence estimates} for corresponding
high-order discretizations in $(0,\Y)$. 
We consider two situations: 
\begin{enumerate}[a)]
\item
   The case where $\bmr$ is a so-called linear degree vector in $(0,\Y)$,  
   which will imply exponential convergence with respect to $y$ 
   (cf.~Lemma~\ref{lemma:semidiscretization-error} below).
   If fixed order FEM on a sequence $\{ \calT^\ell_\beta \}_{\ell \geq 0}$
   of regular, simplicial corner-refined meshes in $\Omega$ are used, 
   near optimal, algebraic convergence rates (with respect to the number $\calN_\Omega$
   of degrees of freedom in $\Omega$) result for the solution of 
   \eqref{fl=f_bdddom} in $\Omega$
   (Theorem~\ref{thm:hpy-gradedx}). 
   We mention \cite{khoromskij-melenk03} where, in a structurally
   similar context, analyticity in the extended variable is also exploited by an $hp$-FEM.
\item
   The case where $\bmr$ is a linear degree vector in $(0,\Y)$, 
   and where we use the $hp$-FEM in $\Omega$; 
   in this case, and under the additional assumption \eqref{eq:AnData}
   of analyticity on the data $c,f,A$, 
   \emph{exponential convergence in terms of the number $\calN_{\Omega,\Y}$
   of degrees of freedom in $\C_\Y$} can be achieved.
   We confine the exposition to $\Omega = (0,1)$ and 
   to $\Omega \subset {\mathbb R}^2$ with analytic boundary. 
   This will be the content of Section~\ref{sec:hpx}.
\end{enumerate}
\subsubsection{A univariate $hp$-interpolation operator}
\label{S:hp-interpolant}
We present here the construction of a univariate interpolation operator that leads to exponential
convergence for analytic functions that may have a singularity at $y = 0$. 
The construction is essentially taken from the 
work by Babu{\v s}ka and collaborators, \cite{babuska-gui86b,BabGuoCurved1988} 
and discussed in the literature on $hp$-FEM 
(see, e.g., \cite[Sec.~{4.4.1}]{phpSchwab1998}, 
 \cite[Thm.~{8}]{apel-melenk17} and also \cite{MPSV17}). 
 
To make matters precise, we consider geometric meshes $\calG^M_{geo,\sigma}$ on $[0,\Y]$ 
with $M$ elements and grading factor 
$\sigma \in (0,1)$: $\{I_i\,|\,i=1,\ldots,M\}$ 
with $I_1 = [0,\Y \sigma^{M-1}]$ and 
$I_i = [\Y \sigma^{M-i+1},\Y \sigma^{M-i}]$ for $i=2,\ldots,M$.
On such meshes, we consider 
a \emph{linear  degree vector} $\bmr$ with slope $\slope$ given by 
\begin{equation}\label{eq:lindeg}
r_i := \max\{1,\lceil \slope i \rceil \}\;,\quad i=1,2,...,M\;.
\end{equation}

We denote by $\widehat K = (-1,1)$ the reference interval. 
We will require a base interpolation operator 
$\widehat \Pi_r: H^1(\widehat K) \rightarrow {\mathbb P}_r(\widehat K)  $
that allows for exponential convergence in $r$ 
for analytic functions with the following two properties: 
\begin{enumerate}[1.]
\item 
$(\widehat \Pi_r \widehat u)(\pm 1 ) = \widehat u(\pm 1)$ for all $\widehat u \in H^1(\widehat K)$. 

\item 
For every $K_u > 0$ there exist $C = C(K_u)$, $b = b(K_u) > 0$ such that 
if, for all $\ell \in {\mathbb N}_0$, we have
$\|\widehat u^{(\ell)}\|_{L^2(\widehat K)} \leq C_u  K_u^{\ell+1} (\ell+1)!$ then 
\begin{equation*}
\|\widehat u - \widehat \Pi_r \widehat u\|_{H^{1}(\widehat K)} \lesssim C_u e^{-b r}  
\qquad \forall r \in {\mathbb N}. 
\end{equation*} 
\end{enumerate}
Classical examples of such operators include the Gauss-Lobatto interpolation operator 
and the ``Babu{\v s}ka-Szab{\'o} operator'' $\Pi^{BS}_r$ as described, e.g., in the 
survey \cite[Example~{13}]{apel-melenk17} or in \cite[Theorem~{3.14}]{phpSchwab1998}. 

With the aid of $\widehat \Pi_r$ we introduce the operators 
$\pi^{\bmr}_y$ and  $\pi^{\bmr}_{y,\{\Y\}}$ on an arbitrary 
mesh $\calG^M$  on $[0,\Y]$ with $M$ elements and polynomial degree distribution 
$\bmr \in {\mathbb N}^M$ in an element-by-element fashion in the usual way below. 
However, for $\pi^{\bmr}_y$ 
we modify the approximation on the first element $I_1 = [0,y_1]$ by interpolating in the points 
$y_1/2$ and $y_1$ instead of the endpoints. 
The operator $\pi^{\bmr}_{y,\{\Y\}}$ is obtained by a further modification 
that enforces $\pi^{\bmr}_{y,\{\Y\}}(\Y) = 0$. 
Specifically, with 
$F_{I_i}: \widehat K \rightarrow I_i$ denoting the affine, 
orientation-preserving element maps for element $I_i \in \calG^M$ 
we have 
\begin{align*}
((\pi^\bmr_{y} u)|_{I_1} \circ F_{I_1})(\xi)  
&= 
2 (u \circ F_{I_1})(1) (\xi - 1/2) + 2 (u \circ F_{I_1})(1/2)(1-\xi), 
\\
((\pi^\bmr_{y} u)|_{I_i} \circ F_{I_i})(\xi)  
&= \widehat \Pi_{\bmr_m} (u\circ F_{I_i}), \qquad i=2,\ldots,M, 
\\ 
(\pi^\bmr_{y,\{\Y\}}  u)|_{I_i} & = (\pi^\bmr_{y} u)|_{I_i}, \qquad i=1,\ldots,M-1, 
\\
((\pi^\bmr_{y,\{\Y\}}  u)|_{I_M} \circ F_{I_M})(\xi) 
& = ((\pi^\bmr_{y} u)|_{I_M} \circ F_{I_M})(\xi)  - (u \circ F_{I_M})(1) (\xi+1)/2
\;. 
\end{align*}

The definition of $\pi^\bmr_y$, $\pi^{\bmr}_{y,\{\Y\}}$ is naturally extended for 
functions $u \in C^0((0,\Y];X)$, where $X$ denotes a Hilbert space. We will apply these 
operators to functions from the following two classes of analytic functions of the extended variable $y$:
%
%
\begin{multline}
\label{eq:Bbeta1}
    {\mathcal B}^1_{\beta,\gamma}(C_u,K_u;X) 
    :=\big \{u \in C^\infty((0,\infty);X)\,:\, \|u\|_{L^2(\omega_{\alpha,\gamma},(0,\infty);X)} < C_u, 
    \\
    \|u^{(\ell+1)}\|_{L^2(\omega_{\alpha+2(\ell + 1)-2\beta,\gamma},(0,\infty);X)} 
                  < C_u K_u^{\ell+ 1} (\ell+ 1)! \quad \forall \ell \in {\mathbb N}_0 \big \}
\end{multline}
and
\begin{multline}
\label{eq:Bbeta2}
    {\mathcal B}^2_{\beta,\gamma}(C_u,K_u,X) 
      := \big \{ u \in C^\infty((0,\infty);X)\,: \,  \\
      \|u\|_{L^2(\omega_{\alpha,\gamma},(0,\infty);X)} 
    + \|u^\prime \|_{L^2(\omega_{\alpha,\gamma},(0,\infty);X)} \leq C_u, 
    \\
     \|u^{(\ell+2)}\|_{L^2(\omega_{\alpha+2(\ell + 1)-2\beta},\gamma},(0,\infty);X)
  \leq C_u K_u^{\ell+2} (\ell+2)! \quad \forall \ell \in {\mathbb N}_0 \big \}.  
\end{multline}
We recall that the weight $\omega_{\beta,\gamma}$ is defined as in \eqref{eq:weight}. 
In the case that $X={\mathbb R}$, we omit the tag $X$ in \eqref{eq:Bbeta1}, \eqref{eq:Bbeta2}.

The approximation properties of the operators $\pi_y^\bmr$ and $\pi_{y,\{\Y\}}^\bmr$ are given below.

\begin{lemma}[exponential interpolation error estimates]
\label{lemma:interpolant-on-geometric-mesh}
Let $\beta \in (0,1]$, $\gamma > 0$, $C_u$, $K_u \ge 0$. 
Let $\sigma \in (0,1)$.  
Then there exists a slope $\slope_{min}>0$  for the degree vector 
such that on the geometric mesh $\calG^M_{geo,\sigma}$ the following estimates hold for 
any polynomial degree distribution $\bmr$ with $\bmr_i \ge 1 + \slope_{min} (i-1)$: 
\begin{enumerate}[(i)]
\item 
If $u \in {\mathcal B}^1_{\beta,\gamma}(C_u,K_u;X)$ and $\sigma^M \Y \leq 1$, then 
\begin{align}
\label{eq:B1beta-estimate-10}
\|u - \pi^\bmr_{y} u\|_{L^2(\omega_{\alpha,\gamma},(0,\Y);X)} & \lesssim C_u  \Y^\beta e^{-b M}, 
\\
\label{eq:B1beta-estimate-20}
\|u - \pi^\bmr_{y,\{\Y\}} u\|_{L^2(\omega_{\alpha,\gamma},(0,\Y);X)} 
& \lesssim C_u  \left(\Y^\beta e^{-b M} + \Y^{-1/2+\beta} e^{-\gamma \Y/2}\right).
\end{align}
\item 
If $u \in {\mathcal B}^2_{\beta,\gamma}(C_u,K_u; X)$ and $\sigma^M \Y \leq 1$, 
then  
\begin{align}
\label{eq:B2beta-estimate-10}
\|(u - \pi^\bmr_{y} u)^\prime\|_{L^2(\omega_{\alpha,\gamma},(0,\Y);X)} & \lesssim C_u  \Y^\beta e^{-b M}, 
\\
\label{eq:B2beta-estimate-20}
\|(u - \pi^\bmr_{y,\{\Y\}} u)^\prime\|_{L^2(\omega_{\alpha,\gamma},(0,\Y);X)} 
& \lesssim C_u  \left(\Y^\beta e^{-b M} 
+ \Y^{-3/2+\beta} e^{-\gamma \Y/2}\right).
\end{align}
\end{enumerate}
In all the estimates, 
the hidden constant and $b$ depend only on $\beta$, $\gamma$, $\alpha$, $\sigma$, and $K_u$.
\end{lemma}
\begin{proof}
See Appendix~\ref{S:proof-interpolant-geo-mesh}. 
\end{proof}
\subsubsection{$hp$-discretization in $y$ and $P_1$ FEM in $\Omega$}
With the $hp$-approximation operator $\pi^\bmr_y$ of the previous section at 
hand, we can analyze the properties of the space 
$\V^{1,\bmr}_{h,M}(\calT^\ell_\beta,\calG^M_{geo,\sigma})$. 
The following result generalizes \cite{MPSV17} 
in that we allow for a general elliptic operator $\calL$ and 
in that the appropriate mesh grading in $\Omega$ is included 
to compensate for the lack of a full elliptic shift theorem.

\begin{theorem}[error estimates]
\label{thm:hpy-gradedx}
Let $u\in \mathbb{H}^s(\Omega)$ and $\ue\in\HL(y^\alpha,\C)$ solve \eqref{fl=f_bdddom} and 
\eqref{alpha_harm_intro}, respectively, 
with $f\in \mathbb{H}^{1-s}(\Omega)$ and $\Omega\subset {\mathbb R}^2$ 
a bounded polygon with straight sides and (a finite set of) corners $\{ \bmc \}$.
Let $\beta \ge 0$ be such that \eqref{eq:wgtapriori} holds and 
let $\{\calT^\ell_\beta\}_{\ell}$ be a sequence of graded meshes 
that satisfy \eqref{eq:Piellbeta-a} and \eqref{eq:Piellbeta}.
Let $\calG^M_{geo,\sigma}$ be a geometric mesh on $(0,\Y)$ with $\Y \sim |\log h_\ell|$ 
with a sufficiently large constant. 
Let $\ue_{h_\ell,M}$ be the solution of \eqref{eq:alpha_weak_UhM} 
over the space $\V^{1,\bmr}_{h,M}(\calT^\ell_\beta,\calG^M_{geo,\sigma})$.
Then there exists a minimal slope $\slope_{min}$ independent of $h_{\ell}$ and $f$ 
such that for linear degree vectors $\bmr$
with slope $\slope \ge \slope_{min}$ there holds 
\begin{equation}
\label{eq:thm:hpy-graded} 
\| u - \tr \ue_{h,M} \|_{\Hs} \lesssim \|\nabla(\ue - \ue_{h_\ell,M}) \|_{L^2(y^\alpha,\C)} 
\lesssim h_\ell \|f\|_{{\mathbb H}^{1-s}(\Omega)}. 
\end{equation}
In addition, the total number of degrees of freedom behaves like
\[
\dim \V^{1,\bmr}_{h,M}(\calT^\ell_\beta,\calG^M_{geo,\sigma}) 
\sim
\calN_{\Omega,\Y} \sim M^2 h_\ell^{-2} \sim h_\ell^{-2} (\log h_\ell)^2 
\sim \calN_\Omega \log \calN_\Omega,
\]
where $\calN_\Omega = \# \calT_\beta^\ell$. 
More generally, if $f \in {\mathbb H}^{\sigma-s}(\Omega)$ for $\sigma \in [0,1]$, 
then the bound \eqref{eq:thm:hpy-graded} takes the form 
\[
\| u - \tr \ue_{h,M} \|_{\Hs} \lesssim
\|\nabla(\ue - \ue_{h_\ell,M}) \|_{L^2(y^\alpha,\C)} 
\lesssim h_\ell^\sigma \|f\|_{{\mathbb H}^{\sigma-s}(\Omega)}.\]
\end{theorem}
\begin{proof}
The starting point is again the error decomposition \eqref{eq:ErrSplt}. 
The univariate $hp$-interpolation
operator $\pi^{\bmr}_y$ constructed in Section~\ref{S:hp-interpolant} makes the  
semidiscretization error $\ue - \pi^\bmr_y \ue$ in $y$ exponentially small in $M$  
(see Lemma~\ref{lemma:semidiscretization-error} below for details). 
In turn, the assumption $M \sim |\log h_\ell|$ implies 
any desired algebraic convergence in $h_\ell$ by suitably selecting the implied constant. 
On the other hand, the error $\ue - \Pi^q_{x'} \ue$ in \eqref{eq:ErrSplt} 
is controlled as in the proof of Theorem~\ref{thm:P1Graded}.

Finally, the estimate for $f \in \mathbb{H}^{\sigma-s}(\Omega)$ follows by interpolation.
\end{proof}

\section{Diagonalization: semidiscretization in $y$}
\label{S:diagonalization-abstract-setting}
%
We now explore the possibilities offered by a semidiscretization in $y$. 
We will observe, among other things, that this leads to a sequence of 
decoupled singularly perturbed, \emph{linear second order}
elliptic problems in $\Omega$.

For an arbitrary mesh $\calG^M$ on $[0,\Y]$ and for a 
polynomial degree distribution $\bmr$, we 
consider the following $y$-semidiscrete problem: 
Find $\ue_M \in \V^\bmr_M(\calC_\Y)$ such that
\begin{equation}
\label{eq:semidiscrete-galerkin-formulation}
\blfa{\C}(\ue_M,\phi) = 
d_s\langle f, \tr \phi \rangle 
\qquad \forall \phi \in \V^\bmr_M(\C_\Y),
\end{equation}
where $\V^\bmr_M(\C_\Y)$ is defined as in \eqref{eq:xySemiDis} and 
is a closed subspace of $\HL(y^{\alpha},\C)$. In what follows we obtain 
an explicit formula for $\ue_M$. 
To accomplish this,  we consider the following eigenvalue problem: 
Find $(v,\mu) \in S^\bmr_{\{\Y\}}( (0,\Y), \calG^M) \setminus\{0\} \times {\mathbb R}$ such that 
\begin{equation}
\label{eq:eigenvalue-problem}
\mu \int_0^\Y y^\alpha v^\prime(y) w^\prime(y)\, \diff y = 
        \int_0^\Y y^\alpha v(y) w(y)\, \diff y  
\qquad \forall w \in S^\bmr_{\{\Y\}}( (0,\Y), \calG^M),
\end{equation} 
where $S^\bmr_{\{\Y\}}( (0,\Y), \calG^M)$ is defined as in Section \ref{S:NtFESpc}. 
All eigenvalues $\mu$ are positive, and the space $S^\bmr_{\{\Y\}}((0,\Y),\calG^M)$ has an eigenbasis 
$(v_i)_{i=1}^\calM$, with $\calM:= \dim S^\bmr_{\{\Y\}}( (0,\Y), \calG^M)$, such that,
for $i$, $j \in \{1,\ldots,\calM \}$,
\begin{equation}
\label{eq:eigenbasis-normal}
\int_0^\Y y^\alpha v_i^\prime(y) v_j^\prime(y)\, \diff y = \delta_{i,j}, 
\qquad 
\int_0^\Y y^\alpha v_i(y) v_j(y)\, \diff y = \mu_i \delta_{i,j}. 
\end{equation}
We now write
$\ue_M(x',y):= \sum_{j=1}^\calM U_j(x') v_j(y)$ and consider
$\phi(x',y)=V(x')v_i(y)$, with $V\in H^1_0(\Omega)$ as a test function, in 
\eqref{eq:semidiscrete-galerkin-formulation}. This yields the following
system of \emph{decoupled} problems for $i=1,\ldots,\calM$: 
Find $U_i \in H^1_0(\Omega)$ such that
\begin{align}
\label{eq:decoupled-problems}
\blfa{\mu_i,\Omega}(U_i,V) = d_s v_i(0) \langle f,V\rangle
\qquad \forall V \in H^1_0(\Omega),
\end{align}
where
\[
\blfa{\mu_i,\Omega}(U,V):= \mu_i \blfa{\Omega}(U,V) + \int_\Omega U V \diff x',
\]
and $\blfa{\Omega}$ is introduced in \eqref{eq:blfOmega}.
An important observation is that, for functions of the form 
$Z(x',y)= \sum_{i=1}^\calM V_i(x') v_i(y)$ with $V_i\in H^1_0(\Omega)$,  we have
the equality
\begin{equation}
\label{eq:pythagoras}
\blfa{\C}(Z,Z)  = 
\blfa{\C_\Y}(Z,Z)  = 
\sum_{i=1}^\calM \|V_i\|^2_{\mu_i,\Omega}, 
\qquad \|V\|^2_{\mu_i,\Omega}:= \blfa{\mu_i,\Omega}(V,V). 
\end{equation}

To obtain a fully discrete scheme, select a mesh $\calT$ on $\Omega$ and the corresponding space 
$S^q_0(\Omega,\calT)$ and let $\Pi_i:H^1_0(\Omega) \rightarrow S^q_0(\Omega,\calT)$ be the Ritz projectors 
for the bilinear forms $\blfa{\mu_i,\Omega}$: 
\begin{equation}
\label{eq:galerkin-for-vi} 
\blfa{\mu_i,\Omega}(u - \Pi_i u,v) = 0 \qquad \forall v \in S^q_0(\Omega,\calT). 
\end{equation}

With this notation at hand, we can formulate an explicit representation of the 
Galerkin approximation $\ue_{h,M} \in S^q_0(\Omega,\calT) \otimes S^\bmr_{\{\Y\}}(\calG^M)$ 
to $\ue$ as well as an error representation.

\begin{lemma}[error representation]
\label{lemma:decoupled-problems}
Let $(\mu_i,v_i)_{i=1}^\calM$ be the eigenpairs given by 
\eqref{eq:eigenvalue-problem}, \eqref{eq:eigenbasis-normal}. 
Let $U_i \in H^1_0(\Omega)$ be the solution to \eqref{eq:decoupled-problems} 
and $\Pi_i:H^1_0(\Omega) \rightarrow S^q_0(\Omega,\calT)$ given as in \eqref{eq:galerkin-for-vi}.  
Let $\ue_M$ be the solution to the semidiscrete problem \eqref{eq:semidiscrete-galerkin-formulation}.
Then the Galerkin approximation 
$\ue_{h,M} \in S^q_0(\Omega,\calT) \otimes S^\bmr_{\{\Y\}}(\calG^M)$ to $\ue$ 
satisfies 
\begin{align}
\label{eq:representation-by-reaction-diffusion-problems}
\ue_{h,M}(x',y) & = \sum_{i=1}^{\calM} \Pi_i U_i(x') v_i(y), \\
\label{eq:error-representation}
\blfa{\C}(\ue_M - \ue_{h,M}, \ue_M - \ue_{h,M}) &= 
\sum_{i=1}^{\calM} \|U_i - \Pi_i U_i\|^2_{\mu_i,\Omega}. 
\end{align}
\end{lemma}
\begin{proof}
Expression \eqref{eq:representation-by-reaction-diffusion-problems}
follows from \eqref{eq:decoupled-problems} and
\eqref{eq:galerkin-for-vi},
whereas \eqref{eq:error-representation} is a consequence of
\eqref{eq:pythagoras}.
\end{proof}

We next show that the semidiscretization error $\ue - \ue_M$ 
can be made exponentially small on geometric meshes $\calG^M_{geo,\sigma}$.

\begin{lemma}[exponential convergence]
\label{lemma:semidiscretization-error} 
Let $f \in {\mathbb H}^{-s+\nu}(\Omega)$ for $\nu \in (0,s)$. Let $c_1 M \leq \Y \leq c_2 M$. 
Consider the geometric mesh $\calG^M_{geo,\sigma}$ on $(0,\Y)$.  
Then there exist $C$, $\slope_{min}$, 
$b > 0$ (depending solely on $s$, $\calL$, $c_1$, $c_2$, $\sigma$, $\nu$)
such that for any linear degree $\bmr$ with 
slope $\slope \ge \slope_{min}$ there holds 
\begin{equation}
\label{eq:lemma:semidiscretization-error-10}
\|\nabla (\ue - \ue_M)\|_{L^2(y^\alpha,\C)} \leq C e^{-b M} \|f\|_{{\mathbb H}^{-s+\nu}(\Omega)}.
\end{equation}
\end{lemma}
\begin{proof} 
We begin the proof by invoking Galerkin orthogonality to arrive at
\begin{align*}
\normC{\ue  - \ue_M}^2 &\leq 
\normC{\ue  - \pi^\bmr_{y,\{\Y\}}\ue} ^2 \\
&\lesssim  
\normCC{\ue  - \pi^\bmr_{y,\{\Y\}}\ue}^2 + 
 \|\nabla \ue\|^2_{L^2(y^\alpha,\C\setminus\C_\Y)},
\end{align*}
where $\| \cdot \|_{\C}$ and $\| \cdot \|_{\C_{\Y}}$ are defined by \eqref{eq:norm-C} and \eqref{eq:blf-A-truncated}, 
respectively. 
Since \eqref{eq:ExpDec} shows that $\|\nabla \ue\|_{L^2(y^\alpha; \C\setminus \C_\Y)}$ 
is exponentially small in $\Y$ we thus focus on the interpolation error term. 
To control such a term we first observe that, in view of the definitions 
of the spaces ${\mathcal B}^j_{\beta,\gamma}$, $j \in \{0,1\}$, given by \eqref{eq:Bbeta1}, \eqref{eq:Bbeta2}, 
the regularity estimates \eqref{eq:reg_y_l} and \eqref{eq:reg_x_grad_mu} of  Theorem~\ref{thm:glob_reg_ue}, 
imply that
$\ue$ viewed as a function in $C^\infty((0,\infty),L^2(\Omega))\cap C^\infty((0,\infty),H^1_0(\Omega))$ 
satisfies for $\nu \in (0,s)$ and $K > \kappa$ (with $\kappa$ as in Theorem~\ref{thm:glob_reg_ue}) 
\begin{align}
\ue \in {\mathcal B}^1_{\nu,\gamma}(C\|f\|_{{\mathbb H}^{-s+\nu}(\Omega)},K; H^1_0(\Omega)) 
\cap 
{\mathcal B}^2_{\nu,\gamma}(C\|f\|_{{\mathbb H}^{-s+\nu}(\Omega)},K; L^2(\Omega)) .
\label{eq:reg_Bj}
\end{align}
%
{}From Lemma~\ref{lemma:interpolant-on-geometric-mesh} together with the fact that 
$\Y \sim M$ we conclude that
\begin{align}
\|\nabla_{x'} (\ue - \pi^\bmr_{y,\{\Y\}} \ue)\|_{L^2(y^\alpha,\C_\Y)} 
&\leq C e^{-b M} \|f\|_{{\mathbb H}^{-s+\nu}(\Omega)}, \\
\|\partial_y  (\ue - \pi^\bmr_{y,\{\Y\}} \ue)\|_{L^2(y^\alpha,\C_\Y)} &\leq C e^{-b M}
\|f\|_{{\mathbb H}^{-s+\nu}(\Omega)},
\end{align}
with $b>0$ slightly smaller than that in \eqref{eq:B1beta-estimate-20} and
\eqref{eq:B2beta-estimate-20}. This implies the desired estimate \eqref{eq:lemma:semidiscretization-error-10} and concludes the proof. 
\end{proof}

Finally, for the geometric mesh $\calG^M_{geo,\sigma}$ with the linear
degree vector $\bmr$ and truncation parameter $\Y\sim M$, 
we have the following estimates for the eigenvalues $\mu_i$ 
of problem \eqref{eq:eigenvalue-problem} 
and for the point values $v_i(0)$ in \eqref{eq:decoupled-problems}.

\begin{lemma}[properties of the eigenpairs]
\label{lemma:lambda}
Let $\calG^M_{geo,\sigma}$ 
be a geometric mesh on $(0,\Y)$ and $\bmr$ a linear degree vector with slope $\slope$. 
If $c_i M\le \Y \le c_2 M$, then
there are constants $C$, $b$ depending only on $\sigma$ such that for the 
eigenpairs $(\mu_i,v_i)_{i=1}^\calM$ given by 
\eqref{eq:eigenvalue-problem}, \eqref{eq:eigenbasis-normal} we have that: 
\begin{align*}
\|v_i\|_{L^\infty(0,\Y)} \leq C M^{(1-\alpha)/2}, 
\qquad C^{-1} \slope^{-2} M^{-1} \sigma^{M} \leq \mu_i \leq C M^2. 
\end{align*}
\end{lemma}
\begin{proof}
The results follow from Lemmas~\ref{lemma:weighted-Linfty}, \ref{lemma:upper-bound-lambda}, 
and \ref{lemma:lower-bound-lambda}.
\end{proof}

The previously described approach that perform a semidiscretization in $y$ 
leads to structural insight into the 
regularity properties of the solution $\ue$: 
it shows that, up to an exponentially small, in $\Y$, error 
introduced by cutting off at $\Y$, the solution $\ue$ can be expressed in terms of solutions of 
singularly perturbed reaction--diffusion type problems. 
(A similar structural property for $\ue(\cdot,0)$ 
can also be seen from the Balakrishnan formula, e.g., \cite[Equation (4)]{BP:13}). 
In what follows we will exploit this to design appropriate 
approximation spaces in the $x'$-variable.
Nevertheless, the diagonalization 
\eqref{eq:semidiscrete-galerkin-formulation}--\eqref{eq:decoupled-problems}
has more far-reaching ramifications:
\begin{enumerate}[$\bullet$]
\item 
The diagonalization technique can be exploited numerically 
as it is not restricted to the semi-discrete case.
It holds for arbitrary, closed tensor product approximation spaces 
${\mathbb W} \otimes {\mathbb Q}$, where ${\mathbb W} \subset H^1_0(\Omega)$ and 
${\mathbb Q} \subset H^1_{\{\Y\}}(y^\alpha ,(0,\Y))$.
It completely decouples the solution of the full Galerkin problem, based on
${\mathbb W} \otimes {\mathbb Q}$, into the (parallel) solution of 
$\dim {\mathbb Q}$ 
problems of size $\dim {\mathbb W}$.  
The numerical experiments in Section~\ref{S:NumExp} exploit this observation;
see Remark~\ref{rem:implementation} below.

\item 
The observation \eqref{eq:pythagoras} allows one to gauge the impact 
of solving approximately the $\dim {\mathbb Q}$ problems that are of (singularly perturbed) reaction--diffusion type. 
For convex domains $\Omega$ and spaces ${\mathbb W}$ based on piecewise linears on quasi-uniform meshes, robust, 
(with respect to the singular perturbation parameter), multigrid methods are available 
(see, e.g., \cite{reusken-olshanskii00}).  

\item 
The diagonalization technique \eqref{eq:eigenvalue-problem}--\eqref{eq:decoupled-problems} 
also suggests another numerical technique: 
approximate each 
solution $U_i$ from a different (closed) space $W_i \subset H^1_0(\Omega)$. 
This leads to the approximation of $\ue$ in 
the space $\sum_{i=1}^{\calM} v^i(y) W_i$. 
The resulting Galerkin approximation still satisfies 
\eqref{eq:representation-by-reaction-diffusion-problems} and 
\eqref{eq:error-representation}. 
This approach produces approximation spaces in $\Omega\times (0,\Y)$ 
that do not have tensor product structure but still provides exponential convergence.
As in the sparse grids case of Section~\ref{S:sGP1FEM}
this approach allows for reducing the number of degrees of freedom without sacrificing much accuracy;
specifically, the exponent $1/4$ in the 
exponential convergence bound \eqref{eq:ExpN1/4} that we obtain in the next section 
could be reduced to $1/3$ if $\Omega$ is an interval and 
the exponent $1/5$ in \eqref{eq:ExpN1/5} could be reduced to $1/4$ 
if $\Omega \subset {\mathbb R}^2$ has an analytic boundary, albeit
at the expense of breaking the tensor product structure of the discretization.
\end{enumerate}
\section{$hp$-FE discretization in $\Omega$}
\label{sec:hpx}

Up to this point, we have exploited the analytic regularity of the solution $\ue$
in the extended variable $y$ in order to 
recover (up to logarithmic terms) optimal complexity of a $P_1$-FEM, 
for \eqref{fl=f_bdddom} posed in the polygon $\Omega\subset \bbR^2$, by full tensorization of a 
$hp$-FEM with respect to $y$ with the $P_1$-FEM in $\Omega$

As a final goal, in this section we employ, 
\emph{in addition}, an $hp$-FEM in $\Omega$ to obtain an
\emph{exponentially convergent, local FEM}
for the fractional diffusion problem \eqref{fl=f_bdddom}. 
Naturally, stronger regularity assumptions on the data $f$, $A$ and $c$
will be required: 
in addition to the previously made assumptions
on these data, we assume in Section~\ref{S:hp-hp}
\begin{equation}\label{eq:AnData}
c,f \in \calA(\overline{\Omega},\bbR)\;, 
\quad 
A\in \calA(\overline{\Omega},\GL(\bbR^d))
\;.
\end{equation}
Here, $\calA(\overline{\Omega},G)$ 
denotes the set of functions which are analytic in $\overline{\Omega}$
and take values in the group $G$.
%

\subsection{Tensorized $hp$-FEM in $\Omega \times (0,\Y)$}
\label{S:hp-hp}
The choice of the meshes $\calG^M$ and $\calT$ as well as 
the degree vector $\bmr$ and the polynomial degree $q$ 
were not specified in Section~\ref{S:diagonalization-abstract-setting}. 
Mesh design principles for problems as \eqref{eq:decoupled-problems} 
are available in the literature. 
For meshes, in an $h$-version context, we mention 
the so--called Shishkin meshes and 
refer to \cite{roos-stynes-tobiska96} for an in-depth discussion 
of numerical methods for singular perturbation problems. 
Here, we focus on the $hp$-version. 
Appropriate mesh design principles 
ensuring robust exponential convergence of $hp$-FEM
have been developed in \cite{schwab-suri96,schwab-suri-xenophontos98,melenk97,melenk-schwab98,melenk02}. 
In these references, linear second order elliptic singular perturbations 
with a single length scale and exponential boundary layers were considered.
As is revealed by the diagonalization \eqref{eq:decoupled-problems}, 
the $y$-semidiscrete solution
\eqref{eq:semidiscrete-galerkin-formulation} 
contains $\calM$ separate length scales $\mu_i$, $i=1,...,\calM$.
These need to be resolved \emph{simultaneously} by the $x'$-discretization space. 
To this end, 
based on \cite{schwab-suri96,schwab-suri-xenophontos98,melenk97,melenk-schwab98,melenk02},
we employ a mesh 
that is geometrically refined towards $\partial\Omega$ such that 
the smallest length scale $\mu_{\calM}$ is resolved. 
We illustrate the key points in the following 
Sections~\ref{S:1d-bdy-geomesh} and \ref{S:2d-bdy-geomesh}
in dimension $d=1$, and in dimension $d=2$ for smooth boundaries. 
\subsubsection{Exponential convergence of $hp$-FEM in one dimension}
\label{S:1d-bdy-geomesh}
To gain insight into how to discretize the family of problems \eqref{eq:decoupled-problems}, 
we first consider the following reaction-diffusion problem in $\Omega = (0,2)$: 
given $f\in \calA(\overline{\Omega};\bbR)$
and a parameter $ 0 < \varepsilon \leq 1$,
find $u_\varepsilon \in H^1_0(\Omega)$ 
such that 
\begin{equation}
\label{eq:1d-model-reaction-diffusion}
- \varepsilon^2 u_\varepsilon^{\prime\prime} + u_\varepsilon = f \quad \mbox{ on $\Omega$}, 
\qquad u_\varepsilon(0) = u_\varepsilon(2) = 0\;.
\end{equation}
For \eqref{eq:1d-model-reaction-diffusion},
$hp$-Galerkin FEM afford \emph{robust exponential convergence}. 
The following result is a particular instance of \cite[Proposition~{20}]{melenk97}.
\begin{proposition}[exponential convergence]
\label{prop:melenk97-prop20}
Let $\Omega = (0,2)$.  Let $\geomesh{1D}$ be a mesh on $\Omega$ that is geometrically refined towards $\partial \Omega = \{ 0,2 \}$ with $L$ layers and 
grading factor $\sigma \in (0,1)$:
\begin{equation}
\label{eq:1D-geomesh}
\geomesh{1D} := \{(0,\sigma^{L}), (2-\sigma^L,2)\}  
\cup \{(\sigma^{L-i+1}, \sigma^{L-i}), (2-\sigma^{L-i}, 2-\sigma^{L-i+1})
\}_{i=1}^L.  
\end{equation}
Select $L$ such that $\sigma^L \leq \varepsilon\leq 1$.  
Let $f$ satisfy the analytic regularity estimates
\begin{equation} \label{eq:rhs-analytic}
\|f^{(\ell)} \|_{L^2(\Omega)} \leq C_f K_f^\ell \ell! \qquad \forall \ell \in {\mathbb N}_0,
\end{equation}
for some constants $C_f$, $K_f>0$ that depend on $f$.
Then there exist constants $C$, $b > 0$ independent of $\varepsilon \in (0,1]$ 
such that for the
Galerkin approximation $u_{\varepsilon}^{q,L} \in S^q_0(\Omega,\geomesh{1D})$ 
of the solution $u_\varepsilon$ of \eqref{eq:1d-model-reaction-diffusion} one has 
exponential convergence in the energy norm, given by 
$\| w \|_{\varepsilon^2,\Omega}^2 
:= 
\varepsilon^2 \| w' \|_{L^2(\Omega)}^2 + \| w \|_{L^2(\Omega)}^2$, i.e.
$$
\|u_{\varepsilon} - u_{\varepsilon}^{q,L}\|_{\varepsilon^2,\Omega} 
\lesssim C_f e^{-b q}.
$$
Here the hidden constant and the constant $b$ are independent of $\varepsilon$, 
but depend on $\sigma$ and $K_f$. 
Furthermore, $L = \calO(1+ |\log \varepsilon|)$ so that 
$\dim S^q_0(\Omega,\geomesh{1D})  = \calO(q^2 (1+|\log \varepsilon|))$.
\end{proposition}
\begin{remark}[exponential convergence]
\label{rem:prop:melenk97-prop20}
The discretization described in Proposition~\ref{prop:melenk97-prop20} 
and its properties warrant the following comments.
\begin{enumerate}[$\bullet$]
\item The case $\epsilon \geq 1$: 
Although Proposition~\ref{prop:melenk97-prop20} restricts to $\varepsilon \in (0,1]$,
one can check
that for $\varepsilon \ge 1$, the mesh degenerates into a fixed mesh with 
three points $\{0,1,2\}$ and the corresponding approximation result reads 
\begin{equation}
\label{eq:1d-singular-perturbation-convergence}
\|u_{\varepsilon} - u_{\varepsilon}^{q,L}\|_{\varepsilon^2,\Omega}
\lesssim (1 + \varepsilon) C_f e^{-b q}
\;. 
\end{equation}
\item Different length scales: 
Proposition~\ref{prop:melenk97-prop20} gives
\emph{robust exponential convergence} 
and 
\emph{does not require explicit knowledge
of the singular perturbation parameter $\varepsilon$}, 
but only a lower bound for it.
This is crucial for the presently considered fractional diffusion problem, 
where the decoupled problems \eqref{eq:decoupled-problems} depend on 
several length scales  given by $\lambda_i$ 
(which, in turn, depend on the discretization in 
the extended variable $y\in (0,\Y)$). 
Applying a tensor product $hp$-FE space directly (i.e., without
explicit diagonalization 
\eqref{eq:semidiscrete-galerkin-formulation}--\eqref{eq:decoupled-problems})
to the extended problem \eqref{alpha_harm_intro} 
based on the tensor product of the 
$hp$-FE space $S^q_0(\Omega,\geomesh{1D})$
and on the 
$hp$-FE space $S^\bmr_{\{\Y\}}( (0,\Y),\calG^M_\sigma)$
obviates the numerical solution of the 
generalized eigenproblem \eqref{eq:eigenvalue-problem}.
It requires, however, the $hp$-space $S^q_0(\Omega,\geomesh{1D})$
to \emph{concurrently} approximate the solutions of 
\emph{all singularly perturbed problems \eqref{eq:decoupled-problems}}
in $\Omega$ with exponential convergence rates.
\item Different meshes: 
If an eigenbasis $(v_i)_{i=1}^\calM$ satisfying \eqref{eq:eigenbasis-normal} is available,
then for each of the decoupled singularly perturbed problems in $\Omega$,
a geometric boundary layer mesh is not mandatory to achieve 
robust exponential convergence. 
A coarser mesh, tailored to the specific
length scale $\mu_i$ in the $i$-th equation of 
\eqref{eq:decoupled-problems}, will then suffice;
we refer to \cite{schwab-suri96, phpSchwab1998} for details.
\eremk
\end{enumerate}
\end{remark}

Lemma~\ref{lemma:lambda} asserts that the reaction-diffusion problems 
\eqref{eq:decoupled-problems} are singularly perturbed with 
length scale $\mu_i$ ranging from $\calO(M^{-1} \sigma^M)$ to $\calO(M^2)$. 
Proposition~\ref{prop:melenk97-prop20} implies exponential convergence rates
under the analyticity assumption \eqref{eq:AnData}. 
In the next result, we combine these two observations to obtain an exponentially convergent $hp$-FEM 
for the fractional diffusion problem in $\Omega$.
\begin{theorem}[exponential convergence]
\label{thm:hp-for-fractional}
Let $u\in \mathbb{H}^s(\Omega)$ and $\ue\in\HL(y^\alpha,\C)$ solve \eqref{fl=f_bdddom} and \eqref{alpha_harm_intro}, respectively, 
with $\Omega = (0,2)$, $A=I$, $c=0$ and $f$ satisfying \eqref{eq:AnData}. Given   
fixed constants $c_1$, $c_2>0$, let $\calG^M_{geo,\sigma}$ be a geometric mesh on $[0,\Y]$ with 
grading factor $\sigma \in (0,1)$ and such that $c_1 M \leq \Y \leq c_2 M$. 
Let $\bmr$, on $\calG^M_{geo,\sigma}$, 
be the linear degree vector with slope $\slope$.  
Let $\geomesh{1D}$ be a geometric mesh in $\Omega$ 
as described in Proposition~\ref{prop:melenk97-prop20} 
with an integer $L$ such that 
\begin{equation}\label{eq:LMY}
\sigma^{2L} \leq \Y (\slope M)^{-2} \sigma^M\;.
\end{equation}
Then, there are constants $b$, $\slope_{min} > 0$ 
independent of $M$ and $\Y$ such that  for $\slope \ge \slope_{min}$ 
the Galerkin approximation 
$\ue_{q,\bmr} \in S^q_0(\Omega,\geomesh{1D}) \otimes S^\bmr_{\{\Y\}}((0,\Y),\calG^M_{geo,\sigma})$ 
to $\ue$ satisfies 
\begin{align}
\| u - \tr \ue_{q,\bmr}\|_{\Hs} \lesssim \|\nabla (\ue - \ue_{q,\bmr}) \|_{L^2(y^\alpha,\C)} 
\lesssim \left( M^2 e^{-bq} + e^{-b M} \right),
\end{align}
where the hidden constant is independent of $M$ and $\Y$.
%
%
In addition, as $M\to \infty$, with $L$ and $M$ related by
\eqref{eq:LMY}, we have that, uniformly in $q \in {\mathbb N}$, the total number of degrees of freedom behaves like
$$
\calN_{\Omega,\Y}:= 
\dim S^q_0(\Omega, \geomesh{1D}) \otimes S^\bmr_{\{\Y\}}( (0,\Y), \calG^M_\sigma)
= \calO(qM^3). 
$$
Choosing, in particular, $q \sim M$ 
yields a convergence rate bound in terms of the total number of 
degrees of freedom $\calN_{\Omega,\Y}$ of the form
\begin{align} \label{eq:ExpN1/4}
\| u - \tr \ue_{q,\bmr} \|_{\Hs} 
\lesssim \|\nabla (\ue - \ue_{q,\bmr}) \|_{L^2(y^\alpha,\C)} \lesssim \exp(-b' \calN_{\Omega,\Y}^{1/4})
\end{align}
for some $b' > 0$ independent of $\calN_{\Omega,\Y}$. 
\end{theorem}
\begin{proof} 
Let $\ue_M$ solve \eqref{eq:semidiscrete-galerkin-formulation}. 
We proceed in two steps.

\emph{Bounds on the semidiscretization error $\ue - \ue_M$:} 
By the assumption of analyticity of $f$, there exist constants $C_f$, $K_f$ such that \eqref{eq:rhs-analytic} holds. 
We thus have that $f \in {\mathbb H}^{1/2 - \delta}(\Omega)$ for any $\delta > 0$. 
Consequently, an application of Lemma~\ref{lemma:semidiscretization-error}.
reveals that for a sufficiently large slope $\slope$ of the linear degree vector $\bmr$ (depending on the 
constants $K_f$ in the analytic regularity bound \eqref{eq:rhs-analytic} of the data $f$)
there exists $b > 0$ such that
$$
\|\nabla( \ue - \ue_M) \|_{L^2(y^\alpha,\C)} \lesssim e^{-b M}. 
$$

\emph{Bounds on the errors $\|U_i - \Pi_i U_i\|_{\mu_i,\Omega}$:} 
We first notice that  Lemma~\ref{lemma:lambda} immediately yields $\slope^{-2} M^{-1} \sigma^{M} \lesssim \mu_i$. 
This, in view of the assumption \eqref{eq:LMY}, implies that $\sigma^{2L} \lesssim \mu_i$. 
Consequently, given that $f$ is analytic on $\overline{\Omega}$, we apply Proposition~\ref{prop:melenk97-prop20}
(more precisely, the refinement \eqref{eq:1d-singular-perturbation-convergence} to obtain that
\begin{equation}
\label{eq:thm:hp-for-fractional-10}
\|U_i  - \Pi_i U_i\|_{\mu_i,\Omega} \lesssim \Y e^{-b q} \lesssim M e^{-b q},
\end{equation}
where we have also used that $\mu_i \lesssim M^2 \lesssim \Y^2$, which follows, again, from Lemma~\ref{lemma:lambda} and the condition $c_1 M \leq \Y \leq c_2 M$. We recall that $\| \cdot \|_{\mu_i,\Omega}$ is defined as in \eqref{eq:pythagoras}. 
Finally, combining \eqref{eq:thm:hp-for-fractional-10} with \eqref{eq:error-representation} 
and recalling that $\calM \lesssim M^2$ give 
$$
\|\ue_M - \ue_{h,M}\|^2_{L^2(y^\alpha,\C)} \lesssim  \calM M^2 e^{-2 bq} 
\lesssim M^4 e^{-bq}. 
$$
This concludes the proof. 
\end{proof}

\begin{remark}[other operators]
Theorem~\ref{thm:hp-for-fractional} also holds for 
$0< c \in \bbR$ by arguing as in the proof 
of Theorem~\ref{thm:hp-for-fractional-2D} ahead. 
\eremk
\end{remark}

\begin{remark}[mesh gradings $\Omega$]
The condition \eqref{eq:LMY} is a sufficient condition ensuring that the smallest 
boundary layer length scale (characterized by $\min_{i} \mu_i$) that arises 
from the diagonalization is resolved by the mesh $\geomesh{1D}$. 
More generally, if the geometric mesh of \eqref{eq:1D-geomesh} 
were based on the mesh grading factor $\sigma_{x'} \in (0,1)$
(distinct from the factor $\sigma$ in the mesh in the extended variable $y$), 
then condition \eqref{eq:LMY} could be replaced with 
$\sigma_{x'}^{2L} \lesssim \Y (\slope M)^{-2} \sigma^{M}$ for some constant 
independent of $L,M,\Y$.
\eremk
\end{remark}
%

%
\subsubsection{Exponential convergence of $hp$-FEM in two dimensions}
\label{S:2d-bdy-geomesh}

\begin{figure}
\psfragscanon 
\psfrag{1}{$1$}
\psfrag{s}{\small$\sigma$}
\psfrag{s2}{\small $\sigma^2$}
\psfrag{sl}{\small $\sigma^L$}
\psfrag{S0}{$S^L$}
\psfrag{S1}{$S^{L-1}$}
\psfrag{SL}{$S^0$}
\includegraphics[width=0.25\textwidth]{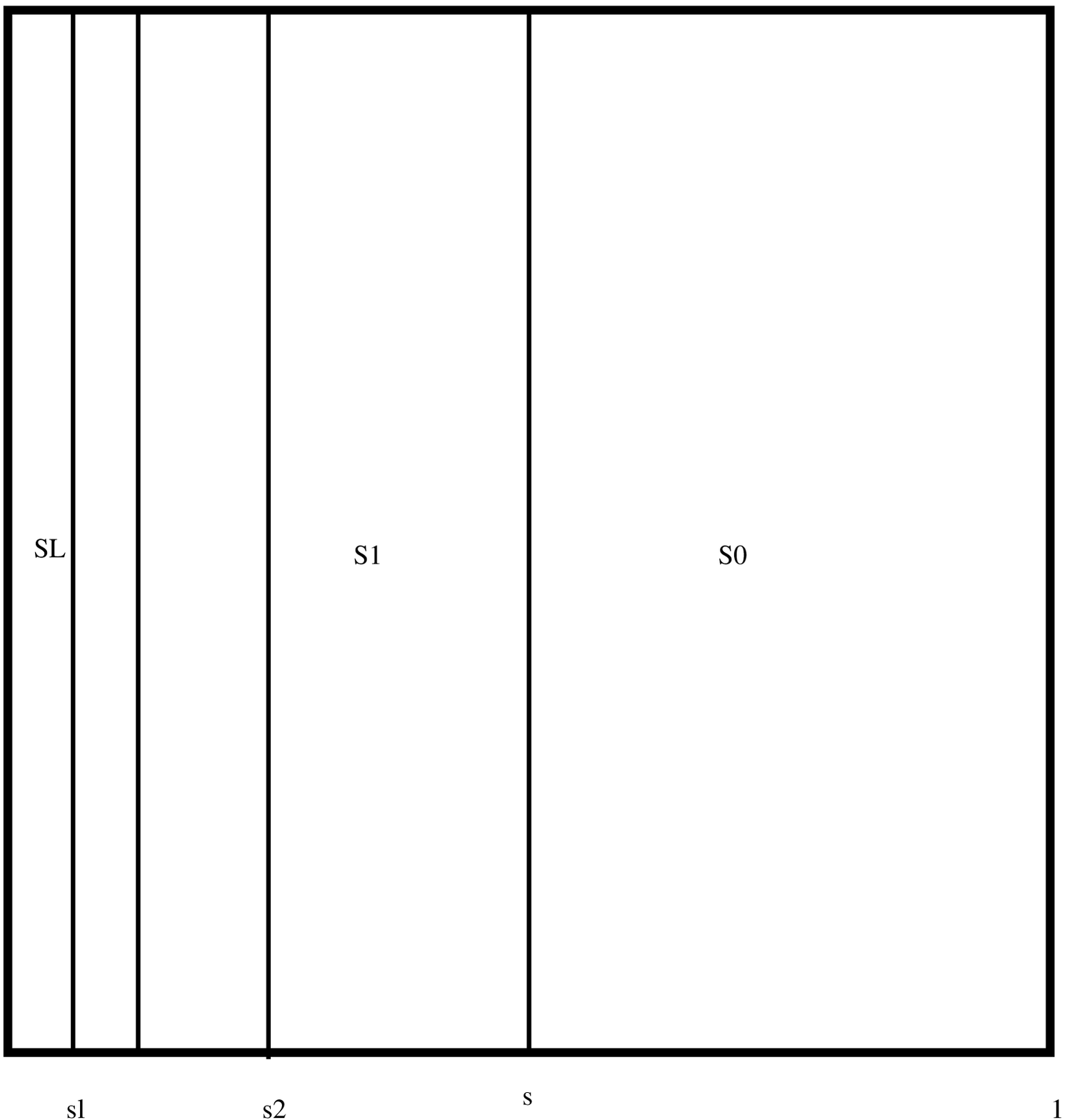}
\hfill 
\psfrag{O1}{$\Omega_1$}
\psfrag{O2}{$\Omega_2$}
\psfrag{O3}{$\Omega_3$}
\psfrag{O4}{$\Omega_4$}
\psfrag{O5}{$\Omega_5$}
\includegraphics[width=0.25\textwidth]{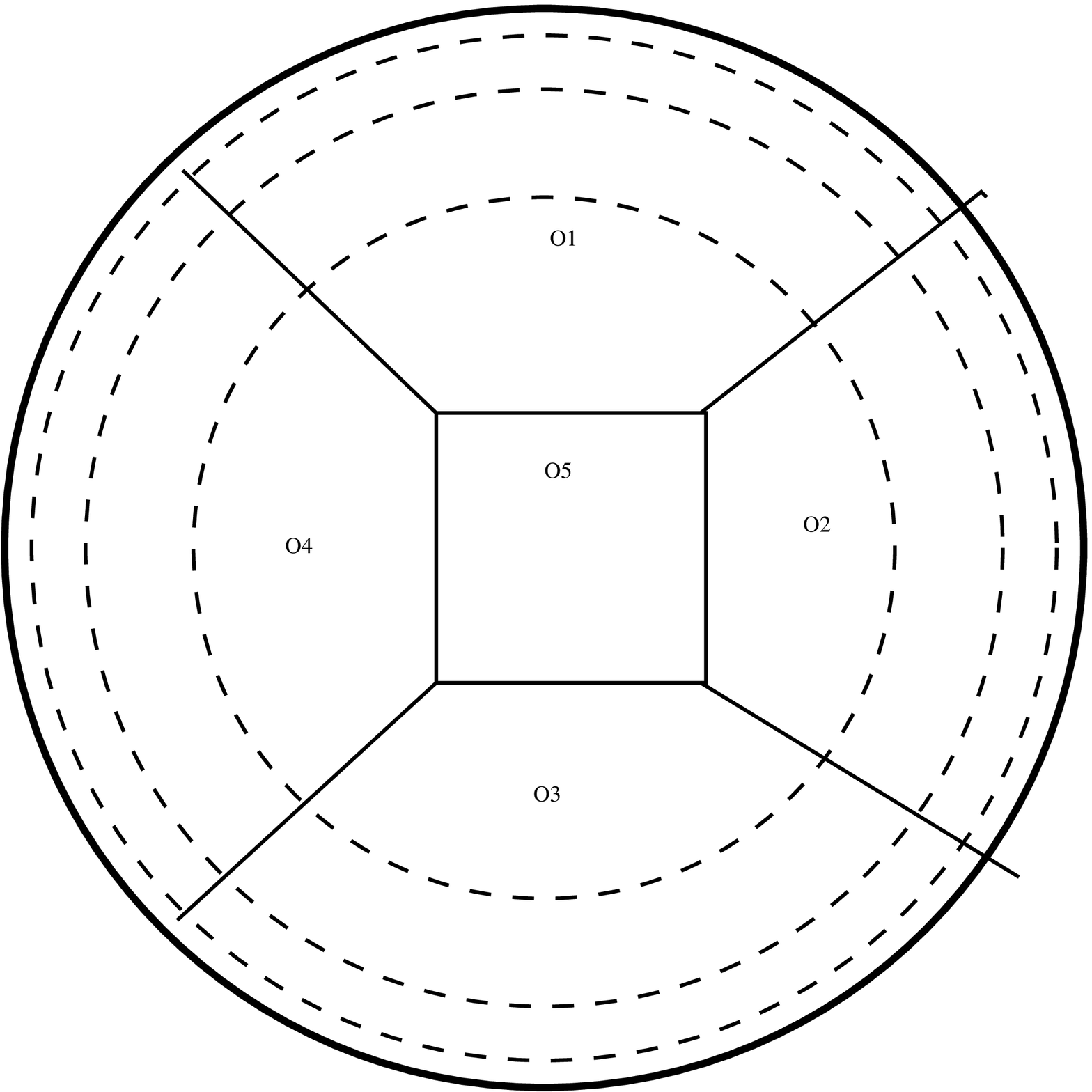}
\psfragscanoff
\caption{\label{fig:aniso-mesh} 
Anisotropic geometric mesh (see Definition~\ref{def:anisoGeoMesh}). 
Left: geometric refinement of the reference patch. 
Right: Example of mesh with $N = 5$ and $n = 4$.  
Solid lines indicate patches, dashed lines represent 
mesh lines introduced by refinement of reference patches.
}
\end{figure}

Let us now discuss the extension of the ideas of Section~\ref{S:1d-bdy-geomesh} to the 
two dimensional case. As it is structurally similar to the univariate case, we proceed briefly. 
For domains $\Omega \subset {\mathbb R}^d$, $d > 1$, with \emph{smooth} 
boundary, the boundary layers presented in the solutions $U_i$ of the singularly perturbed problems \eqref{eq:decoupled-problems} can be resolved
by meshes that are \emph{anisotropically} refined towards the boundary $\partial\Omega$. 
A two dimensional analogue of the meshes $\calT^{1D}_{geo,L}$ of Proposition~\ref{prop:melenk97-prop20} 
is presented in \cite[Section~{3.4.3}]{melenk-schwab98} and illustrated in Figure~\ref{fig:aniso-mesh} (right). 
These \emph{anisotropic geometric meshes} $\geomesh{2D}$ are created as push-forwards of
anistropically refined geometric meshes on references patches 
as detailed in the following definition, where
we follow the notation employed in \cite[Section {3.4.3}]{melenk-schwab98}.

\begin{definition}[anisotropic geometric meshes $\geomesh{2D}$]
\label{def:anisoGeoMesh}
Denote by $S = [0,1]^2$ the reference element. 
Let $\Omega_i$, $i=1,\ldots,N$, be a fixed mesh on $\Omega \subset {\mathbb R}^2$ consisting of curvilinear
quadrilaterals with bijective element maps $M_i:S \rightarrow \Omega_i$ satisfying the ``usual'' conditions
for $H^1$-conforming triangulations (see \cite[(M1)--(M3) in Section~{3.1}]{melenk-schwab98}
for the precise definition). The elements $\Omega_i$ are called \emph{patches} and the associated
maps $M_i$ \emph{patch maps}. 
Let $\Omega_i$, $i=1,\ldots,n \leq N$, be such that the left edge $e:= \{0\} \times (0,1)$ of 
$S$ is mapped to $\partial\Omega$, i.e., 
$M_i(e_1) \subset \partial\Omega$, and that $M_i(\partial S \setminus e) \cap \partial\Omega = \emptyset$.
Assume that the remaining elements $\Omega_{i}$, $i=n+1,\ldots,N$ satisfy $\overline{\Omega}_i \cap \partial\Omega = \emptyset$.

Subdivide the reference element $S$ into $L+1$ rectangles $S^\ell$, $\ell=0,\ldots,L$, 
as follows for chosen grading factor $\sigma \in (0,1)$:
\begin{equation}
\label{eq:aniso-geo} 
S^0 = (0,\sigma^L) \times (0,1), 
\qquad S^\ell = (\sigma^{L+1-\ell},\sigma^{L-\ell}) \times (0,1), \qquad \ell=1,\ldots,L. 
\end{equation}
Define elements $\Omega_i^\ell$, $i=1,\ldots,n$, $\ell=0,\ldots,L$, and the corresponding element
maps $M_i^\ell:S\rightarrow \Omega_i^\ell$ by
\begin{align*}
\Omega_i^0 & := M_i(S^0),&  M_i^0(\xi,\eta)&:= M_i(\xi \sigma^L,\eta), \\
\Omega_i^\ell & := M_i(S^\ell),&  M_i^\ell(\xi,\eta)&:= M_i(\sigma^{L+1-\ell} + \xi \sigma^{L-\ell},\eta), 
\qquad \ell=1,\ldots,L. 
\end{align*}
The mesh $\geomesh{2D}$ given by the elements
$\{\Omega_i^\ell\,:\, i=1,\ldots,n, \ell=0,\ldots,L\} \cup \{\Omega_j\,:\,  j=n+1,\ldots,N \}$ 
with corresponding element maps introduced above is a triangulation of $\Omega$ that 
satisfies the ``usual'' conditions of $H^1$-conforming
triangulations, i.e., conditions \cite[(M1)--(M3) in Section~{3.1}]{melenk-schwab98}. For $\geomesh{2D}$ 
the FE-space is given by the standard $H^1_0(\Omega)$-conforming space of mapped polynomials of degree $q$: 
\begin{equation}
\label{eq:quadFEspace}
S^q_0(\geomesh{2D}):= \{u \in H^1_0(\Omega)\,:\, u|_K \circ F_K \in {\mathbb Q}_q(S)
\quad \forall K \in \geomesh{2D}\}, 
\end{equation}
where $F_K:S \rightarrow K$ is the element map of $K \in \geomesh{2D}$
and ${\mathbb Q}_q(S)$ is the space of polynomials of degree $q$ in each variable on $S$.
\end{definition}

For such anisotropically refined meshes, we have the following exponential convergence result.

\begin{theorem}[exponential convergence]
\label{thm:hp-for-fractional-2D}
Let $u\in \mathbb{H}^s(\Omega)$ and $\ue\in\HL(y^\alpha,\C)$ 
solve \eqref{fl=f_bdddom} and \eqref{alpha_harm_intro}, 
respectively, with $\Omega \subset {\mathbb R}^2$ having an analytic boundary, 
$A=I$, $0 \leq c \in \bbR$, and $f$ satisfying the 
regularity requirement \eqref{eq:AnData} \eqref{eq:AnData}.
Given fixed constants $c_1$, $c_2>0$, let $\calG^M_{geo,\sigma}$ be a geometric mesh on $[0,\Y]$ 
with grading factor $\sigma \in (0,1)$ and such that $c_1 M \leq \Y \leq c_2 M$. 
Let $\bmr$, on $\calG^M_{geo,\sigma}$, be the linear degree vector with slope $\slope$.
Assume that $L$ is chosen such that \eqref{eq:LMY} holds.
Let $\geomesh{2D}$ be an anisotropic geometric
mesh with $L$ layers as described in Definition~\ref{def:anisoGeoMesh} where, 
additionally, the patch maps $M_i$, $i=1,\ldots,N$ are assumed to be analytic.
Then, there are constants $C$, $b$, $\slope_{min} > 0$
independent of $M$ and $\Y$ such that for $\slope \ge \slope_{min}$
the Galerkin approximation
$\ue_{q,\bmr} \in S^q_0(\Omega,\geomesh{2D}) \otimes S^\bmr_{\{\Y\}}((0,\Y),\calG^M_{geo,\sigma})$ 
to $\ue$ satisfies
\begin{align}
\label{eq:thm:hp-2D-conv}
\| u - \tr \ue_{q,\bmr}\|_{\Hs} \lesssim \|\nabla (\ue - \ue_{q,\bmr}) \|_{L^2(y^\alpha,\C)} 
\leq C \left( M^2 e^{-bq} + e^{-b M} \right).  
\end{align}
Furthermore, as $M\to \infty$, with $L$ related to $M$ by
\eqref{eq:LMY}, we have that, uniformly in $q \in {\mathbb N}$, 
the total number of degrees of freedom behaves like
$$
\calN_{\Omega,\Y}:= 
\dim S^q_0(\Omega,\geomesh{2D}) \otimes S^\bmr_{\{\Y\}}( (0,\Y), \calG^M_\sigma)
= \calO(q^2M^3). 
$$
Choosing, in particular, $q \sim M$
yields a convergence rate bound in terms of the total number of
degrees of freedom $\calN_{\Omega,\Y}$ of the form
\begin{align} \label{eq:ExpN1/5}
\| u - \tr \ue_{q,\bmr}\|_{\Hs} 
\lesssim \|\nabla (\ue - \ue_{q,\bmr}) \|_{L^2(y^\alpha,\C)} 
\lesssim \exp(-b' \calN_{\Omega,\Y}^{1/5}) 
\end{align}
for some $b' > 0$ independent of $\calN_{\Omega,\Y}$.
\end{theorem}
\begin{proof}
The proof parallels that of Theorem~\ref{thm:hp-for-fractional}. 
We start with the case $c = 0$. 
By the arguments in \cite[Section~{3.4.3}]{melenk-schwab98} 
the meshes $\geomesh{2D}$ allow for estimates of the form 
\begin{equation}
\label{eq:Ui-approximation-2D}
\inf_{v \in S^q_0(\Omega,\geomesh{2D})} \|U_i - v\|_{\mu_i,\Omega} \lesssim e^{-b q} 
\end{equation}
for the solutions $U_i$ of \eqref{eq:decoupled-problems}, provided $L$ and $\mu_i$ 
satisfy $\sigma^{2L} \lesssim \mu_i$, which is ensured by 
assumption \eqref{eq:LMY}. 
Here, the implied constant and $b > 0$ depend on $f$, $\partial\Omega$, and the 
analyticity of the patch maps $M_i$, $i=1,\ldots,N$. 
The estimates \eqref{eq:Ui-approximation-2D} then 
allow us to conclude the proof for $c = 0$ as in Theorem~\ref{thm:hp-for-fractional}.

For $c \ne 0$, we observe that the singularly perturbed problems \eqref{eq:decoupled-problems} 
in $\Omega$ take the form
\begin{equation*}
-\mu_i \Delta U_i + (1 + c \mu_i) U_i = f \quad \mbox{ on $\Omega$}, \qquad U_i|_{\partial\Omega} = 0. 
\end{equation*}
This can be transformed to the case $c = 0$ by rewriting it
in terms of
$\widetilde \mu_i:= \mu_i/ (1 + c \mu_i)$ as
\begin{equation*}
-\widetilde \mu_i \Delta U_i + U_i = \widetilde f:=\frac{1}{1 + c \mu_i} f  \quad \mbox{ on $\Omega$}, 
\qquad U_i|_{\partial\Omega} = 0. 
\end{equation*} 
The approximation result \eqref{eq:Ui-approximation-2D} 
holds again (with $\mu_i$ replaced with $\widetilde\mu_i$ there). 
\end{proof}

\begin{remark}[limitations and extensions]
\label{rem:hp-for-fractional}
The result of Theorem~\ref{thm:hp-for-fractional-2D} warrants the following remarks:
\begin{enumerate}[(i)]
\item 
Theorem~\ref{thm:hp-for-fractional-2D} is restricted to $A = I$ and 
to the coefficient $c$ being constant,
as it relies on \cite{melenk-schwab98}, 
which in turn builds on the regularity theory developed in \cite{MR1664765}. 
The results of \cite{melenk-schwab98} can be generalized to 
$A$ and $c$ that satisfy \eqref{eq:AnData} using the results from \cite{melenk02}. 
In turn, Theorem~\ref{thm:hp-for-fractional-2D} 
could be generalized to this setting as well.

\item 
Theorem~\ref{thm:hp-for-fractional-2D} 
can be expected to generalize to $\Omega \subset {\mathbb R}^d$ 
with $d>2$ if $\partial \Omega$ is analytic. The underlying reason for this 
is that the boundary layers are structurally a one dimensional phenomenon, 
which can be resolved
with anisotropic refinement towards $\partial\Omega$.  
The approximation result 
\eqref{eq:thm:hp-2D-conv} can therefore be expected to hold, 
however, the complexity 
is then $\calN_{\Omega,\Y} = \calO( q M^{d+2})$, 
resulting in an exponential convergence bound of 
$\exp(-b'\calN_{\Omega,\Y}^{1/(d+3)})$. 

\item 
Theorem~\ref{thm:hp-for-fractional-2D} does generalize to 
so-called ``bounded, curvilinear polygonal domains'' 
$\Omega \subset {\mathbb R}^2$. 
The analogue of Proposition~\ref{prop:melenk97-prop20}, i.e., 
a rigorous convergence analysis of $hp$-FEM in $\Omega$ for 
the single-scale reaction diffusion problem with the appropriate
mesh refinement towards the corners of $\Omega$ is available in \cite{melenk02}. 
\eremk
\end{enumerate}
\end{remark}

\section{Numerical experiments} 
\label{S:NumExp}
We consider $A = I$ and $c = 0$, \ie $\mathcal{L}^s = (-\Delta)^s$. 
Most of the numerical experiments will be performed on the so-called $L$-shaped 
polygonal domain $\Omega \subset \R^2$ determined by the vertices
\[ 
\bmc\in \{ (0,0), (1,0), (1,1), (-1,1), (-1,-1), (0,-1)\}.
\]
For validation purposes again, we consider the following smooth exact 
solution with the corresponding right-hand side (recall $x'=(x_1,x_2) \in \Omega$)
\begin{equation}
  \label{eq:exact_L}
  u(x_1,x_2) = \sin \pi x_1 \sin \pi x_2, \quad
  f(x_1,x_2) = (2\pi^2)^s \sin \pi x_1 \sin \pi x_2.
\end{equation}
To investigate the effect of mesh refinement in $\Omega$, we also consider
\begin{equation}
  \label{eq:nonsmooth_f}
  f(x_1,x_2) \equiv 1 \;.
\end{equation} 
Notice that, in this case, $f\in \calA(\overline{\Omega},\bbR)$, 
but $f\in \Ws$ only for $ s > 1/2 $ due to boundary incompatibility.
The exact solution is not known, so that
the error will be estimated numerically, with reference to
an accurate numerical solution. 
The error measure will always be the energy norm
\[
  \|u-\tr{\ue_{h,M}}\|^2_{\Hs}
  \lesssim \|\nabla (\ue- \ue_{h,M})\|^2_{L^2(y^\alpha,\C)}
  = d_s\int_\Omega f (u - \tr{\ue_{h,M}} ),
  \]
where $\ue_{h,M}$ denotes the discrete solution in $\C_\Y$. 

Finally, a one-dimensional example $\Omega = (0,1)$ will be described to 
illustrate $hp$-FEM in $\Omega \times (0,\Y)$.

\begin{remark}[implementation]
\label{rem:implementation}
Let us provide some algorithmic details of the methods used in practical computations.
For the chosen discrete spaces the mass and stiffness matrices in $\Omega$ and $(0,\Y)$ are computed. 
We then numerically solve the generalized eigenvalue problem \eqref{eq:eigenbasis-normal}, 
thereby arriving at $\calM$ decoupled linear systems: 
Find  $U_i \in S^q_0(\Omega,\calT)$ such that
  \begin{equation}
\label{eq:decoupled-problems-disc}
\blfa{\mu_i,\Omega}(U_i,V) = d_s v^i(0) {\int_\Omega f V \diff x'} 
             \qquad \forall V \in S^q_0(\Omega,\calT),
\end{equation}
where $\blfa{\mu_i,\Omega}$ is defined in \eqref{eq:decoupled-problems}.
Following \eqref{eq:representation-by-reaction-diffusion-problems}, 
the solution is then obtained by
\[
  \ue_{h,M}(x',y)  = \sum_{i=1}^{\calM} v^i(y) U_i(x').
\]
The implementation was done in Matlab R2017a, 
with the generalized eigenvalue problem solved with  {\tt eig} and the decoupled linear systems  
by a direct solver, i.e., Matlab's ``backslash'' operator.
\eremk
\end{remark}

\subsection{$P_1$-FEM in $\Omega$ with radical meshes in $(0,\Y)$}
\label{S:P1y}
In the following examples we make use of the family of graded meshes $\calG^{k}_{gr,\eta}$ 
as described in Section~\ref{S:p1y} with particular choices $\eta = 2/s$, $k = h/2$, 
and $\Y = |\log h|$, where $h$ denotes the mesh width of the mesh in $\Omega$ to be described next.

\subsubsection{Smooth solution}
\label{S:SmoothSol}
For the first experiment we investigate the smooth solution \eqref{eq:exact_L}. 
We use the $P_1$-FEM in $\Omega$ on a hierarchy of uniformly refined meshes $\calT^\ell$.
The results are displayed in Figure~\ref{fig:conv_smooth}. 
As the theory predicts we see linear convergence in the energy norm 
with respect to the meshwidth $h$.

\begin{figure}
  \begin{minipage}[c]{0.48\textwidth}
      \includegraphics[width=\textwidth]{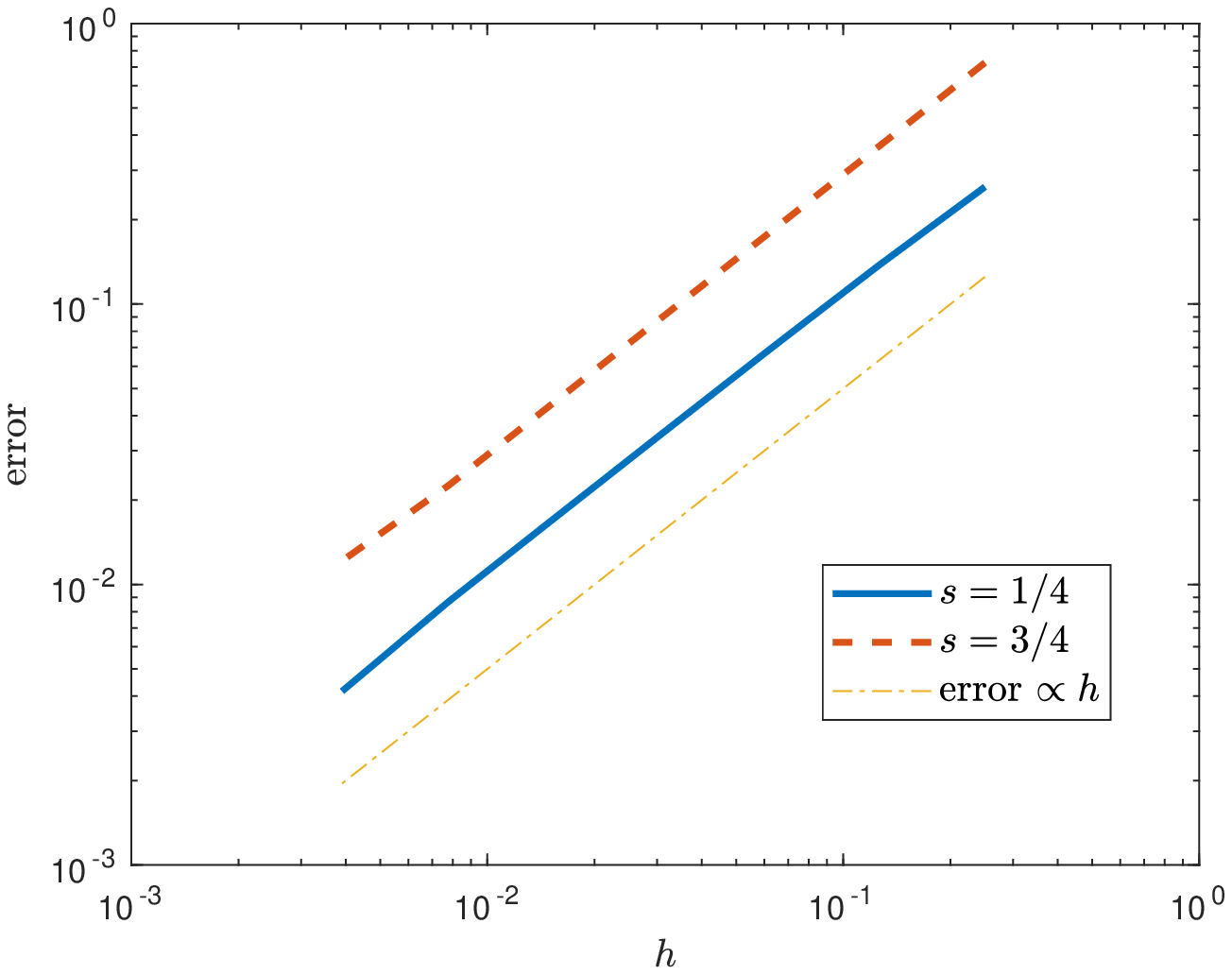}
  \caption{Convergence of the error in the energy norm versus the meshwidth in $\Omega$ 
with the (smooth) exact solution given by \eqref{eq:exact_L}. 
A $P_1$-FEM on uniformly refined meshes in $\Omega$ and $P_1$-FEM on radical meshes in $(0,\Y)$ is used.}
  \label{fig:conv_smooth}
\end{minipage}
\hfill
\begin{minipage}[c]{0.48\textwidth}
  \includegraphics[width=\textwidth]{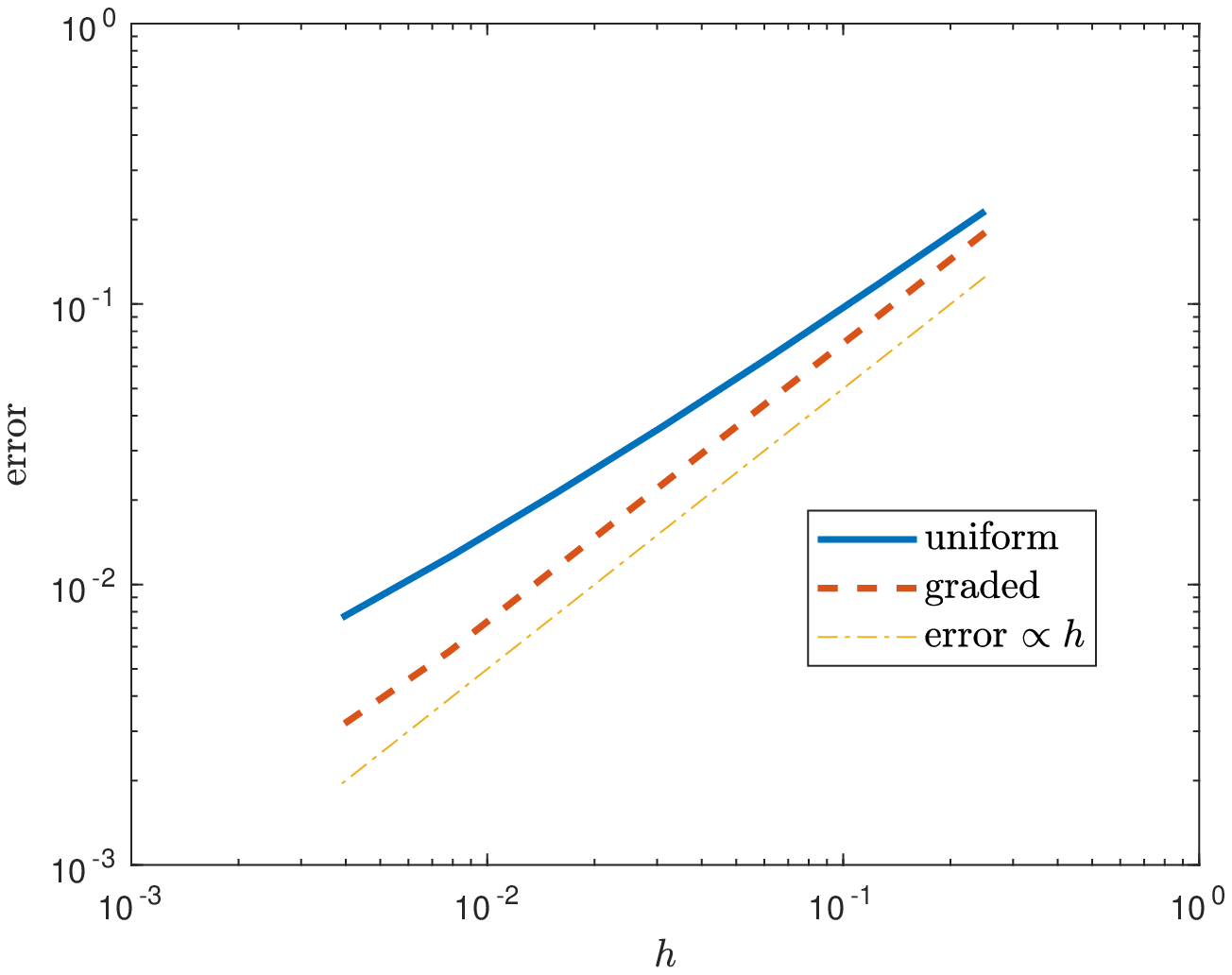}
  \caption{Convergence of the error in the energy norm versus meshwidth
           in $\Omega$ with the right-hand side $f \equiv 1$ and $s = 3/4$, 
           leading to a solution with singular behavior near the re-entrant
            corner $(0,0)$. 
            Error graphs are shown for a $P_1$-FEM on uniformly refined meshes 
            in $\Omega$ and on meshes refined towards the corner.}
  \label{fig:conv_nonsmooth}
\end{minipage}%
\end{figure}

\subsubsection{Mesh refinement at $(0,0)$}
\label{S:MeshRef00}
In the next experiment we consider the case $f \equiv 1 \in \Ws$ for $s \in (1/2,1)$. 
As above  we use the graded mesh $\calG^{k}_{gr,\eta}$ in $(0,\Y)$, 
whereas we now use a hierarchy $\{ \calT^\ell_\beta \}_{\ell \geq 0}$
of bisection--tree meshes in $\Omega$ that are refined towards
the re-entrant corner at $(0,0)$ as constructed in \cite{GspMrin_IMAJNA2009}. 
In Figure~\ref{fig:conv_nonsmooth} 
we see linear convergence with respect to the 
mesh width as predicted by Theorem~\ref{thm:P1Graded} and 
in contrast to the results obtained with uniformly refined meshes.

To the best of the authors' knowledge, the nature of the geometric singularity 
of the solution at the re-entrant corner of the $L$-shaped domain 
for general $0<s<1$ is not known. 

\subsubsection{Sparse grid $P_1$-FEM with mesh refinement at $(0,0)$}
\label{S:SpGP1}
With the above described discrete spaces we are able 
to obtain optimal order convergence with respect to the number of 
degrees of freedom $\calN_\Omega$. 
Nevertheless, the number of degrees of freedom in the extended problem is 
of size $\calO(\calN_\Omega^{1+1/2}\log\log \calN_\Omega)$, i.e., it grows
superlinearly with respect to $\calN_\Omega$. 
To reduce the complexity to nearly linear,
we use sparse grids as explained in Section~\ref{S:sGP1FEM}; 
see in particular the combination formula described in Remark~\ref{rmk:CmbFrml}.  
The results are shown in Figure~\ref{fig:tensor_sparse_hp}. 
These show that the use of sparse grids dramatically reduces 
the number of degrees of freedom and is comparable to $hp$-FEM, 
which is described next. 

\begin{figure}
\begin{minipage}[c]{0.48\textwidth}
  \includegraphics[width=\textwidth]{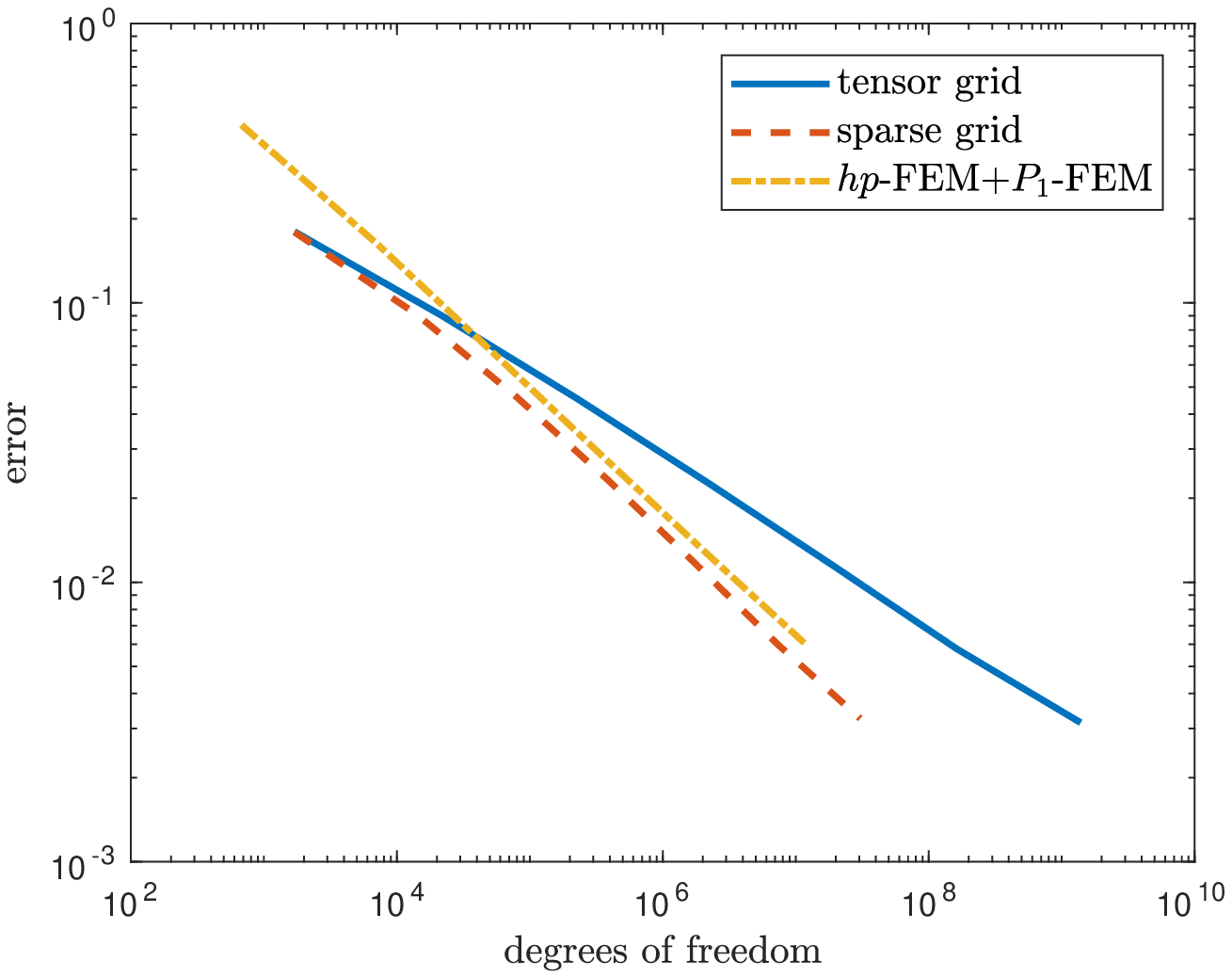}
  \caption{Convergence of the error in the energy norm versus
the number of degrees of freedom of the extended problem with the right-hand side $f \equiv 1$ and $s = 3/4$. 
$P_1$-FEM on corner-refined, regular simplicial meshes is used in $\Omega$. 
We compare  $hp$-FEM in $(0,\Y)$ with tensor grid and sparse grids, the latter two employing radical meshes in $(0,\Y)$.}
  \label{fig:tensor_sparse_hp}
  \end{minipage}
\hfill
  \begin{minipage}[c]{0.48\textwidth}
      \includegraphics[width=\textwidth]{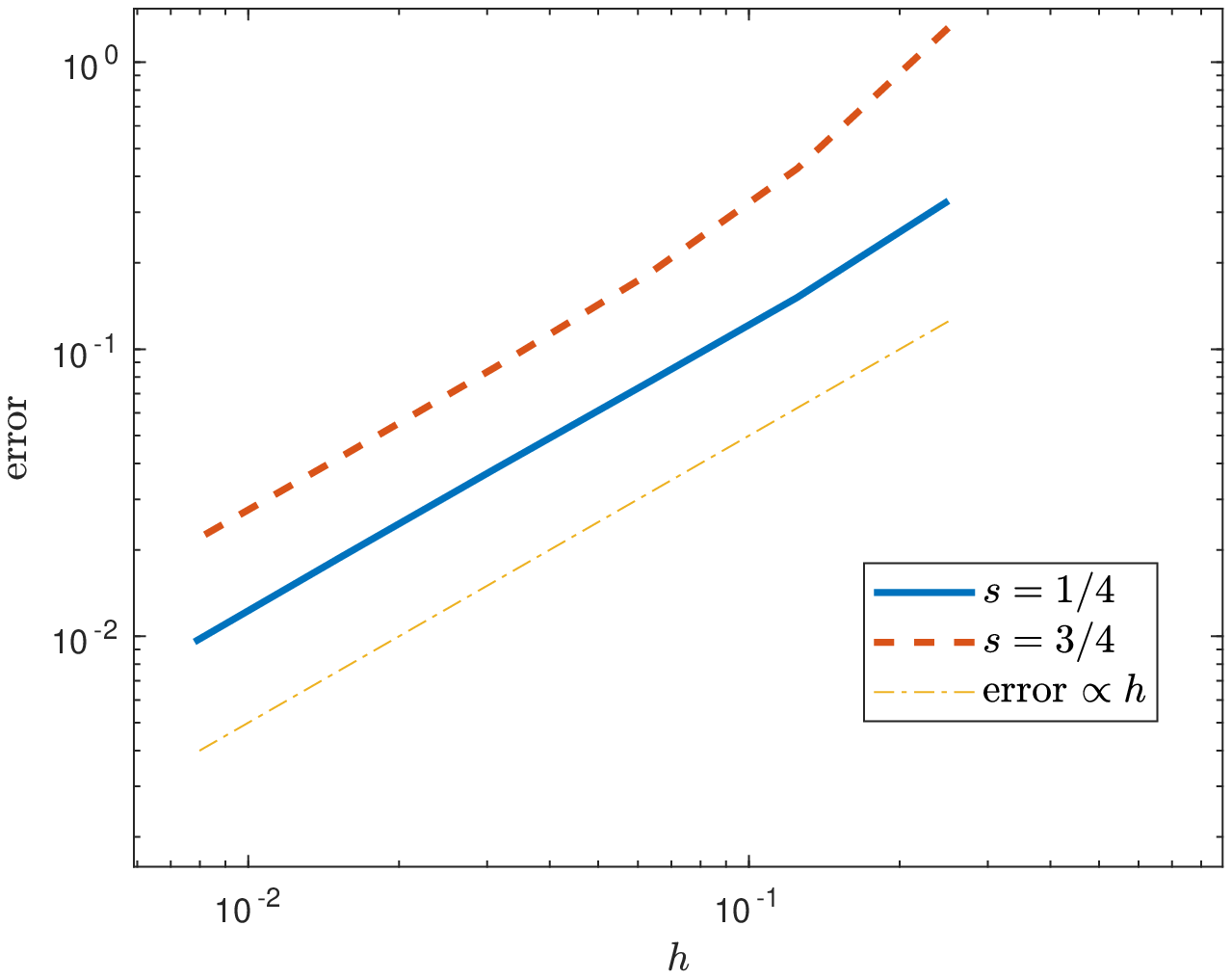}
  \caption{Convergence of the error in the energy norm versus the meshwidth 
           in $\Omega$ with the (smooth) exact solution given by \eqref{eq:exact_L} for two different values of $s$.
A $P_1$-FEM on uniformly refined meshes in $\Omega$ and $hp$-FEM in $(0,\Y)$ is used.}
\label{fig:conv_smooth_hp}
\end{minipage}
\end{figure}
\subsection{$P_1$-FEM in $\Omega$ with $hp$-FEM in $(0,\Y)$}
\label{S:P1xhpy}
We again start with the smooth solution \eqref{eq:exact_L}. 
$P_1$-FEM on uniformly refined meshes is used in $\Omega$, whereas in 
the extended direction $y$ we use $hp$-discretization on the geometric meshes 
$\calG^M_{geo,\sigma}$ on $[0,\Y]$. 
We use $\Y = \tfrac13|\log_2 h|$, 
$M = |\log_2 (h/2)|$, $\sigma = 0.05$,  and linear  degree vector $\bmr$ with slope $\slope = 2$. 
Linear convergence, as predicted by theory, 
can be seen in Figure~\ref{fig:conv_smooth_hp}.
%
We also consider the right-hand side $f \equiv 1$ for $s = 0.75$. 
This time we show convergence 
versus the number of degrees of freedom $\calN_{\Omega,\Y}$ 
in the extended problem and compare with $P_1$-FEM in $\Omega$ on 
so-called \emph{radical meshes}.
We obtain nearly optimal complexity as predicted by theory, but interestingly 
in this example slightly worse behavior compared with sparse grids. 
This is reported in Figure~\ref{fig:tensor_sparse_hp}.

%
\subsection{$hp$-FEM in $(0,1) \times (0,\Y)$}
\label{S:hpFE01}
We consider an example in one space dimension where $\Omega = (0,1)$,
with smooth, but incompatible right-hand side $f \equiv 1$. 
We comment that, according to the regularity results presented in \cite{MR3489634}, 
the solution behaves like
\begin{equation}\label{eq:AlgBdLay}
u(x') \sim
\begin{dcases}
\dist(x',\partial\Omega) + v(x') &\mbox{for}\quad s > 1/2\;, 
\\
\dist(x',\partial\Omega)^{2s} + v(x') & \mbox{for}\quad 0<s<1/2\;,
\end{dcases}
\end{equation}
with $v$ denoting a smoother remainder. 
Here, the singular support of $u$ is $\partial \Omega$, 
i.e.
$u$ exhibits an \emph{algebraic boundary singularity} 
(distinct from the smooth exponential boundary layers 
 arising in linear, elliptic-elliptic singular perturbations)
near the boundary of $\Omega$; see Figure~\ref{fig:blayer}. 

Again, as the exact solution is not known, 
we compare the numerical solution 
with an accurate solution obtained on a finer grid. 

In $(0,\Y)$, we use the same geometric $hp$-FEM space 
$\calG^M_{geo,\sigma}$ as in the previous section. 
The $hp$-FEM space $S^q_0(\Omega,\calT_L)$ 
is as described in Section~\ref{S:hp-hp}, where $q = M$ and $L = M$.  
Exponential convergence with respect to the polynomial degree 
$q$ as predicted by the theory is shown in Figure~\ref{fig:hp_sing_conv}.
\begin{figure}
  \begin{minipage}[c]{0.48\textwidth}
      \includegraphics[width=\textwidth]{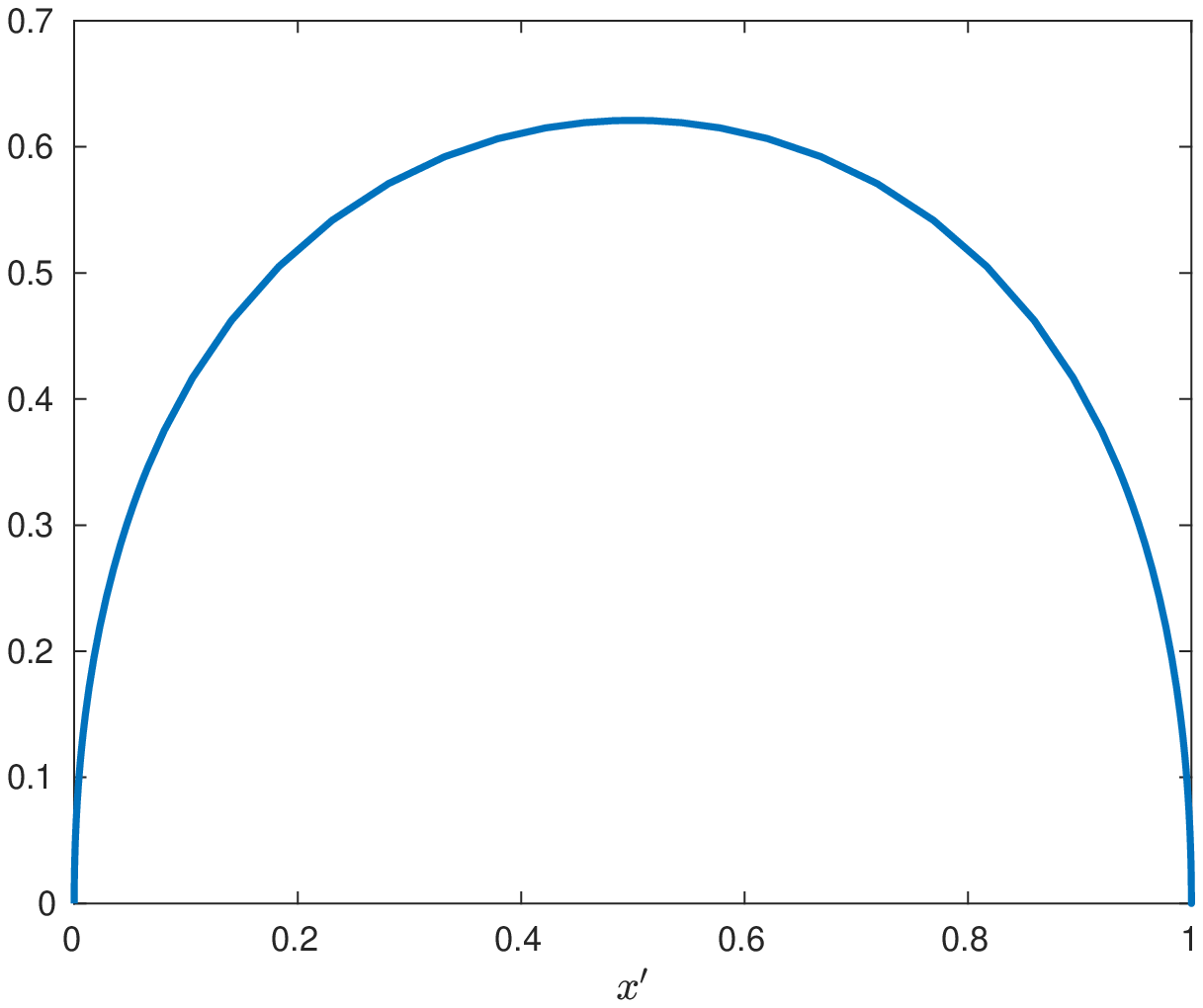}
  \caption{Solution on $\Omega = (0,1)$ with algebraic boundary singularity
           for $s = 0.25$ and  $f \equiv 1$.}
  \label{fig:blayer}
\end{minipage}
\hfill
\begin{minipage}[c]{0.48\textwidth}
  \includegraphics[width=\textwidth]{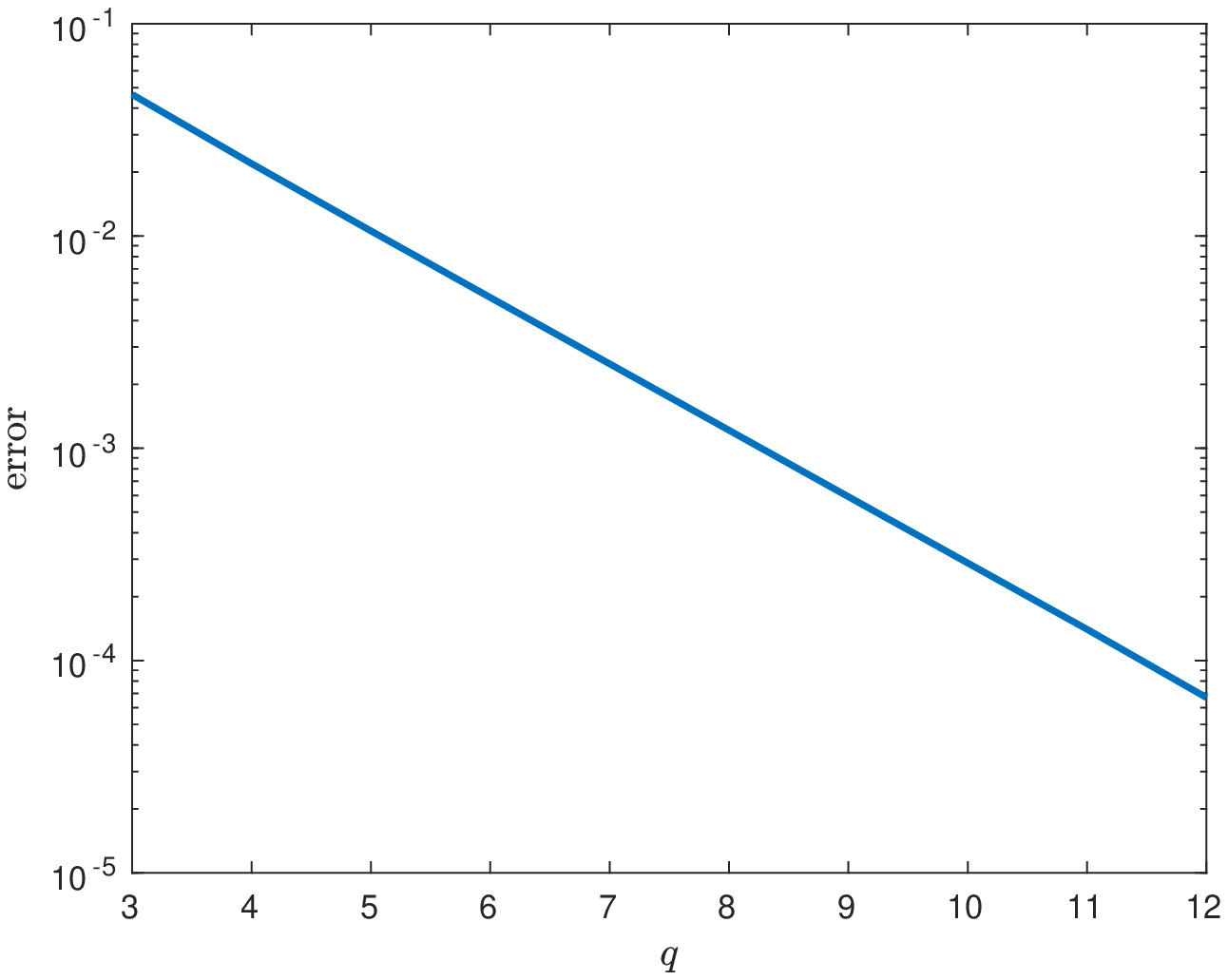}
  \caption{Convergence of error in energy norm of the 
     $hp$-FEM on $\Omega \times (0,\Y)$  against polynomial order $q$ for $s = 0.25$ and $f \equiv 1$.}
  \label{fig:hp_sing_conv}
\end{minipage}%
\end{figure}

In Figure~\ref{fig:AlgbLayer} we illustrate the behavior of the solution given by \eqref{eq:AlgBdLay}. 
We also investigate numerically the borderline case $s = 1/2$ 
in Figure~\ref{fig:AlgbLayer_half}.
\begin{figure}
  \begin{minipage}[c]{0.48\textwidth}
 \includegraphics[width=\textwidth]{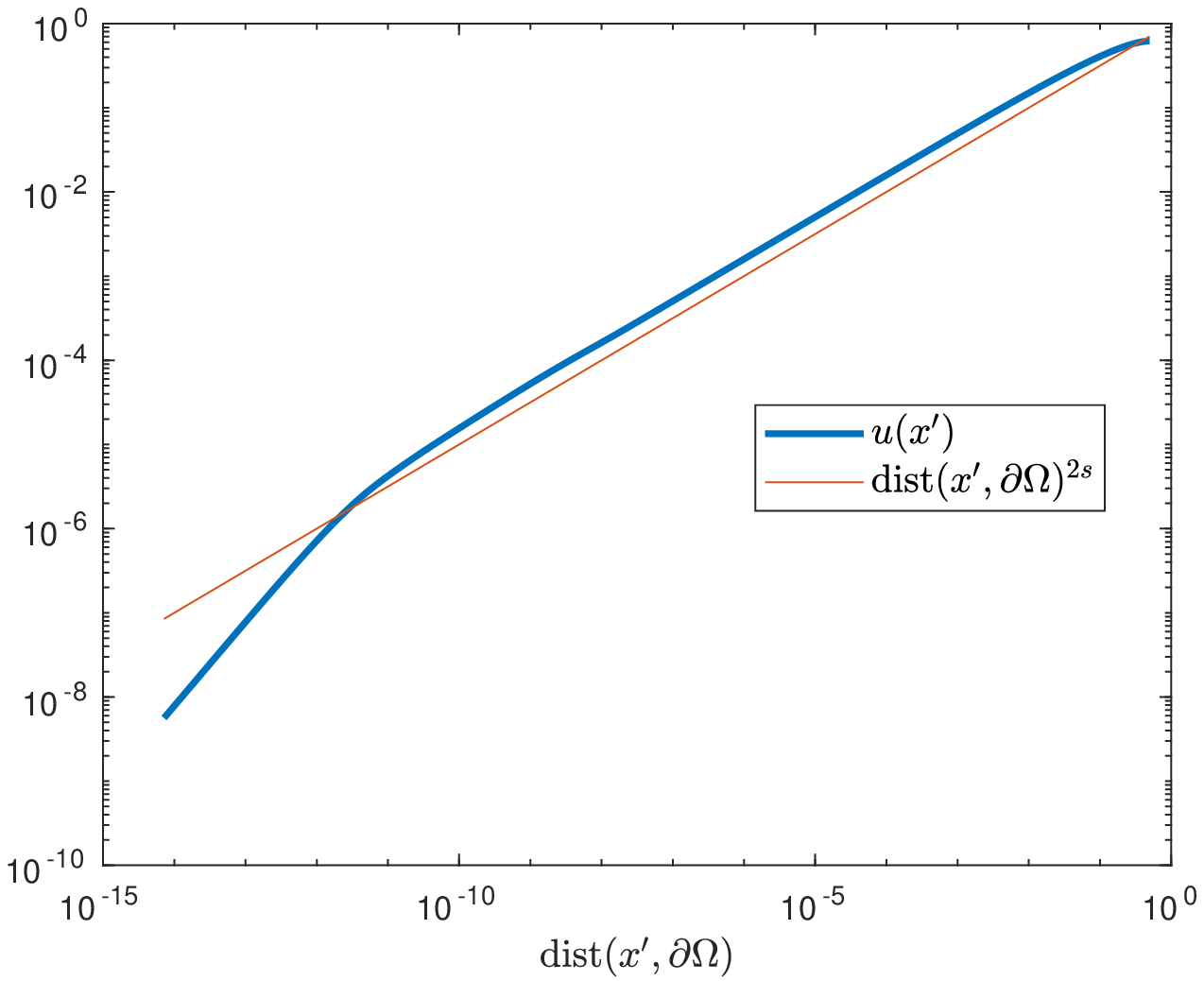}
\caption{
Numerical verification of the algebraic boundary singularity \eqref{eq:AlgBdLay} 
for $x'\in (0,1/2)$ and $s = 1/4$. 
Note that the change in the slope (from $1/2$ to $1$) 
near the boundary is a numerical artifact -- 
as the approximation is improved, the kink moves to the left.}
\label{fig:AlgbLayer}
\end{minipage}
\hfill
\begin{minipage}[c]{0.48\textwidth}
   \includegraphics[width=\textwidth]{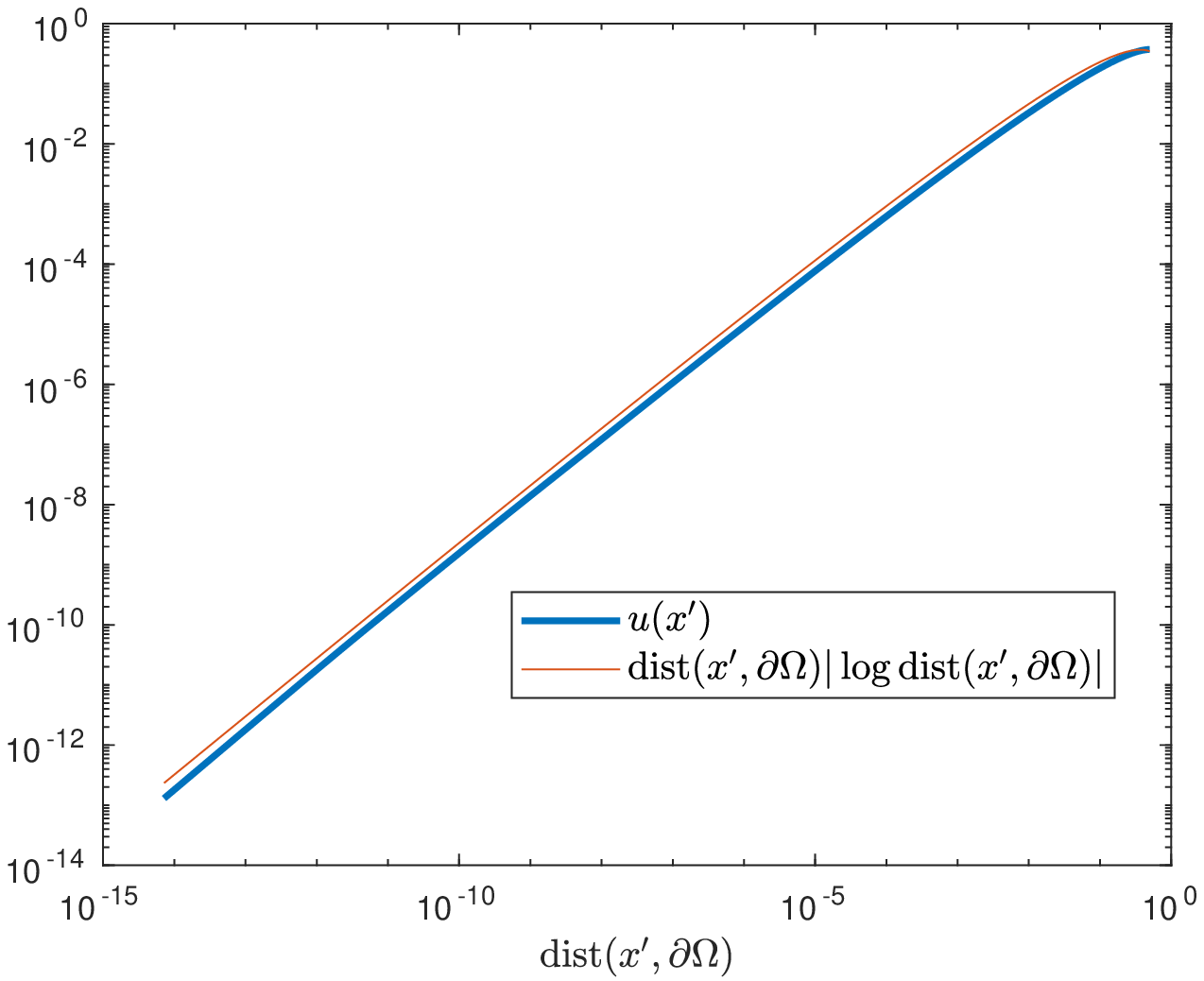}
\caption{
Boundary behavior for $s = 1/2$. 
Here the numerical solution is compared with $\dist(x',\partial\Omega)|\log \dist(x',\partial\Omega)|$.
}
\label{fig:AlgbLayer_half}
\end{minipage}
\end{figure}
Even if the domain $\Omega$ is smooth, 
$u$ exhibits in general a \emph{boundary singularity} 
with singular support $\partial \Omega$. 
For $s=1/2$ and polygonal $\Omega$, 
this boundary singularity is the trace, at $y=0$, of
an \emph{edge singularity} of the solution $\ue$ of the extended problem \eqref{alpha_harm_intro}
in $\C$ whose structure is known; see, for instance, \cite{CDN12} and the references therein.
Here, $hp$-FE approximations with geometric boundary layer meshes 
in $\Omega$ naturally appear as $y=0$ slices of $d+1$-dimensional
geometric meshes in $\C_{\Y}$ as developed in \cite{ScSc1}. 
%
\section{Conclusions and generalizations}
\label{S:ConclRmkGen}
%
In the course of this work, 
we introduced and analyzed four different types of \emph{local} FEM
discretizations for the numerical approximation of the 
spectral fractional diffusion problem \eqref{fl=f_bdddom}
%
%
in a bounded polygonal domain $\Omega \subset {\mathbb R}^2$ with
straight sides (or a bounded interval $\Omega \subset {\mathbb R}$),
subject to homogeneous Dirichlet boundary conditions.
Our local FEM schemes are based on the Caffarelli-Silvestre extension 
of \eqref{fl=f_bdddom} from $\Omega$ to $\C$.
Our main contributions are the following.

\begin{enumerate}[$\bullet$]
\item 
{\bf General operators and nonconvex domains.}
We proposed a tensor product argument for continuous, piecewise linear FEM 
in both $(0,\infty)$, and in $\Omega$ with proper mesh refinement towards $y=0$
and the corners $\bmc$ of $\Omega$.
Assuming that $A$ and $c$ are as in Proposition~\ref{prop:H2Regw},
we showed that the approximate solution to problem \eqref{fl=f_bdddom}
exhibits a 
near optimal asymptotic convergence rate $\calO(h_\Omega|\log h_\Omega|)$
subject to the optimal regularity $f \in {\mathbb H}^{1-s}(\Omega)$.
However, if $\calN_\Omega$ denotes the number of degrees of freedom 
in the discretization in $\Omega$, then
the total number of degrees of freedom grows asymptotically 
as $\calO(\calN_\Omega^{3/2})$ (ignoring logarithmic factors).

This result is analogous to the bounds obtained in \cite{NOS} 
for 
convex domains $\Omega$,
thus generalizing these results to
nonconvex, polygonal domains $\Omega\subset {\mathbb R}^2$.
The error analysis proceeded by a 
suitable form of quasi-optimality in Lemma \ref{lem:GalErr}
and the construction of a tensor product FEM interpolant in the 
truncated cylinder $\calC_{\Y}$. 
This interpolant was constructed from a nodal, continuous and piecewise linear
interpolant $\pi^{1,\ell}_\eta$ 
with respect to the extended variable $y\in (0,\Y)$ on a 
radical-geometric mesh, 
and from an $L^2(\Omega)$ projection $\Pi^\ell_\beta$ in $\Omega$
onto the space of continuous, piecewise linears on a suitable sequence 
$\{ \calT^\ell_\beta \}_{\ell\geq 0}$ of regular nested, bisection-tree,  
simplicial meshes with refinement towards the corners $\bmc$ of $\Omega$.
A novel result from \cite{GspzHeineSiebert2016} implies that $\Pi^\ell_\beta$
is also uniformly $H^1(\Omega)$-stable with respect to the refinement level $\ell$.
The present construction would likewise work with any other concurrently $L^2(\Omega)$ 
and $H^1(\Omega)$ stable family of quasi-interpolation operators,
e.g. those of \cite{SZ:90}.

\item 
{\bf Sparse tensor grids.}
While the regularity requirement $f\in {\mathbb H}^{1-s}(\Omega)$ is, essentially, 
minimal for first order convergence in $\Omega$, 
the complexity $\calO(\calN_\Omega^{3/2})$ 
due to the extra degrees of freedom in the extended variable
results in superlinear work with respect to $\calN_\Omega$.
We therefore proposed in Section~\ref{S:sGP1FEM} a 
\emph{novel, sparse tensor product FE discretization} of the truncated, extended problem. 
Using novel regularity results for the extended solution in $\calC$ in weighted spaces
and sparse tensor product constructions of the interpolation operators
$\pi^{1,\ell}_\eta$ and $\Pi^\ell_\beta$ in $\Omega$, 
we proved that this approach still delivers 
FEM solutions of \eqref{fl=f_bdddom}
with essentially first order convergence rates
(i.e., up to logarithmic factors), under 
the slightly more stringent regularity $f\in {\mathbb H}^{1-s+\nu}(\Omega)$, $\nu > 0$, 
while requiring essentially only $\calO(\calN_\Omega)$ many degrees of freedom.

\item 
{\bf $hp$-FE approximation in the extended variable.}
The solution of the extended problem being analytic with respect to the extended
variable $y>0$ allows for designing $hp$-FE approximations with
respect to the variable $y$ on geometric meshes and proving
\emph{exponential convergence rates}
even under finite regularity of $A$, $c$ and $f$ 
as specified in Proposition \ref{prop:H2Regw}.
The proof is based on a novel framework of countably normed, 
weighted Bochner spaces in $(0,\infty)$ 
to quantify the analytic regularity with respect to $y$. 
We also developed a corresponding family of $hp$-interpolation operators 
\emph{that affords exponential convergence rates} in the extended variable.

Upon tensorization with the projectors $\Pi^\ell_\beta$ 
onto spaces of continuous, piecewise linear finite elements
on simplicial, bisection-tree meshes with corner refinement in $\Omega$, 
we obtained a class of FE schemes that afford essentially optimal,
linear convergence rate in $\Omega$ under the regularity $f\in {\mathbb H}^{1-s}(\Omega)$,
also for nonconstant coefficients and nonconvex polygonal domains $\Omega$, 
thereby generalizing \cite{MPSV17}.
We remark that the convergence rate bounds essentially equal
    the results of so-called wavelet Galerkin discretizations for 
    the integral fractional Laplacian (see \cite{SchneidR98,SSBEM:11} 
    and the references therein). Wavelet Galerkin methods 
    are based on direct, ``nonlocal''
    Galerkin discretization of integro-differential operators, which
    entail numerical evaluation of 
    singular integrals and dense stiffness matrices,
    neither of which occurs in the present local FE approach.
    However, these methods can also cope with variable exponent 
    $s(x')$, which seems to be beyond reach with the present
    approach; see \cite{SRS:10,KarniadEtAl:17} and the references therein.
We also point out that the \emph{boundary compatibility of $f$}, which is 
implicit in the assumption $f\in {\mathbb H}^{1-s}(\Omega)$, 
is essential in the arguments in Section \ref{S:FEMy} as well as 
in the results of \cite{NOS,MPSV17,BP:13}.

\item 
{\bf Diagonalization.}
We developed a novel diagonalization approach which allows us to decouple
the second order elliptic system in $\C_\Y$,
resulting from \emph{any} Galerkin semidiscretization in the extended variable $y$
(either of $h$-FEM or of $hp$-FEM type) of the truncated problem,
into a finite number of decoupled, singularly perturbed, second order elliptic problems
in $\Omega$. This approach is instrumental for both the design
of $hp$-FEMs in $\Omega$ in Section~\ref{sec:hpx} as well as the
implementation of parallel and inexact solvers in Section~\ref{S:NumExp}.

\item 
{\bf $hp$-FEMs.}
Exploiting results on \emph{robust exponential convergence of $hp$-FEMs}
for second order, singularly perturbed problems
\cite{MR1664765,melenk-schwab98,melenk97,melenk02}, 
and tensorization with the exponentially convergent $hp$-FEM in $(0,\Y)$ 
resulted in {\it exponential convergence for analytic input data $A$, $c$, $f$, and $\Omega$}
for {\it incompatible} forcing $f$ 
(i.e. $f\in H^{1-s}(\Omega)$ but $f\notin\mathbb{H}^{1-s}(\Omega)$).
The boundary incompatibility of $f$ leads to the formation of a 
strong boundary singularity for $0 < s \leq 1/2$ 
and 
a  weaker one for $s > 1/2$ with $\partial \Omega$ analytic,
which is a genuine fractional diffusion effect.
Our analysis in Section \ref{S:2d-bdy-geomesh} revealed that
for incompatible data $f$ in space dimension $d>1$, 
\emph{anisotropic, geometric meshes in $\Omega$}
capable of resolving boundary layers over a wide range of length scales, 
are generally indispensable, even if $\partial \Omega$ is smooth.
Section \ref{S:NumExp} displays an example.
\end{enumerate}
%
The following generalizations of the results of the present work suggest themselves. 
\begin{enumerate}[$\bullet$]
\item {\bf Boundary conditions.}
The present analysis was limited to polygonal domains in two space dimensions
and to homogeneous Dirichlet boundary conditions. 
The extension \eqref{alpha_harm_intro} is also available for homogeneous Neumann
boundary conditions in \cite[Section 7]{MR3489634} and for combinations
of Dirichlet and Neumann boundary conditions on parts of $\partial\Omega$.
Solutions $\ue$ of these extensions also admit the representation \eqref{eq:SepVar},
so that the analytic regularity results in Section \ref{S:analytic-regularity}
extend almost verbatim.
Likewise, all regularity results in Section \ref{S:FEMy}, being based
on \cite{BacutaLiNistor2016},  extend verbatim to
homogeneous Neumann and Dirichlet-Neumann boundary conditions on polygonal domains.

\item {\bf Higher dimensions and elements of degree $q\geq 2$ in $\Omega$.}
Analogous results as in Section \ref{S:FEMy} 
hold for polyhedral domains $\Omega \subset {\mathbb R}^3$ 
with plane faces, using corresponding regularity results for the Dirichlet 
Laplacian in weighted spaces in the polyhedron $\Omega$,
combined with corresponding FE projections on anisotropically refined
FE meshes (with corner and edge-refinements in $\Omega)$, as described in \cite{Apel:99}.

Returning to polygons, if we consider piecewise polynomials of degree $q \geq 2$ 
on families of simplicial meshes which are sufficiently
refined towards the vertices $\bmc$ of $\Omega$, we expect 
algebraic convergence rates higher than for linear elements
\emph{provided} the forcing $f\in {\mathbb H}^{q-s}(\Omega)$. 
This implies, in particular, that $f$ should satisfy besides $f\in H^{q-s}_{loc}(\Omega)$
also certain higher-order boundary compatibility on $\partial\Omega$,
a consequence of the eigenfunction expansions used in our regularity analysis. 
\end{enumerate}
%
\appendix 
\section{Proof of Lemma~\ref{lemma:interpolant-on-geometric-mesh}}
\label{S:proof-interpolant-geo-mesh}
%
We will only show \eqref{eq:B1beta-estimate-10}, \eqref{eq:B1beta-estimate-20} 
as the estimates \eqref{eq:B2beta-estimate-10}, \eqref{eq:B2beta-estimate-20} 
are proved using similar arguments; see, for instance, the proof of \cite[Theorem~{8}]{apel-melenk17}. 
We distinguish between the first element $I_1$, the terminal element $I_M$, 
and the remaining ones. 
We write $h_i = |I_i|$. 
We simplify the exposition by assuming $X = {\mathbb R}$.  
It is convenient to define, for each interval $I_i$, $i=2,\ldots,M$, 
the quantity $C_i$ by 
\begin{equation}
\label{eq:Ci2}
C_i^2:= \sum_{\ell=1}^\infty (2 K_u)^{-\ell} \frac{1}{\ell!^2} 
\|u^{(\ell)}\|^2_{L^2(\omega_{\alpha + 2\ell-2\beta,\gamma}, I_i)}.  
\end{equation}
We observe that, since $u \in {\mathcal B}_{\beta,\gamma}^1(C_u,K_u)$,
\begin{equation}
\sum_{i=2}^M C_i^2 \leq 2 C_u^2,
\end{equation}
where, we recall that the space ${\mathcal B}_{\beta,\gamma}^1(C_u,K_u)$ corresponds to a class of analytic functions and is defined as in \eqref{eq:Bbeta1}. We begin the proof with an auxiliary result about linear interpolation on the reference element.

\begin{lemma}[linear interpolant]
\label{lemma:I1}
Let $X$ be a Hilbert space,   
$\widehat K = (0,1)$, and 
let $\widetilde  \pi_1$ be the linear interpolant in the points $1/2$, $1$.
Let $\alpha > -1$ and $\delta \leq 1$. 
Then, for $u \in C((0,1];X)$ and provided the terms on the right-hand side 
are finite, we have
\begin{align}
\label{eq:lemma:I1-10}
\int_{\widehat K} y^\alpha \|u - \widetilde \pi_1 u\|_X^2\, \diff y &\lesssim 
\int_{\widehat K} y^{\alpha+2\delta} \|u^\prime\|_X^2\, \diff y , \\
\label{eq:lemma:I1-20}
\int_{\widehat K} y^\alpha \|(u - \widetilde \pi_1 u)^\prime\|_X^2\, \diff y &\lesssim 
\int_{\widehat K} y^{\alpha+2\delta} \|u^{\prime\prime}\|_X^2\, \diff y,
\end{align}
where the hidden constant is independent of $u$.
\end{lemma}
\begin{proof}
For notational simplicity, we will prove the lemma only for the case $X = {\mathbb R}$. 

We begin with the proof of \eqref{eq:lemma:I1-10}. Since $(u - \widetilde \pi_1 u)(1) = 0$ we have,
for $y \in \widehat K$,
\[
(u - \widetilde \pi_1 u)(y) = \int_1^y (u - \widetilde \pi_1 u)^\prime(t)\,\diff t, 
\]
so that 
$$
\int_0^1 y^\alpha |u - \widetilde\pi_1 u|^2\, \diff y \leq 
2 \int_0^1 y^\alpha \left| \int_{y}^1 |u^\prime(t)|\,\diff t\right|^2 \diff y + 
2 \int_0^1 y^\alpha \left| \int_{y}^1 |(\widetilde \pi_1 u)^\prime(t)|\,\diff t\right|^2 \diff y.  
$$

From Hardy's inequality (e.g., \cite[Chapter~2, Theorem~{3.1}]{devore93}) we infer
\begin{align*}
\int_0^1 y^\alpha \left| \int_{y}^1 |u^\prime(t)|\,\diff t\right|^2\, \diff y 
\leq (\alpha+1)^{-2} \int_0^1 y^{\alpha+2} |u^\prime(y)|^2\, \diff y. 
\end{align*}

From
$
(\widetilde \pi_1 u)^\prime = 2\int_{1/2}^1 u^\prime(t)\,\diff t 
$
we obtain
$
| (\widetilde \pi_1 u)^\prime|^2  \leq C \int_{1/2}^1 t^{\alpha+2\delta} |u^\prime(t)|^2\,\diff t 
$
and therefore, in view of $\alpha > -1$, the estimate
$$
\int_0^1 y^\alpha | (\widetilde \pi_1 u)'|^2\, \diff y \lesssim \int_0^1 y^{\alpha+2} |u^\prime(y)|^2\, \diff y. 
$$
This concludes the proof of (\ref{eq:lemma:I1-10}) for the case $\delta = 1$. Since the integration
range is $y \in \widehat{K} = (0,1)$, we may replace $y^{\alpha+2}$ by $y^{\alpha +2 \delta}$. 

Let us now prove \eqref{eq:lemma:I1-20}. Again, it suffices to consider the limiting case $\delta = 1$
and use Hardy's inequality. We write
\begin{align*} 
(u - \widetilde \pi_1 u)^\prime(y)  & = u^\prime(y) - 2 \int_{1/2}^1 u^\prime(t)\, \diff t
 = 2\int_{1/2}^1 \left( u^\prime(y) - u^\prime(t) \right)\,\diff t \\
&= 
2\int_{1/2}^1 \int_{t}^y u^{\prime\prime}(\tau)\,\diff\tau\,\diff t . 
\end{align*}
Therefore,
\begin{align*} 
  \int_{0}^1  y^\alpha |(u -  & \widetilde\pi_1 u)^\prime(y)|^2\,\diff y  = 
  4 \int_{0}^1 y^\alpha \left| 
    \int_{1/2}^1 \int_{t}^y u^{\prime\prime}(\tau)\,\diff \tau\, \diff t 
  \right|^2\, \diff y \\
  & \leq 2 \int_{1/2}^1 
    \int_{0}^1 y^\alpha \left| \int_{t}^y |u^{\prime\prime}(\tau)|^2\,\diff \tau \right|^2 \diff y
  \diff t\\
  & \lesssim \int_{1/2}^1 \int_{0}^1 y^\alpha 
    \left[ \left| \int_{y}^1 |u^{\prime\prime}(\tau)|^2\,\diff\tau\right|^2 
   +  \left| \int_{t}^1 |u^{\prime\prime}(\tau)|^2\,\diff \tau \right|^2 \right] \diff y \diff t\\
  &\lesssim \int_{0}^1 y^{\alpha+2} |u^{\prime\prime}(y)|^2\,\diff y  + 
  \int_{1/2}^1 y^{\alpha +2} |u^{\prime\prime}(y)|^2\, \diff y, 
\end{align*}
where, in the last step we applied Hardy's inequality.

The Lemma is thus proved.
\end{proof}

With this auxiliary result at hand we can estimate $I_1$ as follows: 
scaling the estimate \eqref{eq:lemma:I1-10} gives 
\begin{equation}
\label{eq:hp-approx-on-I_1}
\|u - \pi^\bmr_{y} u\|_{L^2(\omega_{\alpha,0},I_1)} 
\leq C h_1^{\beta} \|u^\prime\|_{L^2(\omega_{\alpha+2-2\beta,0},I_1)}. 
\end{equation}
The assumption $|I_1| = \sigma^M \Y \leq 1$ implies that we may insert the weight $e^{\gamma y}$ 
on both sides of \eqref{eq:hp-approx-on-I_1}. 

We now proceed the estimation over the elements away from the origin, 
\ie on $I_i$, $i=2,\ldots,M$. 
These elements satisfy $h_i \sigma/(1-\sigma)=  \dist(I_i,0) $. 
For $I_i = (y_{i-1}, y_i)$ the pull-back 
$\widehat u_i: = u|_{I_i} \circ F_{I_i}$ satisfies 
\begin{align*}
& 
\|\widehat u_i^{(\ell+1)}\|^2_{L^2(-1,1)}   
 =(h_i/2)^{-1+2(\ell+1)} \|u^{(\ell+1)}\|^2_{L^2(I_i)}  
\\
& \leq (h_i/2)^{-1+2(\ell+1)} e^{-\gamma y_{i-1}} 
\max_{y \in I_i} y^{-\alpha-2(\ell+1)+2\beta} 
\|u^{(\ell+1)}\|^2_{L^2(\omega_{\alpha+2(\ell+1)-2\beta,\gamma},I_i)} 
\\
& \lesssim e^{-\gamma y_{i-1}} h_i^{-1 + 2 (\ell+1)} 
h_i^{-\alpha - 2 (\ell+1)+2\beta} (2 (1-\sigma))^{-2(\ell+1)}  
C_i^2 (2 K_u)^{2(\ell+1)} (\ell+1)!^2,
\end{align*}
where in the last step we have used \eqref{eq:Ci2}. The assumption on the 
operator $\widehat \Pi_r$, 
defined on the reference element, 
then yields the existence of a $b>0$ that depends 
solely on $K_u$ and $\sigma$, for which
$$
\| \widehat u -  \widehat \Pi_{\bmr_i}\widehat u\|_{L^2(-1,1)} 
\lesssim 
 C_i e^{-\gamma y_{i-1}} e^{-b \bmr_i} h_i^{-(1+\alpha)/2 + \beta}
\;.
$$
Scaling back to $I_i$ and using again $h_i \sim \operatorname*{dist}(I_i,0)$ yields 
$$
\|u - \pi^{\bmr}_{y} u\|^2_{L^2(\omega_{\alpha,\gamma},I_i)} \leq C h_i^{2\beta} C_i^2 e^{-2 b \bmr_i}. 
$$
Summation over $i$ and taking the slope of the linear degree vector sufficiently large 
(see, for instance, the proof of \cite[Theorem~{8}]{apel-melenk17} for details) 
gives 
$$
\sum_{i=2}^M \|u - \pi^{\bmr}_y u\|^2_{L^2(\omega_{\alpha,\gamma},I_i)} \lesssim {\Y}^{2\beta} e^{-2 b' M} 
$$
for suitable $b' > 0$. Combining this with 
\eqref{eq:hp-approx-on-I_1} gives the desired \eqref{eq:B1beta-estimate-10}. 

It remains to prove \eqref{eq:B1beta-estimate-20}. 
We begin with a preparatory result.

\begin{lemma}[exponential decay]
\label{lemma:inftyExp}
Let $X$ be a Hilbert space and let $\delta \in {\mathbb R}$, $\gamma > 0$, $\Y_0 > 0$. 
 Then the following holds for $u \in C^1((\Y_0,\infty);X)$ in items 
(\ref{item:lemma:inftyExp-i}), (\ref{item:lemma:inftyExp-ia}) and for 
 $u \in C^2((\Y_0,\infty);X)$ in items 
(\ref{item:lemma:inftyExp-ii}), (\ref{item:lemma:inftyExp-iia}) with implied constants depending 
solely on $\delta$, $\gamma$, and $\Y_0$: 
 \begin{enumerate}[(i)]
 \item 
 \label{item:lemma:inftyExp-i}
If $\lim_{y \rightarrow \infty} u(y) = 0$ and 
$\|u^\prime\|_{L^2(\omega_{\delta,\gamma},(\Y_0,\infty);X)} < \infty$, then 
\begin{equation}
\label{eq:lemma:inftyExp-i}
\|u(\Y)\|_X \lesssim 
 \Y^{-\delta/2} \exp(-\Y \gamma/2) \|u^\prime\|_{L^2(\omega_{\delta,\gamma},(\Y,\infty);X)} 
\qquad \forall \Y \ge \Y_0. 
\end{equation}
\item 
 \label{item:lemma:inftyExp-ia}
If 
$\sum_{j=0}^1 \|u^{(j)} \|_{L^2(\omega_{\delta,\gamma},(\Y_0,\infty);X)} < \infty,
$
then 
$\lim_{y \rightarrow \infty} u(y) = 0.$ 
 \item 
 \label{item:lemma:inftyExp-ii}
If $\lim_{y \rightarrow \infty} u^{(j)}(y)=0$ for $j=0$, $1$ 
and $\|u^{\prime\prime}\|_{L^2(\omega_{\delta,\gamma},(\Y_0,\infty);X)} < \infty$, then 
 \begin{equation}
 \label{eq:lemma:inftyExp-ii}
 \|u(\Y)\|_X \lesssim \Y^{-\delta/2} \exp(-\Y \gamma/2) 
 \|u^{\prime\prime}\|_{L^2(\omega_{\delta,\gamma},(\Y,\infty);X)} 
\qquad \forall \Y \ge \Y_0. 
 \end{equation}
\item 
 \label{item:lemma:inftyExp-iia}
If $\sum_{j=0}^2 \|u^{(j)} \|_{L^2(\omega_{\delta,\gamma},(\Y_0,\infty);X)} < \infty$, 
 then $\lim_{y \rightarrow \infty} u(y)  = \lim_{y\rightarrow \infty} u^\prime(y) = 0$. 
 \end{enumerate}
\end{lemma}
\begin{proof}
 We will only prove items (\ref{item:lemma:inftyExp-i}) and (\ref{item:lemma:inftyExp-ia}) as 
the remaining two are proved by similar arguments. 

We begin the proof with the following observation: 
There is a constant that depends only on 
$\delta$, $\Y_0$, and $\gamma$ such that
\begin{equation}
\label{eq:asymptotics}
\int_{\Y}^\infty y^{-\delta} \exp(-\gamma y)\,\diff y 
               \lesssim  \Y^{-\delta} \exp(-\gamma \Y). 
\end{equation}
For $\delta \ge 0$, this is immediate. 
For $\delta <  0$, one integrates by parts once
to discover that the leading order asymptotics (as $\Y \rightarrow \infty$) 
of the integral is $\gamma^{-1} \exp(-\gamma \Y) \Y^{-\delta}$. 

We now proceed with the proof of \eqref{eq:lemma:inftyExp-i}:
Since $\gamma > 0$,  we can write  
\begin{align*}
\|- u(\Y)\|_{X}  = \left\|\int_{\Y}^\infty u^\prime(y)\,\diff y \right\|_{X}  
\leq 
\sqrt{ \int_{\Y}^\infty y^{-\delta} \exp(-\gamma y)\,\diff y} 
\|u^\prime\|_{L^2(\omega_{\delta,\gamma},(\Y,\infty))}, 
\end{align*}
and \eqref{eq:lemma:inftyExp-i} follows from \eqref{eq:asymptotics}. The assertion 
of item (\ref{item:lemma:inftyExp-ia}) follows by a similar argument, 
starting from $u(y) =  u(\eta) + \int_{\eta}^y u^\prime(t)\,\diff t$, squaring, multiplying
by $\exp(\gamma' \eta)$ for arbitrary $0 < \gamma' < \gamma$, and integrating in $\eta$. 
\end{proof}

To prove \eqref{eq:B1beta-estimate-20} we have to estimate $u(\Y)$. 
Lemma~\ref{lemma:inftyExp} shows 
$$
\|u(\Y)\|_X \lesssim \Y^{-\alpha/2-(1-\beta)} \exp(-\Y \gamma/2) C_u. 
$$
With this estimate in hand, we can show \eqref{eq:B1beta-estimate-20}, 
recalling that $|I_M | \sim \Y$. 
%
\section{Analysis of the decoupling eigenvalue problem}
\label{S:AnEVPy}
\begin{lemma}[weighted Poincar\'e]
\label{lemma:weighted-Linfty}
Let $\Y>0$ and  $\alpha \in (-1,1)$. 
Then, for $v \in C^1((0,\Y])$ with $v(\Y) = 0$ there holds
\begin{equation}
\|v\|_{L^\infty(0,\Y)} \leq \Y^{(1-\alpha)/2} (1-\alpha)^{-1/2} \|v^\prime\|_{L^2(y^\alpha,(0,\Y))}. 
\end{equation}
\end{lemma}
\begin{proof}
From $v(\Y) = 0$ we get $v(y) = - \int_{y}^\Y v^\prime(t)\,\diff t$. 
Hence, for $y \in (0,\Y)$,
\begin{align*}
|v(y)|  & = \left| \int_{y}^\Y v^\prime(t)\,\diff t\right| 
            = \left| \int_{y}^\Y t^{-\alpha/2} t^{\alpha/2} v^\prime(t)\,\diff t\right| 
\leq 
            \left( \int_{y}^\Y t^{-\alpha}\,\diff t\right)^{1/2} \|v^\prime\|_{L^2(y^\alpha,(0,\Y))}  \\
            & \leq \Y^{(1-\alpha)/2} (1-\alpha)^{-1/2} \|v^\prime\|_{L^2(y^\alpha,(0,\Y))} , 
\end{align*}
which finishes the proof. 
\end{proof}

\begin{lemma}[eigenvalue upper bound]
\label{lemma:upper-bound-lambda}
Let $\Y>0$ and $\alpha \in (-1,1)$. 
Assume that $(v,\mu)$ satisfy
\begin{equation} 
\|v^\prime\|^2_{L^2(y^\alpha,(0,\Y))} = 1, 
\qquad \|v\|^2_{L^2(y^\alpha,(0,\Y))} = \mu,
\qquad v(\Y) = 0.  
\end{equation}
Then,
$0 < \mu \leq \Y^2 (1 - \alpha^2)^{-1}$. 
\end{lemma}
\begin{proof}
We compute, using Lemma~\ref{lemma:weighted-Linfty}
\begin{align*}
\mu 
& = \|v\|^2_{L^2(y^\alpha,(0,\Y))} = \int_{0}^\Y t^\alpha |v(t)|^2\,\diff t
\leq \|v\|^2_{L^\infty(0,\Y)} \Y^{1+\alpha} (1+\alpha)^{-1}   \\
&\leq
\Y^{1+\alpha} \Y^{1-\alpha} (1+\alpha)^{-1} (1-\alpha)^{-1} \|v^\prime\|^2_{L^2(y^\alpha,(0,\Y))}
 = \Y^2 (1-\alpha^2)^{-1}, 
\end{align*}
which finishes the proof. 
\end{proof}

We also need lower bounds for eigenvalues.

\begin{lemma}[eigenvalue lower bound]
\label{lemma:lower-bound-lambda}
Let $\alpha > -1$.
Let $\calG^M$ be an arbitrary mesh on $(0,\Y)$ 
with the property that for all elements
$I_i$, $i=2,\ldots,M$, not abutting $y=0$ there holds $|I_i| \leq C_{geo} \dist(I_i,0)$.
Let $V_h \subset H^1(y^\alpha,(0,\Y))$ be
a subspace of the space of piecewise polynomials of degree $q$ on $\calG^M$. 
Then, with $h_{min}$ denoting the smallest element size,
\begin{equation}
\|v^\prime\|_{L^2(y^\alpha,(0,\Y))} \lesssim h_{min}^{-1} q^2 \|v\|_{L^2(y^\alpha,(0,\Y))}, 
\qquad \forall v \in V_h, 
\end{equation}
where the hidden constant depends solely on $C_{geo}$ and $\alpha$.
\end{lemma}
\begin{proof}
We emphasize that the condition $h_i \leq C_{geo} \operatorname*{dist}(I_i,0)$ is satisfied
for all meshes where neighboring elements have comparable size.
We also remark that (slightly) sharper estimates (in the dependence on the polynomial degree $q$) 
are possible on geometric meshes with linear degree vector. 
We write $h_i = |I_i|$. 
We note the polynomial inverse estimate 
\begin{equation}
\label{eq:poly-inverse}
\int_{-1}^1 (1+y)^\alpha w^\prime(y)^2\, \diff x \lesssim q^4 \int_{-1}^1 (1+y)^\alpha w^2(y)\,\diff y
\quad \forall w \in {\mathbb P}_q(\widehat K).  
\end{equation}
For the first element $I_1 = (0,y_1)$ we calculate for $v \in V_h$ 
and its pull-back $\widehat v:= v|_{I_1} \circ F_{I_1}$
\begin{align} 
\nonumber 
\|v^\prime\|^2_{L^2(y^\alpha,\widehat K)} 
& = (h_1/2)^{\alpha+1-2} \int_{-1}^1 (1+y)^\alpha |\widehat v^\prime(y)|^2\,\diff y  \\
\label{eq:inverse-estimate-1}
& \lesssim h_1^{\alpha+1-2} q^4 \int_{-1}^1 (1+y)^\alpha |\widehat v(y)|^2\,\diff y 
 \sim h_1^{-2} q^4 \|v\|^2_{L^2(y^\alpha,I_1)},
\end{align}
where, in the last step, we used the inverse estimate \eqref{eq:poly-inverse}.
For the remaining elements $I_i$, we exploit that the assumption $h_i \ge C_{geo} \dist(I_i,0)$
to obtain that the weight is slowly varying over them, \ie
$$
\max_{y \in I_i} y^\alpha \leq (1+C_{geo})^{|\alpha|} \min_{y \in I_i} y^\alpha, 
\qquad i=2,\ldots,M.
$$
Hence, the polynomial inverse estimate \eqref{eq:poly-inverse} (with $\alpha= 0$ there) 
yields by scaling arguments
\begin{equation}
\label{eq:inverse-estimate-2}
 \|v^\prime\|_{L^2(y^\alpha,I_i)} \leq C h_i^{-1} q^2 \|v\|_{L^2(y^\alpha,I_i)}.
\end{equation}
Combining \eqref{eq:inverse-estimate-1}, \eqref{eq:inverse-estimate-2} yields the result.
\end{proof}

\bibliographystyle{plain}
\bibliography{biblio}
\end{document}